\documentclass[11pt]{amsart}
\usepackage[dvipdfm]{graphicx, color}
\usepackage{amssymb, amsmath,amsthm}
\usepackage{epstopdf}
\newtheorem{thm}{{\bf Theorem}}[section]
\newtheorem{lem}[thm]{{\bf Lemma}}

\newtheorem{prop}[thm]{{\bf Proposition}}

\newtheorem{rem}[thm]{Remark}
\newtheorem{ex}[thm]{Example}

\newtheorem{ques}[thm]{Question}

\newtheorem{definition}[thm]{Definition}
\numberwithin{equation}{section}

\begin{document} 

\title[Dynamics of the monodromies of the fibrations on the magic $3$-manifold]{
Dynamics of the monodromies of the fibrations on the magic $3$-manifold}

\author[E. Kin]{%
    Eiko Kin
}
\address{%
       Department of Mathematics, Graduate School of Science, Osaka University Toyonaka, Osaka 560-0043, JAPAN
}
\email{%
        kin@math.sci.osaka-u.ac.jp
}

\subjclass[2000]{%
	Primary 57M27, 37E30, Secondary 37B40
}

\keywords{%
pseudo-Anosov, dilatation, topological entropy, train track representative, 
magic manifold, branched surface
}

\date{%
	October 5, 2014
	}

\thanks{%
	The  author is partially supported by 
	Grant-in-Aid for Young Scientists (B) (No. 20740031), MEXT, Japan.  
	} 

\begin{abstract}

We study the magic manifold $N$ which is a hyperbolic and fibered $3$-manifold. 
We give an explicit construction of a fiber $F_a$ and its monodromy $:F_a \rightarrow F_a$ of the fibration 
associated to  each fibered class $a$ of $N$. 
Let $\delta_g$ (resp. $\delta_g^+$) 
be the minimal dilatation of pseudo-Anosovs (resp. pseudo-Anosovs with orientable invariant foliations) 
defined on an orientable closed surface of genus $g$. 
As a consequence of our result, we obtain the first explicit construction of the following pseudo-Anosovs; 
a minimizer of $\delta_7^+$ and conjectural minimizers of $\delta_g$ for large $g$. 
%

\end{abstract} 
\maketitle

\section{Introduction}
\label{section_Introduction}

In this paper, we explore the dynamics of  monodromies of  fibrations on a hyperbolic, fibered $3$-manifold, 
called the {\it magic $3$-manifold} $N$, 
which is the exterior of the $3$ chain link $\mathcal{C}_3$, see Figure~\ref{fig_3braid}(1). 
We first set some notations, then describe the motivation of our study. 
Let $\varSigma$ be  an orientable surface (possibly with punctures). 
A homeomorphism $\Phi: \varSigma \rightarrow \varSigma$ is {\it pseudo-Anosov} 
if there exist a pair of transverse measured foliations 
$(\mathcal{F}^u, \mu^u)$ and 
$(\mathcal{F}^s, \mu^s)$ 
and  a constant $\lambda = \lambda(\Phi)>1$ such that 
$$\Phi(\mathcal{F}^u, \mu^u) = (\mathcal{F}^u, \lambda \mu^u) \  \hspace{2mm}\mbox{and}\ \hspace{2mm}
\Phi(\mathcal{F}^s, \mu^s) = (\mathcal{F}^s, \lambda^{-1} \mu^s).$$
Then 
$\mathcal{F}^u$ and $\mathcal{F}^s$ are called the {\it unstable} and {\it stable foliations} (or {\it invariant foliations}),  and 
$\lambda$ is called the {\it dilatation} of $\Phi$. 
Let  $\mathrm{Mod}(\varSigma)$  be the mapping class group of   $\varSigma$, that is 
$\mathrm{Mod}(\varSigma)$ is the group of isotopy classes 
of orientation preserving homeomorphisms on $\varSigma$ fixing punctures setwise. 
A mapping class  $\phi \in \mathrm{Mod}(\varSigma)$ is called {\it pseudo-Anosov} 
if $\phi$ contains 
a pseudo-Anosov homeomorphism  $\Phi: \varSigma \rightarrow \varSigma$ as a representative. 
The  topological entropy  $\mathrm{ent}(\Phi)$ of $\Phi$ is equal to  $ \log \lambda(\Phi)$, 
and  $\mathrm{ent}(\Phi)$  attains the minimal entropy among all  homeomorphisms  
which are isotopic to  $\Phi$, 
see  \cite[Expos\'{e} 10]{FLP}. 
In the case $\phi= [\Phi]$, 
we denote by $\lambda(\phi)$ and 
$\mathrm{ent}(\phi)$, the dilatation $\lambda(\Phi)$ and 
topological entropy $\mathrm{ent(\Phi)}= \log \lambda(\Phi)$. 

We take an element $\phi \in \mathrm{Mod}(\varSigma)$. 
Let ${\Bbb T}_{\phi}$ be its mapping torus, 
i.e, 
if $\Phi$ is a representative of $\phi$, then 
$${\Bbb T}_{\phi}= \varSigma \times {\Bbb R}/ \sim $$ 
where $\sim$ identifies $(x,t+1)$ with $ (\Phi(x),t)$ 
for $x \in \varSigma$ and $t \in {\Bbb R}$. 
Such a $\Phi$ is called the {\it monodromy} of ${\Bbb T}_{\phi}$. 
The vector field 
$\frac{\partial}{ \partial t}$ on $\varSigma \times {\Bbb R}$ induces a  flow $\Phi^t$ on ${\Bbb T}_{\phi}$, 
which is called the {\it suspension flow}. 
The hyperbolization theorem by Thurston \cite{Thurston3} tells us that 
a $3$-manifold $M$ which is homeomorphic to ${\Bbb T}_{\phi}$ admits a hyperbolic structure 
if and only if $\phi$ is pseudo-Anosov. 
The magic  manifold $N$ is in fact a hyperbolic, fibered $3$-manifold, 
since $N$ is homeomorphic to a $4$-puncture sphere bundle over the circle with the pseudo-Anosov monodromy 
as in  Figure~\ref{fig_3braid}(2) (see Lemma~\ref{lem_summary}(1)).  

\begin{center}
\begin{figure}
\includegraphics[width=3in]{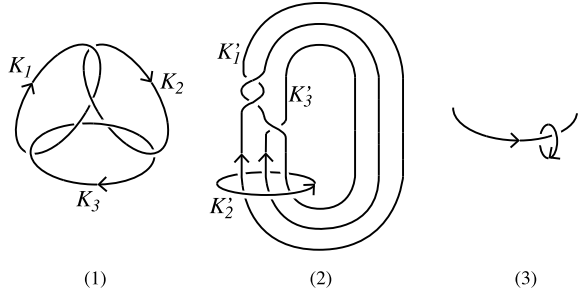}
\caption{(1) $3$ chain link $\mathcal{C}_3$. 
(2) Braided link $\mathrm{br}(\sigma_1^2 \sigma_2^{-1})$. 
Our convention of the orientation is that 
the meridian of a component of the link is chosen as in (3).} 
\label{fig_3braid}
\end{figure}
\end{center}

In a paper \cite{Thurston1} 
Thurston introduced a norm $\|\cdot\|$ on $H_2(M, \partial M; {\Bbb R})$ 
for hyperbolic $3$-manifolds  
and proved that the unit ball $U_M$ with respect to the Thurston norm $\|\cdot\|$ is a compact, convex polyhedron.  
When $M$ is homeomorphic to a hyperbolic, fibered $3$-manifold  ${\Bbb T}_{\phi}$, 
he gave a connection between the Thurston norm $\|\cdot\|$ and fibrations on $M$. 
In particular if such a $3$-manifold $M$ has the second Betti number $b_2(M)$ which is greater than $1$, 
he proved that 
there exists a top dimensional face $\Omega$ on $\partial U_M$, called a {\it fibered face} such that 
every integral class $a$ of $H_2(M, \partial M; {\Bbb Z})$ which is in the open cone over $\Omega$  
corresponds to a fiber $F_a$ of the fibration associated to $a$ with a monodromy $\Phi_a: F_a \rightarrow F_a$, 
that is ${\Bbb T}_{\phi_a}$ is homeomorphic to $M$, 
where $\phi_a= [\Phi_a]$. 
Such an integral class $a$ is called a {\it fibered class}.

Since $\phi_a= [\Phi_a]$  is pseudo-Anosov for each fibered class $a$,  
$M$ provides infinitely many pseudo-Anosovs on fibers with distinct topological types. 
A theorem by Fried \cite{Fried,Fried1} asserts that 
the monodromy $\Phi_a: F_a \rightarrow F_a$ and  the (un)stable foliation $\mathcal{F}_a$ of $\Phi_a$ 
can be described by using the suspension flow $\Phi^t$ and the suspension of the (un)stable foliation of $\Phi$. 
A question we would like to pose is that 
how about practical constructions of $\Phi_a:F_a \rightarrow F_a$ and $\mathcal{F}_a$ for each fibered class $a$. 
The theorem by Fried does not give us 
concrete descriptions of them.

E.~Hironaka gave concrete descriptions of the monodorimies of fibrations associated to sequences of fibered classes 
on some class of hyperbolic fibered $3$-manifolds, 
see \cite{Hironaka1,Hironaka2}. 
However no one constructed explicitly the monodromy of the fibration associated to each fibered class on 
a single hyperbolic, fibered $3$-manifold $M$ with $b_2(M) > 1$.  
In this paper we describe them concretely for the magic manifold $N$. 
The motivation of our study comes from  minimal dilatations on pseudo-Anosovs and their asymptotic behaviors. 
We  fix a surface $\varSigma$, and consider the set of dilatations of pseudo-Anosovs  on $\varSigma$, 
$$ \mathrm{dil}(\varSigma)= \{\lambda(\phi)\ |\ \phi \in \mathrm{Mod}(\varSigma)\  \mbox{is pseudo-Anosov}\}. 
$$
Arnourx-Yoccoz and Ivanov observed that 
for any constant $c>1$, there exist finite elements $\lambda \in \mathrm{dil}(\varSigma)$ so that $\lambda < c$, 
see \cite{Ivanov}. 
In particular, there exists a minimum $\delta(\varSigma)$ of  $\mathrm{dil}(\varSigma)$. 
Let $\varSigma_g$ be a closed surface of genus $g$, and 
$\varSigma_{g,n}$ a closed surface of genus $g$ removing $n \ge 1$ punctures. 
We let $\delta_g= \delta(\varSigma_g)$ and $\delta_{g,n}= \delta(\varSigma_{g,n})$. 
We denote by $D_n$, an $n$-punctured disk. 
A mapping class $\phi \in \mathrm{Mod}(D_n)$ defines a mapping class $\phi' \in \mathrm{Mod}(\varSigma_{0,n+1})$ fixing a puncture 
and vice versa. 
Moreover $\phi$ is pseudo-Anosov if and only if $\phi'$ is pseudo-Anosov. 
In this case the equality $\lambda(\phi)= \lambda(\phi')$ holds.  
Thus $\delta(D_n)$ is equal to the minimal dilatations among pseudo-Anosov elements $\phi' \in \mathrm{Mod}(\varSigma_{0,n+1})$ 
fixing a puncture. 
In particular we have $\delta(D_n) \ge \delta(\varSigma_{0,n+1})$. 

The minimal dilatation problem is to 
determine an explicit value of $\delta(\varSigma)$, and 
to identify a pseudo-Anosov element  $ \phi \in \mathrm{Mod}(\varSigma)$ which achieves  $\delta(\varSigma)$ 
(i.e, minimizer of $\delta(\varSigma)$).  
Naive but natural question is this: 
What does the pseudo-Anosov homeomorphism $\Phi$ on $\varSigma$ which achieves $\delta(\varSigma)$ look like? 
In other words, what does a train track representative of $\phi = [\Phi]$ 
(which enables us to describe the dynamics of $\Phi:\varSigma \rightarrow \varSigma$) 
look like?

\begin{small}
\begin{table}[hbtp]
\caption{Asymptotic behaviors of minimal dilatations and smallest known upper bounds, 
where 
$\delta(D_4) \approx 2.2966$ is the largest root of $t^4 -2t^3 -2t+1$ \cite{KLS}, and 
$\delta(D_5) \approx 1.7220$ is the largest root of $t^4-t^3-t^2-t+1$ \cite{HS}.} 
\label{table_section}
\begin{center}
$\left|\begin{array}{c r|c c r}
\hline
\mbox{asymptotic behaviors} & & & \mbox{smallest known upper bounds}& \\
\hline
 \log \delta_g \asymp 1/g  &\mbox{\cite{Penner}}  &  \mbox{(U1)}
 &  \displaystyle \limsup_{g \to \infty} \, g \log  \delta_g \le  \log(\tfrac{3+ \sqrt{5}}{2}) &\mbox{\cite{Hironaka,AD,KT1}} \\
\hline
 \log \delta_g^+ \asymp 1/g & \mbox{\cite{Minakawa,HK}}  & \mbox{(U2)}& 
\displaystyle \limsup_{\substack{g \not\equiv 0 \pmod{6} \\ g \to \infty}} g \log  \delta_g^+ \le \log(\tfrac{3+ \sqrt{5}}{2}) &\mbox{\cite{Hironaka,KKT2}}\\
& & \mbox{(U3)}& \displaystyle \limsup_{\substack{g \equiv 6 \pmod{12} \\ g \to \infty}}  g  \log \delta^+_g  \le 2 \log  \delta(D_5) &\mbox{\cite{KKT2}}\\
\hline
\log \delta_{0,n} \asymp 1/n  &\mbox{\cite{HK}}&\mbox{(U4)}&    \displaystyle  \limsup_{n \to \infty}n \log \delta_{0,n} \le 2 \log (2+ \sqrt{3}) &\mbox{\cite{HK,KT}}\\
\hline
\log \delta_{1,n} \asymp 1/n  & \mbox{\cite{Tsai}}& \mbox{(U5)}&   
\displaystyle  \limsup_{n \to \infty} n \log \delta_{1,n} \le 2 \log \delta(D_4) & \mbox{\cite{KKT2}}\\
 \hline 
 \mbox{Given}\  g \ge 2,\ \log \delta_{g,n} \asymp \frac{\log n}{n}  &\mbox{\cite{Tsai}} & \mbox{(U6)}&  
 \displaystyle\limsup_{n \to \infty} \tfrac{n \log \delta_{g,n}}{\log n} \le 2\ 
\mbox{if}\ g \ \mbox{enjoys\ }(*)&\mbox{\cite{KT2}}\\
\hline
\end{array}\right|$
\end{center} 
\end{table}  
\end{small}

Some of the minimal dilatations are already determined. 
Also there are partial results.  
For example, $\delta_2$ is computed in \cite{CH}, 
but an explicit value of $\delta_g$ is not known for $g \ge 3$. 
If  we denote by $\delta_g^+$,  the minimal dilatation of pseudo-Anosovs defined on $\varSigma_g$ 
with orientable invariant foliations, 
then  explicit values of $\delta_g^+$  for $2 \le g \le 8$ except $g=6$ are known, 
see \cite{Zhirov,LT} and \cite{Hironaka,AD,KT1}. 
The minimal dilatation on an $n$-punctured disk, $\delta(D_n)$ is determined for $ 3 \le n \le 8$, 
see \cite{KLS,HS, LT1}.

The asymptotic behaviors of the minimal dilatations are shown in the left column of Table~\ref{table_section}. 
Here $A_g \asymp B_g$ means that 
there exists a constant $c>0$ which does not depend on $g$ so that 
$\tfrac{A_g}{c} < B_g < c A_g$. 
As we can see from table, 
we have $\log \delta_{0,n} \asymp 1/n $ and $\log \delta_{1,n} \asymp 1/n  $, but 
a result by Tsai~\cite{Tsai} says that 
the situation in the case $g \ge 2$ is quite different from the case $g =0$ or $1$.  
In the right column of Table~\ref{table_section}, 
the smallest known upper bounds of the minimal dilatation. 
We give a precise condition $(*)$ in  (U6) in the following.

\begin{thm}[\cite{KT2}] 
\label{thm_boundary}
Suppose that $g \ge 2$  satisfies 
\begin{quote} 
$(*)$ \hspace{2mm}
$\gcd(2g+1, s)=1$ or $\gcd(2g+1, s+1)=1$ for each $0 \le  s \le g$.
\end{quote}
Then 
\begin{equation}
\label{equation_Tsai}
\limsup_{n \to \infty} \frac{n \log \delta_{g,n}}{\log n} \le 2.
\end{equation}
In particular, if $2g+1$ is prime, then 
$g$ enjoys $(*)$, and hence 
(\ref{equation_Tsai}) holds. 
\end{thm}

\noindent
The upper bounds (U1)--(U6) are proved by examples of sequences of pseudo-Anosovs. 
The magic manifold $N$ is involved in these upper bounds. 
There exists a sequence of fibered classes of $N$ corresponding to each of (U1), $\cdots$, (U6). 
The projection onto a fibered face 
of the sequence corresponding to (U6) converges to 
a single point which lies on the boundary of a fibered face of $N$. 
On the other hand, the projection of other sequence  converges to some point in the interior of the fibered face of $N$. 
An interesting feature is the following. 
The mapping torus ${\Bbb T}_{\phi}$ of each example $\phi$ which appears in the sequences  is either $N$ or 
the fibration of ${\Bbb T}_{\phi}$ comes from a fibration of $N$ by Dehn filling cusps along the boundary slopes of a fiber. 
This is also true for known minimizers of the minimal dilatations $\delta_2$, $\delta(D_n)$ for $3 \le n \le 8$ and 
$\delta_g^+$ for $2 \le g \le 8$ except $g=6$. 
These results say that the topological types of fibers of fibrations on $N$ are surprisingly full of variety. 
However, no explicit constructions of sequences of pseudo-Anosovs needed for  
the proofs of (U1)--(U6) except (U4)  
were given so far. 
Also an explicit example of a minimizer of $\delta_7^+$ was not given. 
In this paper we prove the following which allows us to 
construct pseudo-Anosovs in question explicitly.

\begin{thm}
\label{thm_algorithm}
We have algorithms to construct the followings. 
For each fibered class  $a$ of $N$, 
\begin{enumerate}
\item[(1)] 
 the monodromy $\Phi_a: F_a \rightarrow F_a$ of the fibration on $N$ associated to  $a$, and 

\item[(2)] 
in the case $a$ is primitive, 
a train track representative $\mathfrak{p}_a: \tau_a \rightarrow \tau_a$ of $\phi_a= [\Phi_a]$ 
whose incidence matrix  is Perron-Frobenius.
\end{enumerate}
\end{thm}

\noindent
In \cite{Oertel} Oertel constructs branched surfaces which carry fibers of fibrations on hyperbolic, fibered $3$-manifolds. 
In the proof of Theroem~\ref{thm_algorithm}(1), 
we construct branched surfaces $\mathcal{B}_+$ and $\mathcal{B}_-$ following \cite{Oertel} 
which carry fibers of fibrations associated to fibered classes on $N$.

It is well-known that 
a train track representative $\mathfrak{p}_a: \tau_a \rightarrow \tau_a$ as in Theorem~\ref{thm_algorithm}(2) 
 can recover a pseudo-Anosov homeomorphism 
which represents $\phi_a$, and 
it serves the monodromy $\Phi_a:F_a \rightarrow F_a$ of the fibration associated to $a$. 
However we do not need the claim (2) for the proof of (1). 
We can construct the both fiber $F_a$ and 
monodromy $\Phi_a: F_a \rightarrow F_a$ in an explicit and combinatorial way.

Let $N(r)$ be the manifold obtained from $N$ by Dehn filling one cusp along the slope $r \in {\Bbb Q}$. 
See Figure~\ref{fig_3braid}(3) for our convention of the orientation. 
As a consequence of Theorem~\ref{thm_algorithm}, 
we can give constructive descriptions of monodromies of fibrations associated to any fibered class on the hyperbolic, fibered manifolds $N(r)$ 
for infinitely many $r \in {\Bbb Q}$. 
For example, we can do them for 
Whitehead sister link exterior $N(\tfrac{3}{-2})$, 
the the simplest $3$-braided link exterior $N(\tfrac{1}{-2})$ and 
the Whitehead link exterior $N(1)$.  
In particular, we can construct the following pseudo-Anosov homeomorphisms explicitly:

\begin{itemize}

\item 
A minimizer of $\delta_7^+$, see Example~\ref{ex_ori79}.

\item 
Conjectural minimizers of $\delta_g^+$ for $g \equiv 2,4 \pmod 6$, 
see \cite[Question~6.1]{LT}. 
See also Remark~\ref{rem_hironaka-san} and Example~\ref{ex_LT}.

\item 
Conjectural minimizers of $\delta_g^+$ for large $g$ such that  $g \not\equiv 0 \pmod 6$, 
see \cite[Conjecture~1.12(2)]{KKT2}. 
See also Examples ~\ref{ex_LT}, \ref{ex_ori79}, \ref{ex_ori15}. 

\item 
Conjectural minimizer of $\delta_g$ for large $g$, 
see \cite[Conjecture~1.12(1)]{KKT2}. 
See also Example~\ref{ex_ori79}. 

\item 
Conjectural minimizers of $\delta_{1,n}$ for large $n$, 
see \cite[Conjecture~1.13]{KKT2}. 
See also Example~\ref{ex_whitehead}.

\item 
Sequences of pseudo-Anosov homeomorphisms to give the smallest known upper bounds (U1)--(U6) 
in Table~\ref{table_section}. 
\end{itemize}

\begin{rem}
\label{rem_hironaka-san}
Hironaka gave the first explicit construction of the orientable train track representative on $\varSigma_g$ 
for $g \equiv 2,4 \pmod 6$ 
whose dilatation equals the conjectural minimum $\delta_g^+$ for such a $g$, 
see \cite{Hironaka3}. 
Hironaka also constructed explicitly 
the infinite subsequence of pseudo-Anosov homeomorphisms defined on $\varSigma_g$ with some condition on $g$ 
to prove the upper bound (U1) and (U2), 
see \cite{Hironaka1}. 
\end{rem}

\begin{ques}
Find the word which represents the mapping class $\phi_a= [\Phi_a]$ for each fibered class $a$ of $N$ 
by using the standard generating set on $\mathrm{Mod}(\varSigma_{g,n})$. 
Its word length could be long, but would be represented by a simple word, 
see Remark~\ref{rem_delta0}.  
\end{ques}

\begin{ques}
Develop the methods given in Section~\ref{section_construction} and 
prove the same claim in Theorem~\ref{thm_algorithm} for some class of 
hyperbolic, fibered $3$-manifolds. 
\end{ques}

The paper is organized as follows. 
In Section~\ref{section_preliminarlies}, we review basic facts on train tracks, Thurston norm and clique polynomials.  
We also review  some properties of the magic manifold. 
In Section~\ref{section_construction}, we prove Theorem~\ref{thm_algorithm}. 
In its proof, we construct the directed graph $\Gamma_a$ with a metric on the set of edges, 
which is induced from the train track representative 
 $\mathfrak{p}_a: \tau_a \rightarrow \tau_a$ of $\phi_a$. 
 Such a directed graph $\Gamma_a$ captures the dynamics of both 
 $\Phi_a: F_a \rightarrow F_a$ and $\mathfrak{p}_a: \tau_a \rightarrow \tau_a$. 
 Then we construct the curve complex $G_a$ induced from $\Gamma_a$, 
 which is an undirected,  weighted graph on the set of vertices. 
 Such curve complexes  are recently studied by McMullen \cite{McMullen2}. 
 In our setting,  $G_a$ gives us some insight into 
 what the train track representative $\mathfrak{p}_a: \tau_a \rightarrow \tau_a$  looks like. 
 In Section~\ref{section_catalogue}, 
 we exhibit some subsequences of pseudo-Anosovs 
 which can be used in the proof of the upper bounds (U1)--(U6). 
 We find that the types of curve complexes in each subsequence are fixed. 
 These curve complexes give us some hints to know what the pseudo-Anosovs with the smallest dilations look like. 
\medskip

\noindent
{\bf Acknowledgment.} 
I would like to thank  Hideki Miyachi, Mitsuhiko Takasawa and Hiroyuki Minakawa. 
H.~Miyachi and M.~ Takasawa gave me valuable comments on this paper. 
H.~Minakawa gave a series of lectures on his work at Osaka University in 2004. 
I learned many things on pseudo-Anosovs and pseudo-Anosov flows during his course. 
Theorem~\ref{thm_algorithm}(1) is inspired by his construction of pseudo-Anosovs  \cite{Minakawa}.

\section{Preliminaries} 
\label{section_preliminarlies}

\subsection{Train track} 


Definitions and basic results on train tracks are contained in \cite{PaPe}. 
See also \cite{Los}. 
In this section, we recall them for  convenience. 

Throughout the paper, surfaces are  orientable. 
Let $F$ be a surface with possibly punctures or boundary.  
Let $\tau$ be a branched $1$-submanifold on $F$. 
We say that $\tau$ is a {\it train track} if 
\begin{enumerate}
\item[(1)]
$\tau$ is a smooth graph such that the edges are tangent at the vertices, 
i.e, $\tau$ looks as in Figure~\ref{fig_local}(1) near each vertex of $\tau$,

\item[(2)]
each component of $F \setminus \tau$ is a disk with more than $3$ cusps on its boundary 
or 
an annulus with more than $1$ cusp on one boundary component and with no cusps on the other boundary component 
(i.e, the other boundary component is the one of $\partial F$ or a puncture of $F$.) 
\end{enumerate}

\noindent
See Figure~\ref{fig_seedABCD}(1) for an example of a train track on $\varSigma_{0,4}$. 

Two edges of $\tau$ which are tangent at some vertex make a {\it cusp}, see Figure~\ref{fig_folding}(left).  
Associated to the train track $\tau$, 
we can define a fibered neighborhood $\mathcal{N}(\tau) \subset F$ 
whose fibers are segments given by a retraction $\mathcal{N}(\tau) \searrow \tau$. 
The fibers in this case are called {\it ties}, 
see Figure~\ref{fig_local}(2). 

\begin{center}
\begin{figure}
\includegraphics[width=2.4in]{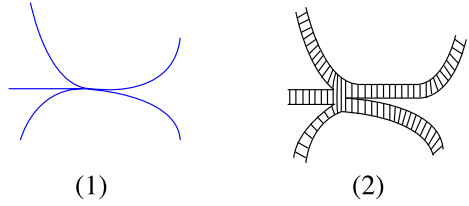}
\caption{(1) Train track near vertex. 
(2) Fibered neighborhood.}
\label{fig_local}
\end{figure}
\end{center}

\begin{center}
\begin{figure}
\includegraphics[width=2.3in]{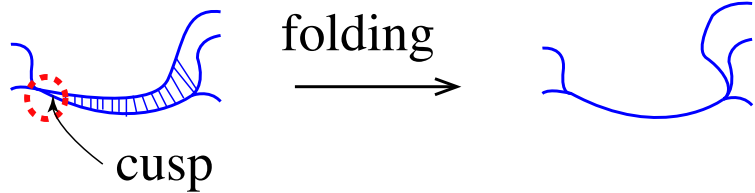}
\caption{Folding map near a cusp.}
\label{fig_folding}
\end{figure}
\end{center}

Let $\mathcal{F}$ be a measured foliation on $F$. 
We say that $\mathcal{F}$ {\it is carried by}  $\tau$ 
if $\mathcal{F}$ can be represented by a  partial measured foliation  
whose support is $\mathcal{N}(\tau)$ and which is transverse to the ties. 

Let $\sigma$ be a train track on $F$. 
We say that $\sigma$ {\it is carried by }$\tau$ 
if $\sigma$ is isotopic to a train track $\sigma'$ 
which is contained in $\mathcal{N}(\tau)$ and 
which is transverse to the ties 
(said differently, every smooth edge path on $\sigma'$ is transverse to the ties). 
Let $f: F \rightarrow F$ be a homeomorphism. 
A train track $\tau$ is {\it invariant under} $[f]$ if 
 $f(\tau)$ is carried by $\tau$, 
 that is, $f(\tau)$ is isotopic to some train track 
 $\sigma'$ which satisfies above. 
 In this case, 
folding edges of $\tau'$ near cusps repeatedly (see Figure~\ref{fig_folding} for a folding map), 
in other words, collapsing  $\tau'$ onto $\tau$ smoothly 
yields a map $\mathfrak{p}: \tau \rightarrow \tau$ such that 
$\mathfrak{p}$ maps vertices to vertices, and $\mathfrak{p}$ is  locally injective at any points which do not map into vertices. 
Such a $\mathfrak{p}: \tau \rightarrow \tau$ is called a {\it train track representative} of $[f]$. 
An edge $e$ of $\tau$ is said to be {\it infinitesimal} (for $\mathfrak{p}$) if 
$e$ is eventually periodic under $\mathfrak{p}$, that is 
$\mathfrak{p}^{m+n}(e)= \mathfrak{p}^n(e)$ for some positive integers $m$ and $n$. 
Other edges of $\tau$ is said to be {\it real}. 
Let $k$ be the number of the real edges of $\tau$. 
Then we have a $k \times k$ non-negative integer matrix $M_{\mathfrak{p}} = (m_{ij})$, called the {\it incidence matrix} for 
$\mathfrak{p}: \tau \rightarrow \tau$ 
{\it (with respect to real edges)}, 
where
$m_{ij}$ is the number of times so that the image $\mathfrak{p}(e_j)$ of the $j$th real edge 
passes through the $i$th real edge $e_i$ 
in either direction. 
Also,  $\mathfrak{p}: \tau \rightarrow \tau$ determines 
a finite, directed graph $\Gamma_{\mathfrak{p}}$ 
by taking a vertex for each real edge of $\tau$, 
and then adding $m_{ij}$ directed edges 
from  the $j$th real edge $e_j$ to the $i$th real edge $e_i$. 
In other words, we have $m_{ij}$ directed edges from $e_j$ to $e_i$ if 
$\mathfrak{p}(e_j)$  passes through $e_i$ in either direction $m_{ij}$ times . 
We say that $\Gamma_{\mathfrak{p}}$ is the {\it induced directed graph} of $\mathfrak{p}: \tau \rightarrow \tau$.

A non-negative integer matrix $M$ is said to be {\it Perron-Frobenius} 
if there exists an integer $\ell \ge 1$ such $M^{\ell}$ is  positive, that is each entry of $M^{\ell}$ is positive. 
In this case, the spectral radius of $M$ is given by the largest eigenvalue of $M$ 
called the {\it Perron-Frobenius eigenvalue}, see \cite{Gantmacher}. 
 
The following theorem is well known. 

\begin{thm}[See Theorem~4.1 in \cite{PaPe} and its proof.]  
\label{thm_PaPe}
If $\Phi: F \rightarrow F$ is a pseudo-Anosov homeomorphism, then 
 there exists a train track $\tau$ on $F$ which carries the unstable foliation $\mathcal{F}^u$ of $\Phi$ and 
 a train track representative $\mathfrak{p}: \tau \rightarrow \tau$ 
of $\phi= [\Phi]$. 
Such a representative $\mathfrak{p}: \tau \rightarrow \tau$ satisfies that 
the incidence matrix  $M_{\mathfrak{p}}$ is Perron-Frobenius 
and its Perron-Frobenius eigenvalue  is exactly equal to $\lambda(\Phi)$. 
\end{thm}


Conversely, if $f: F \rightarrow F$ is a homeomorphism and 
if $\mathfrak{p}: \tau \rightarrow \tau$ is a train track representative of $\phi=[f]$ 
such that its incidence matrix $M_{\mathfrak{p}}$  is Perron-Frobenius, then $\phi$ is pseudo-Anosov 
whose dilatation $\lambda(\phi)$ equals the Perron-Frobenius eigenvalue of $M_{\mathfrak{p}}$, 
see Bestvina-Handel \cite[Section~3.4]{BH}

\subsection{Thurston norm, fibered face, entropy function}
\label{subsection_ThurstonNorm}

We review the basic results on Thurston norm and the  relation between Thurston norm and hyperbolic, fibered $3$-manifolds 
developed by Thurston, Fried, Matsumoto and McMullen. 
Let $M$ be an oriented hyperbolic $3$-manifold possibly $\partial M \ne \emptyset$. 
We recall Thurston norm $\|\cdot \|: H_2(M, \partial M; {\Bbb R}) \rightarrow {\Bbb R}$. 
For more detail, see \cite{Thurston1}.  
Let $F= F_1 \cup F_2 \cup \cdots \cup F_k$ 
be a finite union of oriented, connected surfaces. 
We define $\chi_-{(F)}$ to be  
$$\chi_-(F)= \sum_{i=1}^k \max \{0, -\chi (F_i)\}.$$ 
Thurston norm $\|\cdot\|$ is defined for an integral class $a \in H_2(M, \partial M; {\Bbb Z})$ by 
$$\|a\|= \min_{F} \{\chi_-(F)\},$$ 
where the minimum is taken over all oriented surfaces $F$ embedded in $M$ satisfying $a= [F]$. 
The surface $F$ which realizes the minimum is called a {\it minimal representative} of $a$. 
Then $\|\cdot\|$ admits a unique continuous extension 
$\|\cdot\|: H_2(M, \partial M; {\Bbb R}) \rightarrow {\Bbb R}$ 
which is linear on rays through the origin. 
It is known that 
the unit ball $U_M \subset H_2(M, \partial M; {\Bbb R}) $ with respect to Thurston norm $\|\cdot\|$ 
is a compact, convex  polyhedron \cite{Thurston1}.  

Let $\Omega$ be any top dimensional face on the boundary $\partial U_M$ of Thurston norm ball. 
We denote by $C_{\Omega}$, the cone over $\Omega$ with the origin, and 
we denote by $int(C_{\Omega})$, the interiors of $C_{\Omega}$. 
Thurston proved in \cite{Thurston1}  that 
if $M$ is a surface bundle over the circle and if $F$ be any fiber of the fibration on $M$, 
then there exists a top dimensional face $\Omega$ on  $\partial U_M$ so that 
$[F]$ is an integral class of  $int(C_\Omega)$. 
Moreover for any integral class $a \in int(C_{\Omega})$, 
its minimal representative $F_a$  becomes a fiber of a fibration on $M$, and 
$F_a$ is unique up to isotopy along flow lines. 
Such a face $\Omega$ is called an {\it fibered face} and an integral class $a \in int(C_{\Omega})$ 
is called a {\it fibered class}. 
Thus, if the second Betti number $b_2(M)$ is greater than $1$, then the single $3$-manifold $M$ provides 
infinitely many pseudo-Anosovs on surfaces with different topological types.

Let $\Omega$ be a fibered face of $M$. 
If $a \in int(C_{\Omega})$ is primitive and integral, then 
the minimal representative $F_a$ is a connected fiber of the fibration associated $a$. 
The mapping class $\phi_a= [\Phi_a]$ of 
the monodromy $\Phi_a: F_a \rightarrow F_a$ of its fibration is pseudo-Anosov (since $M$ is hyperbolic). 
We define the dilatation $\lambda(a)$ and entropy $\mathrm{ent}(a)$ 
to be the dilatation and entropy of the pseudo-Anosov $\phi_a$. 
The entropy function defined on primitive fibered classes is naturally extended to rational classes. 
In fact for a rational number $r$ and a primitive fibered class $a$, 
the entropy of $\mathrm{ra}$ is defined to be $\mathrm{ent}(ra)= \tfrac{1}{|r|}\mathrm{ent}(a)$.

\begin{thm}[\cite{Fried1,Matsumoto,McMullen}]
\label{thm_Fried1} 
The function given by $a \mapsto \mathrm{ent}(a)$ for each rational class $a \in int(C_{\Omega})$ 
extends to a real analytic convex function on $int(C_{\Omega})$. 
The restriction  $\mathrm{ent}|_{int(\Omega)}: int(\Omega) \rightarrow {\Bbb R}$ is a strictly convex function 
which goes to $\infty$ toward the boundary of $\Omega$. 
\end{thm}

\noindent
By properties of $\|\cdot\|$ and $\mathrm{ent}$, we see that the {\it normalized entropy function} 
$$\mathrm{Ent}= \|\cdot\| \mathrm{ent}: int(C_{\Omega}) \rightarrow {\Bbb R}$$ 
is constant on each ray in $int(C_{\Omega})$ through the origin.

We choose  $\phi= [\Phi] \in \mathrm{Mod}(\varSigma)$, and 
we consider the mapping torus ${\Bbb T}_{\phi}$ with the suspension flow $\Phi^t$. 
Hereafter we fix the orientation of $\varSigma$ 
so that its normal direction coincides with the flow direction of $\Phi^t$. 
For  $S \subset \varSigma$, 
we define  $S^t \subset  {\Bbb T}_{\phi}$ to be 
the image of $S \times \{t\} \subset \varSigma \times \{t\}$ under the projection
$ p: \varSigma \times {\Bbb R} \rightarrow {\Bbb T}_{\phi}$.

\begin{thm}[Theorem~7 and Lemma in \cite{Fried}]
\label{thm_Fried2}
Let $\Phi: F \rightarrow F$ be a pseudo-Anosov homeomorphism 
with stable and unstable foliations $\mathcal{F}^s$ and $\mathcal{F}^u$ 
on an oriented surface $F$, 
and let $\phi= [\Phi]$. 
Let $\widehat{\mathcal{F}^s}$ and $\widehat{\mathcal{F}^u}$ 
denote the suspension of $\mathcal{F}^s$ and $\mathcal{F}^u$ by $\Phi$. 
If $\Omega$ is a fibered face on ${\Bbb T}_{\phi}$ with $[F] \in int(C_{\Omega})$, 
then 
for any minimal representative $F_a$ of any fibered class $a \in int(C_{\Omega})$, 
we can modify $F_a$ by isotopy which satisfies the followings. 
\begin{enumerate}
\item[(1)] 
$F_a$ is transverse to the flow $\Phi^t$, 
and the first return map $:F_a \rightarrow F_a$ is precisely the pseudo-Anosov monodromy $\Phi_a: F_a \rightarrow F_a$ 
of the fibration on ${\Bbb T}_{\phi}$ associated to $a$. 
Moreover $F_a$ is unique up to isotopy along flow lines. 

\item[(2)] 
The stable and unstable foliations of $\Phi_a$ are given by 
$\widehat{\mathcal{F}^s} \cap F_a$ and $\widehat{\mathcal{F}^u} \cap F_a$.  
\end{enumerate}
\end{thm}

Following \cite{Minakawa}, 
we introduce {\it flowbands} in  ${\Bbb T}_{\phi}$. 

\begin{definition}
\label{def_flowband}
Let $J_1$ and $J_2$ be embedded arcs in ${\Bbb T}_{\phi}$. 
Suppose that $J_1$ and $J_2$ are transverse to  $\Phi^t$. 
We say that $J_1$ {\it is connected to} $J_2$ ({\it with respect to} $\Phi^t$) 
if there exists a positive continuous function 
$\frak{t}: J_1 \rightarrow {\Bbb R}$ 
such that for any $ x\in J_1$, we have

\begin{itemize}
\item 
$\Phi^{\frak{t}(x)}(x) \in J_2$,  
$\Phi^t(x) \not\in J_2$ for $0 < t < \frak{t}(x)$, and 

\item 
the map $J_1 \to J_2$ given by $x \mapsto \Phi^{\frak{t}(x)} (x)$ 
is a homeomorphism. 
\end{itemize}
The {\it flowband} $[J_1, J_2]$ is defined by 
$$[J_1, J_2]=\{ \Phi^t(x)\ |\ x \in J_1,\ 0 \le t \le \frak{t}(x)\}.$$
\end{definition}
Flowbands are used to build branched surfaces in Section~\ref{subsection_Branched}.

\subsection{Clique polynomials} 
\label{subsection_clique}

We review some results on clique polynomials, developed by McMullen \cite{McMullen2}. 
As we will see in Section~\ref{subsection_Clique}, 
clique polynomials are useful to compute the dilatations of pseudo-Anosov  monodromies of fibrations on  fibered $3$-manifolds.

Let $(\Gamma, m)$  be a finite, directed graph $\Gamma$ 
with a metric $m: E(\Gamma) \rightarrow {\Bbb R}_+$ on the set of edges $E(\Gamma)$.  
The metric $m$ specifies the length of each edge. 
Parallel edges and loops are allowed.  
We sometimes denote $(\Gamma,m)$  by $\Gamma$ 
when $m$ is obvious. 
The {\it growth rate} $\lambda(\Gamma,m)$  is defined by 
$$\lambda(\Gamma,m)= \lim_{T \to \infty} N_0(T)^{1/T},$$
where $N_0(T)$ is the number of the closed, directed paths in $\Gamma$ of length $\le T$. 

When $m(e)$ is an integer for each $e \in E(\Gamma)$, 
we can add $m(e)-1$ new vertices along each edge $e$ 
to obtain a new directed graph $\Gamma'$ with 
the metric ${\bf 1}:E(\Gamma') \rightarrow  {\Bbb R}_+$ 
sending each edge to $1$. 
Then we have 
$$\lambda(\Gamma, m)= \lambda(\Gamma', {\bf 1}).$$

Suppose that $\phi$ is a pseudo-Anosov mapping class. 
Let $\mathfrak{p}: \tau \rightarrow \tau$ be a train track representative of $\phi$ given in Theorem~\ref{thm_PaPe} 
and let $\Gamma_{\mathfrak{p}}$ be the induced directed graph of  $\mathfrak{p}: \tau \rightarrow \tau$. 
Theorem~\ref{thm_PaPe} implies that 
$ (\Gamma_{\mathfrak{p}}, {\bf 1})$  satisfies 
\begin{equation}
\label{equation_dilatation}
\lambda(\phi)= \lambda(\Gamma_{\mathfrak{p}}, {\bf 1}).
\end{equation}

Let $G$ be a finite, undirected graph with no loops or parallel edges.
Let $w: V(G) \rightarrow {\Bbb R}_+$ be a weight on the set of vertices $V(G)$. 
The subset $K \subset V(G)$ forms a {\it clique} if 
they span a complete subgraph. 
(We allow  $K =\emptyset$.)
The {\it clique polynomial} of $(G, w)$ is defined by 
$$Q(t)= \sum_{ K} (-1)^{(\sharp K)} t^{w(K)},$$ 
where $\sharp K$ denotes the cardinality of $K$, 
the weight of $K$ is given by 
$w(K)= \sum_{v \in K} w(v)$, and the sum is over all cliques $K$'s of $G$. 
We sometimes denote  the weighted, undirected graph $(G, w)$ by  $G$ 
when  $w$ is obvious.

McMullen relates the growth rates $\lambda(\Gamma,m)$'s to the clique polynomials 
via the {\it curve complexes} of  $(\Gamma,m)$'s. 
Let $C \subset E(\Gamma)$ be a collection of edges which form a closed, directed loop. 
If $C$ never visits the same vertex twice, then $C$ is called a {\it simple curve}. 
A {\it multicurve} is a finite union of simple curves such that no two simple curves share a vertex. 
The {\it curve complex} of $(\Gamma,m)$ is the undirected graph $G$ together with the weight $w: V(G) \rightarrow {\Bbb R}_+$, 
which is obtained by taking a vertex for each simple curve $C$ of $\Gamma$, 
and then joining the two vertices $C_1$ and $C_2$ by an edge when $(C_1, C_2)$ is a multicurve of $\Gamma$.  
Then the metric $m$ on $E(\Gamma)$ induces the weight $w$ on $V(G)$ as follows. 
$$w(C)= \sum_{e \in C} m(e).$$ 

\begin{thm}[\cite{McMullen2}]
\label{thm_McMullen}
Let $(G, w)$ be the  curve complex  of $(\Gamma,m)$. 
Then  $\tfrac{1}{\lambda(\Gamma,m)}$ is equal to the 
the smallest positive root of the clique polynomial $Q(t)$ of $(G, w)$. 
\end{thm}

\noindent
By (\ref{equation_dilatation}) and Theorem~\ref{thm_McMullen}, 
we can compute the dilatations of pseudo-Anosovs by using the clique polynomials of curve complexes associated to the pseudo-Anosovs. 
We do not need to compute the characteristic polynomials of the incidence matrices for the dilatations. 
This observation is due to Birman~\cite{Birman}.

%
%

\subsection{Fibered classes of the magic manifold $N$}

In this section, we review some properties on $N$ which will be used in the paper. 
We give orientations of components, $K_1$, $K_2$ and $K_3$ of the $3$ chain link $\mathcal{C}_3$
as in Figure~\ref{fig_3braid}(1).  
Each component  bounds oriented $2$-punctured disks, 
$F_{\alpha}$, $F_{\beta}$ and $F_{\gamma}$ respectively. 
We set $\alpha= [F_{\alpha}]$, $\beta= [F_{\beta}]$, $\gamma= [F_{\gamma}] \in H_2(N, \partial N; {\Bbb Z})$. 
Then $\{\alpha, \beta, \gamma\}$ becomes a basis of $H_2(N, \partial N; {\Bbb Z})$. 
We denote the class $x\alpha+ y \beta+ z \gamma \in H_2(N, \partial N)$ by $(x,y,z)$. 
Thurston norm ball $U_N$ is the parallelepiped with vertices 
$ \pm \alpha = \pm (1,0,0)$, 
$\pm \beta = \pm (0,1,0)$, 
$\pm \gamma = \pm (0,0,1)$ and 
$\pm \alpha+ \beta+ \gamma = \pm (1,1,1)$
 (\cite[Example~3]{Thurston1}), 
see Figure~\ref{fig_Face_map}(1). 
(Note that the minimal representative of $\kappa$ 
 is taken to be the $3$-punctured sphere, embedded in $S^3 \setminus \mathcal{C}_3$, 
 which contains the point $\infty \in S^3 \setminus \mathcal{C}_3$.)

 Let $\mathrm{br}(\sigma_1^2 \sigma_2^{-1})$ be the $3$ components link in $S^3$ as in Figure~\ref{fig_3braid}(2), 
 i.e, it is the link obtained from the closed $3$-braid $\sigma_1^2 \sigma_2^{-1}$ together with the braid axis. 
 Then $N= S^3 \setminus C_3 $ is homeomorphic to $S^3 \setminus \mathrm{br}(\sigma_1^2 \sigma_2^{-1})$, 
 see Lemma~\ref{lem_summary}(1). 
 This implies that $N$ is a surface bundle over the circle 
 with a fiber of the $4$-punctured sphere. 
 Notice that every top dimensional face of $\partial U_N$ is a fibered face, because of the symmetries of $\mathcal{C}_3$. 
To study  monodromies of fibrations on $N$, we can pick a particular fibered face, for example 
the fibered face $\Delta$  with vertices 
$(1,0,0)$, $(1,1,1)$, $(0,1,0)$ and $(0,0,-1)$, see Figure~\ref{fig_Face_map}(1). 
The open face $int(\Delta)$  is written by 
\begin{eqnarray}
\label{equation_OpenFace}
int(\Delta)= \{(x,y,z)\ |\ x+y-z=1, \ x>0,\  y>0,\  x>z,\  y>z\}. 
\end{eqnarray}
Thurston norm $\|a\|$ of $a= (x,y,z) \in C_{\Delta}$ is given by $x+y-z$. 
An integral class $(x,y,z) \in C_{\Delta}$ is fibered (i.e, an integral class $(x,y,z)$ is in $ int(C_{\Delta)}$)
if and only if 
$x$, $y$ and $z$ are integers such that 
$x>0$, $y>0$,  $x>z$ and $y>z$, see (\ref{equation_OpenFace}).

Any class $a= (x,y,z) \in \Delta$ satisfies $z= x+ y-1$. 
Hence we can write such a class $a=(x,y,z)$ by $[x,y]$. 
Then we have 
$$int(\Delta)= \{[x,y]\ |\ 0 < x < 1,\ 0 < y < 1\},$$ 
see Figure~\ref{fig_Face_map}(2).

We denote by $T_{\alpha}$ the torus which is the boundary of a regular neighborhood of $K_1$. 
Let $a =(x,y,z) \in int(C_{\Delta})$ be a primitive integral class. 
We set $\partial_{\alpha} F_a= \partial F_a \cap T_{\alpha}$ 
which consists of the parallel simple closed curves on $T_{\alpha}$. 
We define $T_{\beta}$, $\partial_{\beta} F_a$ and $T_{\gamma}$, $\partial_{\gamma} F_a$ 
in a similar way.

\begin{lem}[\cite{KT} for (1)(3)(5), \cite{KT1} for (6)(7), \cite{KKT2} for (4)]
\label{lem_summary}
Suppose that  $a= (x,y,z) \in int(C_{\Delta})$ is a primitive integral class. 
\begin{itemize}

\item[(1)]
There is an orientation preserving homeomorphism 
$: S^3 \setminus \mathcal{C}_3 \rightarrow S^3 \setminus \mathrm{br}(\sigma_1^2 \sigma_2^{-1})$ 
which sends the minimal representative $F_{\alpha+ \beta}$ associated to $\alpha+ \beta$ to the oriented $3$-punctured disk 
bounded by the braid axis $K_3'$ as in Figure~\ref{fig_BraidHomeo}(4). 
Thus the pseudo-Anosov homeomorphism which represents the mapping class of $\mathrm{Mod}(\varSigma_{0,4})$ 
corresponding  to  $\sigma_1^2 \sigma_2^{-1}$ 
becomes the monodromy $\Phi_{\alpha+ \beta}: F_{\alpha+ \beta} \rightarrow F_{\alpha+ \beta}$ of the fibration 
associated to $\alpha+ \beta$.

\item[(2)]
The boundary slope of  $\partial_{\alpha} F_a$  
(resp. $\partial_{\beta} F_a$,  
$\partial_{\gamma} F_a$)  is given by 
$\tfrac{y+z}{-x}$  
(resp.  $ \tfrac{z+x}{-y}$,  
$ \tfrac{x+y}{-z}$).

\item[(3)] 
We have $\|a\|= x+y-z$. 
The number of the boundary components of $F_a$ 
is computed as follows. 
$$\sharp( \partial_{\alpha} F_{a}) = \gcd(x,y+z), \ \sharp (\partial_{\beta} F_{a} )=  \gcd(y,z+x),\ \sharp( \partial_{\gamma} F_{a}) = \gcd(z,x+y),$$
where $\gcd(0,w)$ is defined by $|w|$.  

\item[(4)]
Let $\Phi_{(x,y,z)}: F_{(x,y,z)} \rightarrow F_{(x,y,z)}$ be the monodromy  of the fibration on $N$ 
associated to $(x,y,z) \in int(C_{\Delta})$. 
Then $(y,x,z ) \in int(C_{\Delta})$, and 
$(\Phi_{(x,y,z)})^{-1}$  is conjugate to $\Phi_{(y,x,z)}$.

\item[(5)] 
The dilatation $\lambda(a) = \lambda_{(x,y,z)}$  is the largest root of  
$$ f_{(x,y,z)}(t)= t^{x+y-z}-t^x - t^y - t^{x-z}- t^{y-z}+1.$$

\item[(6)]
The (un)stable foliation  $\mathcal{F}_{a}$ of  $\Phi_a$ has a property such that 
each component  of $\partial_{\alpha} F_{a}$, $\partial_{\beta} F_{a}$ and  $\partial_{\gamma} F_{a}$ 
has $\tfrac{x}{\gcd(x,y+z)}$ prongs, $\tfrac{y}{\gcd(y,x+z)}$ prongs and $\tfrac{x+y-2z}{\gcd(z,x+y)}$ prongs respectively. 
 Moreover $\mathcal{F}_{a}$  does not have singularities in the interior of $F_{a} $.

\item[(7)]
$\mathcal{F}_a$ is an orientable foliation 
if and only if $x$ and $y$ are even and $z$ is odd. 
\end{itemize}
\end{lem}

\noindent
The proof of (2) is easy. 
For the convenience of the proof of Lemma~\ref{lem_MinRep}, 
we prove the claim (1).

\begin{proof}[Proof of Lemma~\ref{lem_summary}(1)] 
First of all, we observe that the link in Figure~\ref{fig_BraidHomeo}(2) is isotopic to ${\mathcal C}_3$ 
given in Figure~\ref{fig_BraidHomeo}(1). 
Observe also that the link in Figure~\ref{fig_BraidHomeo}(3) is isotopic to 
the braided link $\mathrm{br}(\sigma_1^2 \sigma_2^{-1})$ given in Figure~\ref{fig_BraidHomeo}(4). 
We use the link diagrams in (2) and (3). 
We cut the twice punctured disk ($\simeq F_{\alpha}$) bounded by the component $K_1$. 
Let $F_{\alpha}^1$ and $F_{\alpha}^2$ be the resulting twice punctured disks obtained from $F_{\alpha}$. 
We reglue these twice punctured disks twisting either $F_{\alpha}^1$ or $F_{\beta}^2$ by 360 degrees. 
Then we obtain the link in (3) which is isotopic to $\mathrm{br}(\sigma_1^2 \sigma_2^{-1})$. 
This implies that 
there exists an orientation preserving homeomorphism 
$h: S^3 \setminus {\mathcal C}_3 \rightarrow S^3 \setminus \mathrm{br}(\sigma_1^2 \sigma_2^{-1})$. 
Then one can check that $h$ sends the minimal representative of $\alpha+ \beta$ to the desired $3$-punctured disk. 
\end{proof}

\begin{center}
\begin{figure}
\includegraphics[width=5.3in]{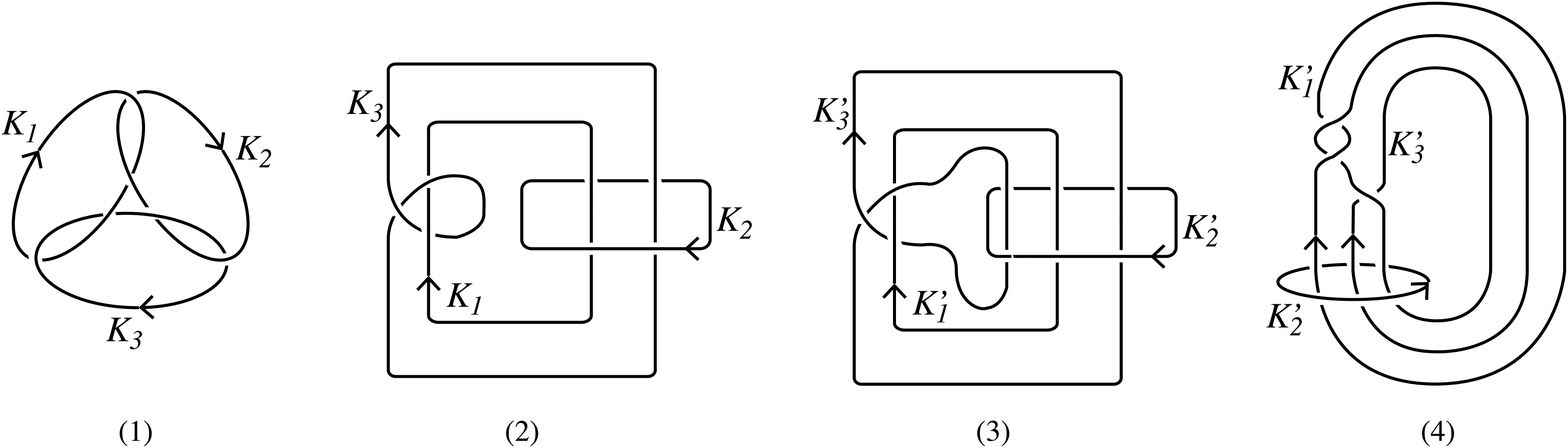}
\caption{(1)(2) $3$ chain link $\mathcal{C}_3$, (3)(4) Braided link $\mathrm{br}(\sigma_1^2 \sigma_2^{-1})$.}
\label{fig_BraidHomeo}
\end{figure}
\end{center}

Lemma~\ref{lem_summary}(4) allows us to 
focus on only fibered classes $(x,y,z) \in int(C_{\Delta})$ such that $y \ge x$ 
for the proof of Theorem~\ref{thm_algorithm}. 
We now  introduce two  bases 
$\{-\gamma, \beta, \alpha+ \beta \}$ and 
$\{ \alpha+ \beta+ \gamma, \beta, \alpha+ \beta \}$ 
of $H_2(N, \partial N, {\Bbb Z})$ to describe such fibered classes $(x,y,z)$. 
If $z \le 0$ (resp. $z \ge 0$), then we represent $(x,y,z)$ by the base 
$\{-\gamma, \beta, \alpha+ \beta \}$ (resp. $\{ \alpha+ \beta+ \gamma, \beta, \alpha+ \beta \}$ ). 
Let us define 
\begin{eqnarray*}
\Delta_{+}&=& \{(x,y,z)\in \Delta\ |\ y \ge x,\ z \ge 0\}, 
\\
\Delta_{-}&=& \{(x,y,z)\in \Delta\ |\ y \ge x,\ z\le 0\},  
\\
\Delta_0 &=& \{ (x,y,z)\in \Delta\ |\ y \ge x, \ z=0\}, 
\end{eqnarray*}
see Figure~\ref{fig_Face_map}(3)(4), 
and  define $\widehat{\Delta}_{\pm} \subset \Delta_{\pm}$ as follows. 
\begin{eqnarray*}
\widehat{\Delta}_+&=& \{(x,y,z) \in \Delta\ |\ x=y,\ z \ge 0 \}, 
\\
\widehat{\Delta}_-&=& \{(x,y,z) \in \Delta\ |\ x=y,\ z \le 0\}, 
\end{eqnarray*}
see Figure~\ref{fig_Face_map}(4). 
For non-negative integers $i,j$ and $k$,  
we define integral classes $(i,j,k)_{\pm}$, $(j,k)_0 \in C_{\Delta} $ to be 
\begin{eqnarray*}
(i,j,k)_+&=& i(1,1,1)+  j(0,1,0)+ k(1,1,0)= (i+k, i+j+k,i), 
\\
(i,j,k)_-&=& i(0,0,-1)+  j(0,1,0) + k(1,1,0)= (k, j+k, -i),
\\
(j,k)_0&=& (0,j,k)_+= (0,j,k)_-, 
\end{eqnarray*}
where 
$(1,1,1)= \alpha+ \beta+ \gamma$, $(0,0,-1)= - \gamma$, $(0,1,0)= \beta$, and $(1,1,0)= \alpha+ \beta$. 
The classes $(i,j,k)_{\pm}$ with $i,j,k >0$ are  said to be {\it non-degenerate}. 
Other classes $(i,j,k)_{\pm}$ are said to be {\it degenerate}. 
Note that an integral class $(i,j,k)_{\pm}$ is fibered  if and only if $i$, $j$ are non-negative integers and 
$k $ is a positive integer, see (\ref{equation_OpenFace}). 
If a fibered class $(x,y,z) \in int(C_{\Delta})$ satisfies $y \ge x$ and $z \le 0$ 
(resp. $y \ge x$ and $z \ge 0$), 
then $(x,y,z)$ is  written by $(i,j,k)_+$ (resp. $(i,j,k)_-$) for some $i,j\ge 0$ and $k \ge 1$.

We use the notations $f_{(i,j,k)_{\pm}}$ and $\lambda_{(i,j,k)_{\pm}}$ 
in the same manner as $f_{(x,y,z)}$ and $\lambda_{(x,y,z)}$ appeared in Lemma~\ref{lem_summary}. 
We denote by $[a]$, the projection of $a$ to the fibered face $\Delta$. 
For simplicity, 
we write $[(i,j,k)_{\pm}]= [i,j,k]_{\pm}$ and $[(j,k)_0]= [j,k]_0$. 
Clearly $[i,j,k]_{\pm}\in \Delta_{\pm}$ and $[j,k]_0 \in \Delta_0$. 
By claims (2)(3) in the following lemma, 
we find that coordinates $(i,j,k)_{\pm}$ are useful to study symmetries of the entropy function on $N$.

\begin{lem}
\label{lem_summary2} 
Let $a= (i,j,k)_{\pm} \in int(C_{\Delta})$ be a primitive integral class. 

\begin{itemize}

 \item[(1)]
 The dilatation $\lambda_{(i,j,k)_{\pm}}$ 
 is the largest root of 
 $$f_{(i,j,k)_{\pm}}(t)= t^{i+j+2k} - t^k - t^{i+k} -t^{j+k} - t^{i+j+k} +1.$$ 
 In particular $\lambda_{(i,j,k)_+}= \lambda_{(i,j,k)_-}$. 
 
 \item[(2)]
 The integral class $(j,i,k)_{\pm}$ is a fibered class in   $ int(C_{\Delta})$, and 
 the equality  $f_{(i,j,k)_{\pm}}= f_{(j,i,k)_{\pm}}$ holds. 
 In particular, $\lambda_{(i,j,k)_{\pm}}= \lambda_{(j,i,k)_{\pm}}$. 
 
 \item[(3)] 
Two classes $[i,j,k]_+$ and $[j,i,k]_+$ have a line symmetry about $x= \frac{1}{2}$, 
and 
$[i,j,k]_+$ and $[i,j,k]_-$ have a line symmetry about $y= -x+1$, 
see Figure~\ref{fig_Face_map}(5)(6). 
\end{itemize}
\end{lem}

\begin{proof}
The claim (1) holds by Lemma~\ref{lem_summary}(5). 
The first part of (2) follows from (\ref{equation_OpenFace}). 
The second part of (2) is obvious from the claim (1).  
The proof of (3) is easy to check. 
(cf. Corollary~2.7 and Remark~2.8 in \cite{KKT2}.) 
\end{proof}

Although all fibered classes $(i,j,k)_{\pm}$ and $(j,i,k)_{\pm}$ have the same Thurston norm $i+j+ 2k$ and same dilatation, 
 the topological types of their fibers could be different in general. 
 To see what the pseudo-Anosovs $\Phi_{(i,j,k)_{\pm}}$ and $\Phi_{(j,i,k)_{\pm}}$ look like, 
 we will see the curve complexes associated to $(i,j,k)_{\pm}$ and $(j,i,k)_{\pm}$ 
  in Section~\ref{subsection_Clique}.

\begin{center}
\begin{figure}
\includegraphics[width=5in]{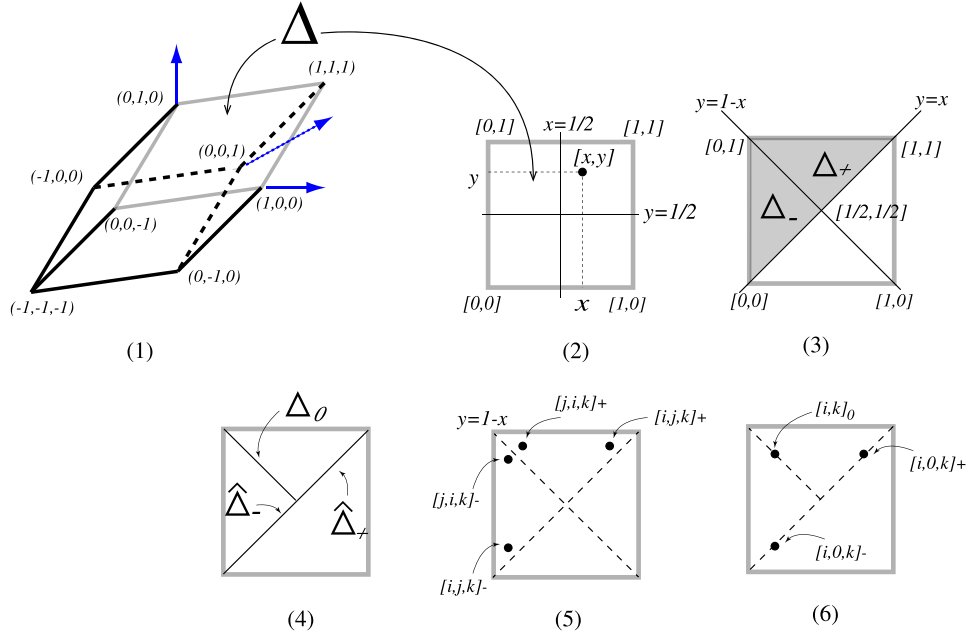}
\caption{(1) Thurston norm ball $U_N$  and the fibered face $\Delta$. 
(2) $[x,y ] \in \Delta$. 
(3)
$\Delta_{\pm}$. 
(4) $\Delta_0$, $\widehat{\Delta}_{\pm} $. 
(5) Projective classes of non-degenerate 
$[i,j,k]_{\pm}$ and $[j,i,k]_{\pm}$. 
(6) Projective classes of  degenerate 
$[i,0,k]_{\pm}$ and  $[i,k]_0= [0,i,k]_{\pm}$.}
\label{fig_Face_map}
\end{figure}
\end{center}

\section{Construction} 
\label{section_construction}

In Section~\ref{subsection_seed}, 
we construct the pseudo-Anosov homeomorphism on $\varSigma_{0,4}$ which represents the mapping class corresponding to 
the $3$-braid $\sigma_1^2 \sigma_2^{-1}$. 
It serves the  monodromy $\Phi_{\alpha+ \beta}: F_{\alpha+ \beta} \rightarrow F_{\alpha+ \beta}$ 
of the fibration on $N$ associated to $\alpha+ \beta$, and 
plays a rule as a pseudo-Anosov homeomorphism $\Phi$ in Theorem~\ref{thm_Fried2}. 
We also construct a train track representative 
$\mathfrak{p}_{\alpha+ \beta}: \tau_{\alpha+ \beta} \rightarrow \tau_{\alpha+ \beta}$ 
of $\phi_{\alpha+ \beta}= [\Phi_{\alpha+ \beta}]$.

Oertel used branched surfaces to describe Thurston norm and 
to study fibers of fibrations on hyperbolic, fibered $3$-manifolds.  
For basic definitions and results on  branched surfaces, see  \cite{LO,Oertel,Oertel2}. 
In Section~\ref{subsection_Minimal}, 
we find minimal representatives of non-fibered classes 
$\alpha$, $\beta$, $\kappa_{\pm}$, 
where $\kappa_+= \alpha+ \beta+ \gamma$ and 
$\kappa_-= -\gamma$. 
In Section~\ref{subsection_Branched}, by using minimal representatives found in Section~\ref{subsection_Minimal}, 
we build two branched surfaces $\mathcal{B}_{\pm}$ 
which carry fibers $F_{(i,j,k)_{\pm}}$ of fibrations associated to  any fibered class $(i,j,k)_{\pm}$. 
Then in Section~\ref{subsection_Train}, we construct the train track $\tau_{(i,j,k)_{\pm}}$ 
which carries the unstable foliation $\mathcal{F}_{(i,j,k)_{\pm}}$ 
of the pseudo-Anosov monodromy 
$\Phi_{(i,j,k)_{\pm}}: F_{(i,j,k)_{\pm}} \rightarrow F_{(i,j,k)_{\pm}}$ of the fibration  associated to $(i,j,k)_{\pm}$. 
In Section~\ref{subsection_Monodromies}, we construct  the  the pseudo-Anosov 
$\Phi_{(i,j,k)_{\pm}}: F_{(i,j,k)_{\pm}} \rightarrow F_{(i,j,k)_{\pm}}$ explicitly. 
In Section~\ref{subsection_TTmap}, we give an explicit construction of  the train track representative  
$\mathfrak{p}_{(i,j,k)_{\pm}}:  \tau_{(i,j,k)_{\pm}} \rightarrow \tau_{(i,j,k)_{\pm}}$ for  $\phi_{(i,j,k)_{\pm}}= [\Phi_{(i,j,i)_{\pm}}]$. 
We also construct 
the directed graph  $\Gamma_{(i,j,k)_{\pm}}$ induced by 
 $\mathfrak{p}_{(i,j,k)_{\pm}}: \tau_{(i,j,k)_{\pm}} \rightarrow \tau_{(i,j,k)_{\pm}}$ 
to indicate where each real edge of $\tau_{(i,j,k)_{\pm}}$  maps to under $\mathfrak{p}_{(i,j,k)_{\pm}}$.  
In Section~\ref{subsection_Clique}, we give the curve complex $G_{(i,j,k)_{\pm}}$ of $\Gamma_{(i,j,k)_{\pm}}$ and 
compute its clique polynomial $Q_{(i,j,k)_{\pm}}(t)$ whose 
largest root equals  the dilatation $\lambda_{(i,j,k)_{\pm}}$.

\subsection{Fibered class $\alpha+ \beta$}
\label{subsection_seed}

Let $L= L_{\mathcal{M}}: {\Bbb R}^2 \rightarrow {\Bbb R}^2$  be the linear map 
induced by 
$\mathcal{M}
=\left(
\begin{smallmatrix}
  3&2\\
  1&1
\end{smallmatrix}
\right) \in SL(2, {\Bbb Z})$. 
Since $\mathcal{M}$ has eigenvalues $\lambda^{\pm 1}(= 2 \pm \sqrt{3})$ 
with $\lambda^{+} >1$ and $0 < \lambda^{-1}< 1$, 
the linear map $L$ descends to the  Anosov diffeomorphism 
$f= f_{\mathcal{M}}: {\Bbb R}^2/ {\Bbb Z}^2 \rightarrow {\Bbb R}^2/ {\Bbb Z}^2$ on the torus. 
Figure~\ref{fig_seed} is an illustration of the image of the unit square 
with the origin (on the left bottom corner) under $L$. 

\begin{center}
\begin{figure}
\includegraphics[width=6in]{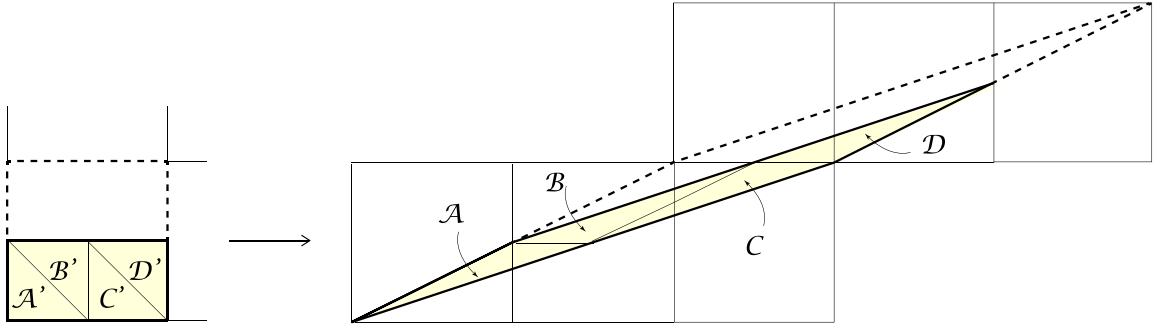}
\caption{
Rectangle $\mathcal{R}$ (left) and its image $L(\mathcal{R})$ (right). 
(The images of the isosceles right-angled triangles $\mathcal{A}'$, $\mathcal{B}'$, $\mathcal{C}'$ and $\mathcal{D}'$ 
under $L$ are the acute-angled triangles 
$\mathcal{A}$, $\mathcal{B}$, $\mathcal{C}$ and $\mathcal{D}$.)} 
\label{fig_seed}
\end{figure}
\end{center}

The linear map induced by  
$\left(
\begin{smallmatrix}
  -1&0\\
  0&-1
\end{smallmatrix}
\right) \in SL(2, {\Bbb Z})$
 defines a reflection 
$\mathfrak{r}: {\Bbb R}^2/ {\Bbb Z}^2 \rightarrow {\Bbb R}^2/ {\Bbb Z}^2$. 
The quotient of ${\Bbb R}^2/ {\Bbb Z}^2$ by $\mathfrak{r}$, denoted by  ${\Bbb S}$, 
is homeomorphic to a sphere. 
(${\Bbb S}$ is called a {\it pillowcase} because of its shape.) 
We set the points $a_i= (\frac{i}{2},0)$, $b_i = (\frac{i}{2}, \frac{1}{2}) \in {\Bbb R}^2$. 
Let $\mathcal{R} $ be the rectangle $a_0 a_2 b_2 b_0$ on $ {\Bbb R}^2$, see Figure~\ref{fig_rectangle}(1).  
Then ${\Bbb S}$ is obtained from $\mathcal{R}$ 
by identifying the three pairs of the oriented closed segments  
$\mathcal{K}=b_0b_1$ with $\mathcal{K}'=b_2 b_1$, 
$\mathcal{I}'= a_0b_0$ with $\mathcal{I}= a_2 b_2$, and 
$\mathcal{J}'= a_1 a_0$ with $\mathcal{J}= a_1 a_2$, 
see Figure~\ref{fig_rectangle}.  
If we let 
$\pi: {\Bbb R}^2 \rightarrow {\Bbb S}$ be the composition of the projections 
${\Bbb R}^2 \rightarrow {\Bbb R}^2/ {\Bbb Z}^2$ and ${\Bbb R}^2/ {\Bbb Z}^2 \rightarrow {\Bbb S}$, 
then the differentiable structure of ${\Bbb S}$  has the four singularities  
 $\mathfrak{b}_1= \pi(b_0)$,  $\mathfrak{b}_2= \pi(b_1)$, $\mathfrak{b}_3= \pi(a_1)$ 
 and $\mathfrak{b}_4= \pi(a_0)$ which lie on the corners of  ${\Bbb S}$. 
We set ${\Bbb B} = \{\mathfrak{b}_1, \mathfrak{b}_2, \mathfrak{b}_3, \mathfrak{b}_4\}$. 

\begin{center}
\begin{figure}
\includegraphics[width=3.5in]{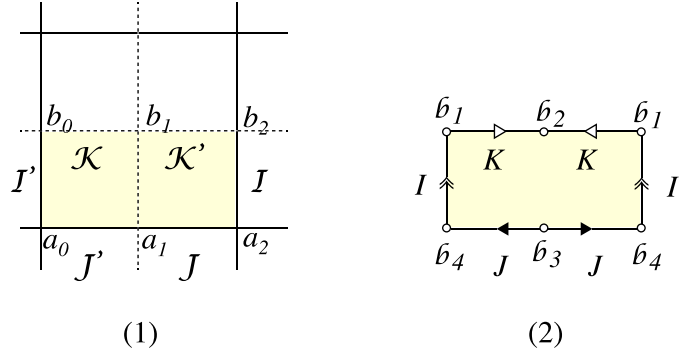}
\caption{(1)  $\mathcal{R}= a_0 a_2 b_2 b_0$. 
(2)  $4$-punctured sphere $ {\Bbb S} \setminus {\Bbb B}$.} 
\label{fig_rectangle}
\end{figure}
\end{center}

We have the identity  $f \circ \mathfrak{r} = \mathfrak{r} \circ f$, and this implies that 
$f: {\Bbb R}^2/ {\Bbb Z}^2 \rightarrow{\Bbb R}^2/ {\Bbb Z}^2 $ 
induces a homeomorphism  $\tilde{f}: {\Bbb S} \rightarrow {\Bbb S} $. 
Clearly ${\Bbb B}$ is invariant under $\tilde{f}$. 
(More concretely, 
$\tilde{f}(\mathfrak{b}_i)= \mathfrak{b}_i$ for $i= 1,4$, 
$\tilde{f}(\mathfrak{b}_2)=\mathfrak{b}_3$ and 
$\tilde{f}(\mathfrak{b}_3)=\mathfrak{b}_2$.) 
We see that 
the homeomorphism $\tilde{f}$ on ${\Bbb S}$ away from ${\Bbb B}$ inherits the (un)stable foliation of the Anosov $f$. 
Therefore, 
$f: {\Bbb R}^2/ {\Bbb Z}^2 \rightarrow{\Bbb R}^2/ {\Bbb Z}^2 $ 
induces  a pseudo-Anosov homeomorphism 
on $ {\Bbb S} \setminus {\Bbb B} \simeq  \varSigma_{0,4}$, 
see Figure~\ref{fig_rectangle}(2).  
Abusing the notation, we denote the pseudo-Anosov on ${\Bbb S} \setminus {\Bbb B} $ 
by the same notation $\tilde{f}$. 
Figure~\ref{fig_pApillow} is an illustration of $\tilde{f}: \varSigma_{0,4} \rightarrow \varSigma_{0,4}$. 
(cf. Figure~\ref{fig_seed}.) 
 Let $A'$, $B'$, $C'$ and $D'$ be isosceles right-angled triangles 
 whose vertices are punctures of $\varSigma_{0,4}$, see Figure~\ref{fig_pApillow}(left). 
 Their  images  under $\tilde{f}$, denoted by $A$, $B$, $C$ and $D$,  
 look as in Figure~\ref{fig_pApillow}(right).

Observe that 
the (un)stable foliation of $\tilde{f}$ has a $1$-pronged singularity at each puncture. 
If we regard the puncture $\mathfrak{b}_4$ as the boundary of the $3$-punctured disk, 
then the mapping class $[\tilde{f}]$ is written by a $3$-braid on the disk. 
Since  $\mathcal{M}$ is of the form 
$\left(
\begin{smallmatrix}
  1&1\\
  0&1
\end{smallmatrix}
\right)^2 
\left(
\begin{smallmatrix}
  1&0\\
  1&1
\end{smallmatrix}
\right),$
we see that 
$[\tilde{f}]$ is represented by the $3$-braid $ \sigma_1^2 \sigma_2^{-1}$ 
(reading the word from the right to the left). 
Equivalently $[\tilde{f}]$ is written by $(h_1)^2  \circ h_2^{-1}$, 
where $h_i$ denotes the mapping class which represents the positive half-twist 
about the segment $\mathfrak{b}_i\mathfrak{b}_{i+1}$, 
see Figure~\ref{fig_rectangle}(2).  

\begin{rem}
\label{rem_Homeo}
We have a natural homeomorphism 
$\tilde{h}:  \mathrm{br}(\sigma_1^2 \sigma_2^{-1}) \rightarrow {\Bbb T}_{[\tilde{f}]}$. 
The cusp of the component $K_2'$ (resp. $K_1'$) of the link $ \mathrm{br}(\sigma_1^2 \sigma_2^{-1})$ maps to (under $\tilde{h}$) 
the cusp corresponding to the orbit of $\mathfrak{b}_4$ (resp. $\mathfrak{b}_1$) of the suspension flow 
(see Figure~\ref{fig_3braid}(2)). 
The cusp of the component $K_3'$ maps to (under $\tilde{h}$) 
the cusp corresponding to the orbit $\mathfrak{b}_2$ (or $\mathfrak{b}_3$). 
\end{rem}

\noindent
By Lemma~\ref{lem_summary}(1) together with Remark~\ref{rem_Homeo}, 
we can regard $\tilde{f}: \varSigma_{0,4} \rightarrow \varSigma_{0,4}$
as the monodromy $\Phi_{\alpha+ \beta}: F_{\alpha+ \beta} \rightarrow F_{\alpha+ \beta}$  
of the fibration on $N$ associated to $\alpha+ \beta$. 
We denote $\tilde{f}$ by $\Phi_{\alpha+ \beta}$.

Next we turn to  a train track representative of $\phi_{\alpha+ \beta}= [\Phi_{\alpha+ \beta}]$. 
Let $\tau_{\alpha+ \beta}$ be a train track on $\varSigma_{0,4}$ as in Figure~\ref{fig_seedABCD}(1). 
Figures~\ref{fig_seedABCD}(1)(5) show that 
 $\tau_{\alpha+ \beta}$  is  invariant  under $\phi_{\alpha+ \beta}$. 
In fact, we have the image $\Phi_{\alpha+\beta}(\tau_{\alpha+ \beta})$  in Figure~\ref{fig_seedABCD}(2). 
(For the illustration of $\Phi_{\alpha+\beta}(\tau_{\alpha+ \beta})$, 
consider the image of the acute-angled triangle $B$, see Figure~\ref{fig_pApillow}(right).)
The train track $\Phi_{\alpha+ \beta}(\tau_{\alpha+ \beta})$ is 
isotopic to the one as  in Figure~\ref{fig_seedABCD}(5) 
which is carried by $\tau_{\alpha+ \beta}$.  
As a result, 
we get the desired train track representative 
$\mathfrak{p}_{\alpha+ \beta}: \tau_{\alpha+ \beta} \rightarrow \tau_{\alpha+ \beta}$ of $\phi_{\alpha+ \beta}$ 
whose  incidence  matrix of $\mathfrak{p}_{\alpha+ \beta}$ (with respect to the real edges $p$ and $q$) is equal to 
$\mathcal{M}= \left(
\begin{smallmatrix}
  3&2\\
  1&1
\end{smallmatrix}
\right)$.

In the rest of the paper, we consider the magic manifold $N$ of the form  
${\Bbb T}_{ \phi_{\alpha + \beta}} = \varSigma_{0,4} \times [0,1]/ \sim$, 
where $\sim$ identifies  $x \times \{1\}$ with $\Phi_{\alpha+ \beta}(x) \times \{1\}$ 
for each $x \in  \varSigma_{0,4} $. 
We investigate  the suspension flow $\Phi_{\alpha+ \beta}^t$ on ${\Bbb T}_{\phi_{\alpha + \beta}}$. 
We choose the orientation of ${\Bbb S} \setminus {\Bbb B}$ so that 
its  normal direction coincides with the flow direction of $\Phi_{\alpha+ \beta}^t$. 
We  fix the illustration of the $4$-punctured sphere ${\Bbb S} \setminus {\Bbb B}$ 
as in Figure~\ref{fig_rectangle}(2), but 
we often omit the names of the punctures $\mathfrak{b}_i$'s.

\begin{rem}
\label{rem_real-edge}
Loop edges of $\tau_{\alpha+ \beta}$ surrounding punctures 
$\mathfrak{b}_1$, $\mathfrak{b}_2$ are $\mathfrak{b}_3$ are infinitesimal edges 
for  $\mathfrak{p}_{\alpha+ \beta}: \tau_{\alpha+ \beta} \rightarrow \tau_{\alpha+ \beta}$. 
Other edges $p$ and $q$ are real edges. 
Each component of $\varSigma_{0,4} \setminus \tau_{\alpha+ \beta}$ is a once punctured monogon. 
(This comes from the fact that 
the (un)stable foliation of $\tilde{f}$ has a $1$-pronged singularity at each puncture.) 
In particular the component of $\varSigma_{0,4} \setminus \tau_{\alpha+ \beta}$ containing the puncture $\mathfrak{b}_4$ 
has exactly one cusp. 
\end{rem}

\begin{center}
\begin{figure}
\includegraphics[width=4in]{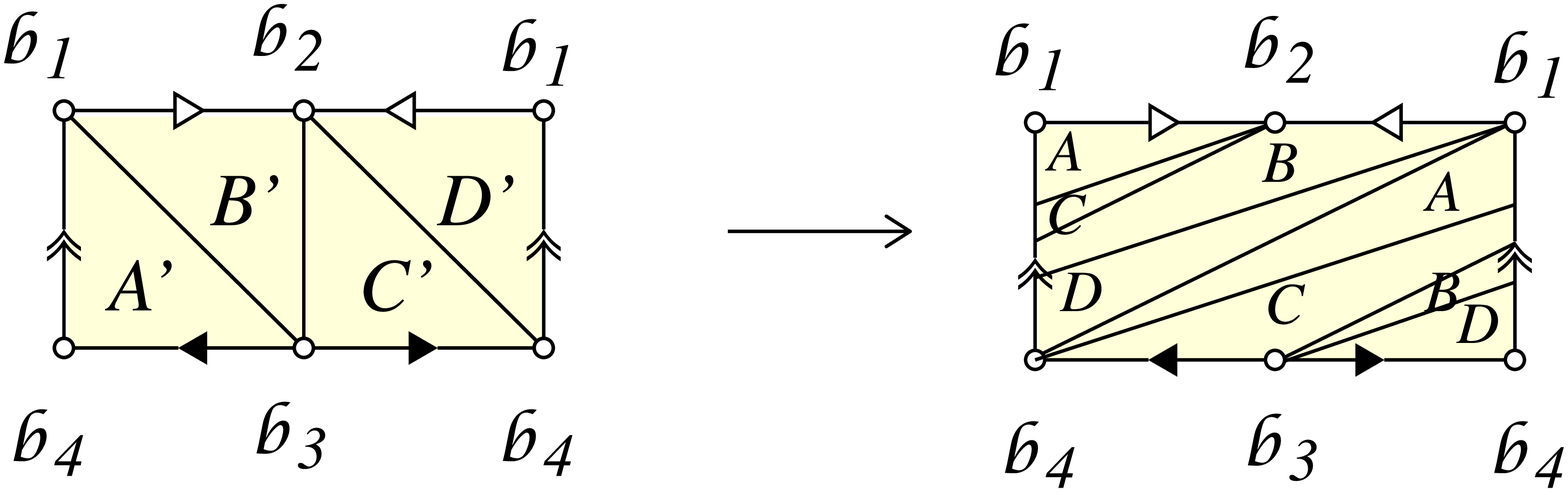}
\caption{$\tilde{f}(= \Phi_{\alpha+ \beta}): \varSigma_{0,4} \rightarrow \varSigma_{0,4}$. 
}
\label{fig_pApillow}
\end{figure}
\end{center}

\begin{center}
\begin{figure}
\includegraphics[width=4in]{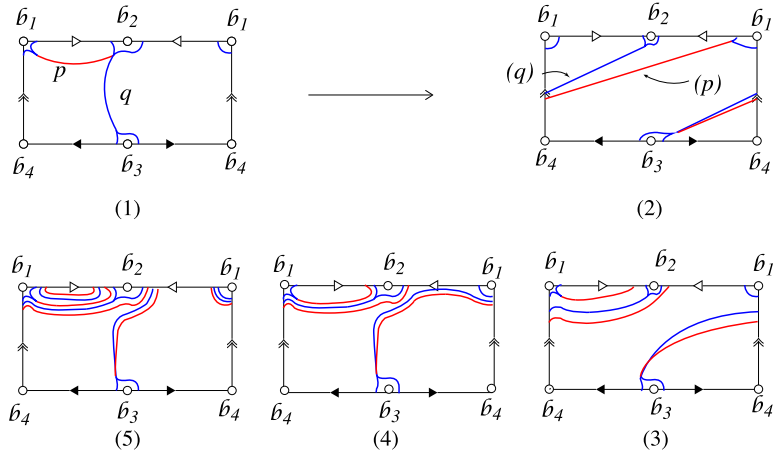}
\caption{(1) Train track $\tau_{\alpha+ \beta}$.  
(2)
The image of $\tau_{\alpha+ \beta}$ under $ \tilde{f}$, 
where the image of the edge $p$ etc.  is denote by $(p)$ etc. 
The image $\tilde{f}(\tau_{\alpha+ \beta})$ 
can be put in the tie neighborhood of $\tau_{\alpha+ \beta}$ under the isotopy, 
see  (2)$\rightarrow$(3)$\rightarrow$(4)$\rightarrow$(5). 
(From (4) $\rightarrow$ (5), 
the edges in (4) that are isotopic to 
the edges $(p)= \tilde{f}(p)$ and $(q)=\tilde{f}(q)$ in (2) cross over the segment $K$ 
(see Figure~\ref{fig_rectangle}(2) for $K$).} 
\label{fig_seedABCD}
\end{figure}
\end{center}

\subsection{Minimal representatives of $\alpha$, $\beta$, $-\gamma$ and $\alpha+ \beta+ \gamma$}
\label{subsection_Minimal}

First  we define several sets on ${\Bbb R}^2$ 
(see Figure~\ref{fig_parallel}(1)).   
Let $\mathcal{T}_{A}$ and $\mathcal{T}_{B}$
 be the triangles $ a_0 b_2 b_0$ and  $ a_0 a_2 b_2$ respectively. 
Let $\mathcal{P}_{K_-}$ and $\mathcal{P}_{K_+}$ 
be the parallelograms $a_0 a_1 b_4 b_3$ and $a_1 a_2 b_5 b_4$ respectively. 
See Figure~\ref{fig_parallel}(1).  
Note that 
$${\Bbb S}= \pi(\mathcal{R}) = \pi (\mathcal{T}_{A} \cup \mathcal{T}_{B}) 
= \pi (\mathcal{P}_{K_-} \cup \mathcal{P}_{K_+}),$$ 
where 
$\pi: {\Bbb R}^2 \rightarrow {\Bbb S}$ is the projection 
in Section~\ref{subsection_seed}. 
Let $\mathcal{U}$ be the oriented closed segment $a_0 b_2$. 
Let $\mathcal{V}$ and $\mathcal{W}$ be the oriented closed segments $b_3 a_0$ and $b_4 a_1$. 

\begin{center}
\begin{figure}
\includegraphics[width=5in]{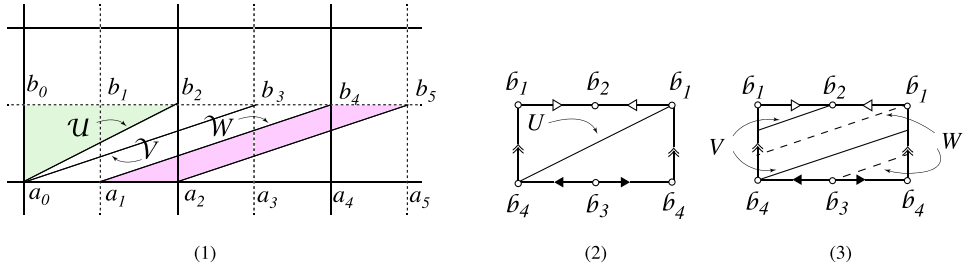}
\caption{(1) $\mathcal{T}_{A}= a_0 b_2 b_0$, $\mathcal{T}_{B}= a_0 a_2 b_2$, 
$\mathcal{P}_{K_+}= a_1 a_2 b_5 b_4$,   $\mathcal{P}_{K_-}= a_0 a_1 b_4 b_3$, 
$\mathcal{U}= a_0 b_2$, $\mathcal{V}= b_3 a_0$, $\mathcal{W}= b_4 a_1$. 
(2) $U$. (3) $V$, $W$.} 
\label{fig_parallel}
\end{figure}
\end{center}

\begin{center}
\begin{figure}
\includegraphics[width=4.5in]{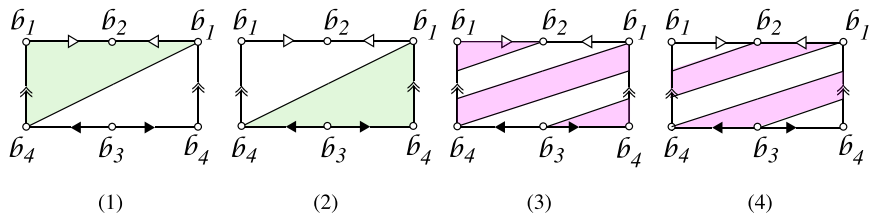}
\caption{(1) $T_{A}$. (2) $T_{B}$. (3) $P_{K_+}$. (4) $P_{K_-}$. }
\label{fig_triangle}
\end{figure}
\end{center}

Let $T_{A}$ be the image $\pi(\mathcal{T}_{A}) \subset {\Bbb S}$  removing all points in ${\Bbb B}$, 
i.e, 
$T_A= \pi(\mathcal{T}_{A}) \setminus {\Bbb B} $. 
We define $T_{B}$, $P_{K_+}$, $P_{K_-} \subset {\Bbb S} \setminus {\Bbb B}$ in the same manner. 
See Figure~\ref{fig_triangle}. 
Similarly, we let 
$$I= \pi(\mathcal{I}) \setminus {\Bbb B} = \pi(\mathcal{I}')  \setminus  {\Bbb B} .$$ 
Said differently, 
$I$ is obtained from $\pi(\mathcal{I})= \pi(\mathcal{I}') $ 
by removing the end points. 
We define $J, K, U, V,W \subset {\Bbb S} \setminus {\Bbb B}$ in the same manner. 
See Figures~\ref{fig_rectangle}(2) and \ref{fig_parallel}(2)(3). 
We choose the orientations of  $T_A$, $T_B$, $P_{K_{\pm}}$ 
which coincide with the ones  induced by the fiber $F_{\alpha+ \beta} = {\Bbb S} \setminus {\Bbb B}$ 
of the fibration associated to $\alpha+ \beta$.

\begin{rem}
\label{rem_ukappa}
We have $U \cap P_{K_+}= U$ and $U \cap P_{K_+}= \emptyset$, 
see Figures~\ref{fig_parallel}(2) and \ref{fig_triangle}(3)(4)). 
\end{rem}

We have $ \Phi_{\alpha+ \beta}(I)= U$. 
(We can check the equality by using Figure~\ref{fig_seed}.) 
This implies that 
$I^1= U^0$ in ${\Bbb T}_{\phi_{\alpha+ \beta}}$, 
and hence the flowband $[I^0, U^0]$ with respect to $\Phi_{\alpha+ \beta}^t$ is defined. 
Let 
\begin{eqnarray*}
F_A &=& T_{A}^0 \cup [I^0, U^0], \ \mbox{and}\ 
 \\
F_B &=& T_{B}^0 \cup  [I^0, U^0], 
\end{eqnarray*}
see Figure~\ref{fig_normOne}(1)(2). 
Observe that 
the both $F_A$ and $F_B$ are $3$-punctured spheres in  ${\Bbb T}_{\phi_{\alpha+ \beta}}$.

We find two more $3$-punctures spheres $F_{K_{\pm}}$ 
in ${\Bbb T}_{\phi_{\alpha+ \beta}}$. 
Note  that $\pi(\mathcal{J})= \pi(\mathcal{J}')$ and $\pi(\mathcal{K})= \pi(\mathcal{K}')$. 
We have that 
$\Phi_{\alpha+ \beta}(J)= V$ and 
$\Phi_{\alpha+ \beta}(K)= W$, 
because $\pi(L(\mathcal{J}))= \pi(\mathcal{V})$ 
and $\pi(L(\mathcal{K}))= \pi(\mathcal{W})$. 
Hence $J^1= V^0$ and $K^1= W^0$ in ${\Bbb T}_{\phi_{\alpha+ \beta}}$, 
and the flowbands $ [J^0, V^0] $ and $[K^0, W^0]$ are defined. 
Let
\begin{eqnarray*}
F_{K_+} &=& P_{K_+} \cup [J^0, V^0] \cup [K^0, W^0],\  \mbox{and}\ 
\\
F_{K_-} &=& P_{K_-} \cup [J^0, V^0] \cup [K^0, W^0], 
\end{eqnarray*}
see Figure~\ref{fig_normOne}(3)(4). 
We choose the orientations of  $F_A$, $F_B$ and $F_{K_{\pm}}$ 
so that they are extended by the orientations of $T_A^0$, $T_B^0$ and $P_{K_{\pm}}^0$. 
(Here we identify $T_A^0$ etc.  with $T_A$ etc.)

\begin{lem}
\label{lem_MinRep}
The $3$-punctures spheres $F_A$, $F_B$, $F_{K_+}$ and $F_{K_-}$ 
are  minimal representatives of $\alpha$, $\beta$, $\kappa_+(= \alpha+ \beta+ \gamma)$ and $\kappa_- (= - \gamma)$ respectively. 
\end{lem}

\begin{proof}[Proof of Lemma~\ref{lem_MinRep}]
Let $m_1$, $m_2$ and $m_3$ be the meridians of the components $K_1$, $K_2$ and $K_3$ of the $3$ chain link $\mathcal{C}_3$. 
We take oriented simple closed curves $S_1$, $S_2$ and $S_3$ in ${\Bbb T}_{\phi_{\alpha+ \beta}} \simeq N$ as in Figure~\ref{fig_scc}. 
We observe that the images of $m_1$, $m_2$ and $m_3$ under the homeomorphism 
$H= \tilde{h} \circ h: S^3 \setminus \mathcal{C}_3 \rightarrow  {\Bbb T}_{\alpha+ \beta}$ 
are $S_1$, $S_2$ and $S_3$ respectively (up to isotopy).  

Now, we consider the intersections $i(\cdot, \cdot)$ 
between the surface 
$F_{B}$ and either $S_1$, $S_2$ or $S_3$. 
We have 
$i(F_B, S_1)= 0$, 
$i(F_B, S_2)= 1$, and 
$i(F_B, S_3)= 0$. 
These imply that $F_B$ is a minimal representative of $\beta$. 
By cut and past construction of the union of surfaces $F_A \cup F_B $, 
we obtain a surface which is homeomorphic to $F_{\alpha+\beta} = {\Bbb S} \setminus {\Bbb B}$. 
Since $\beta= [F_B]$ and $\alpha+ \beta= [F_{\alpha+ \beta}]$,  
we conclude that $F_A$ is a minimal representative of $\alpha$.

Let us consider the intersections between $F_{K_-}$ and either $S_1$, $S_2$, or $S_3$. 
We have 
$i(F_{K_-}, S_1)= 0$, 
$i(F_{K_-}, S_2)= 0$, and 
$i(F_{K_-}, S_3)= -1$. 
Thus we conclude that $F_{K_-}$ is a minimal representative of $-\gamma$. 
Then we see that $F_{K_+}$ is a minimal representative of $\alpha+ \beta+ \gamma$, 
because 
$\alpha+ \beta= [F_{\alpha+ \beta}]$ and 
$F_{\alpha+ \beta}  = {\Bbb S} \setminus {\Bbb B}$ can be obtained from  $F_{K_+} \cup F_{K-}$ by cut and past construction. 
\end{proof}

\noindent
By Lemma~\ref{lem_MinRep}, 
it makes sense to denote 
$ F_A$, $F_B$, $ F_{K_+}$ and $F_{K_-}$ 
by $F_{\alpha}$, $F_{\beta}$, $F_{\kappa_+}$ and $F_{\kappa_-}$ respectively.

 \begin{center}
\begin{figure}
\includegraphics[width=5.5in]{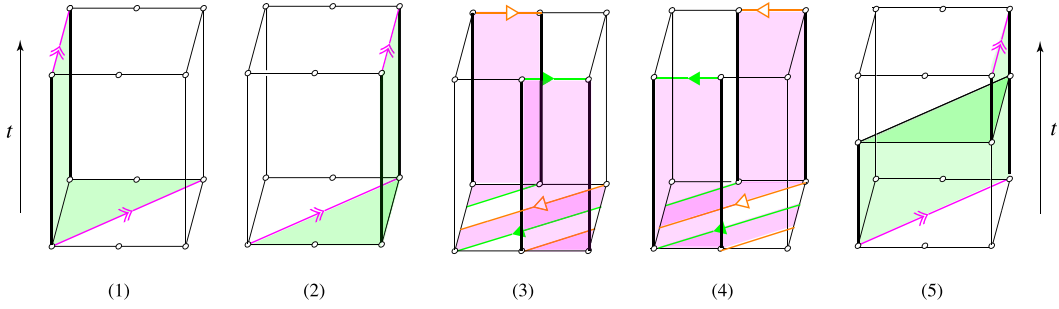}
\caption{(1)--(5) $3$-punctured spheres embedded in  ${\Bbb T}_{\phi_{\alpha+ \beta}}$ 
with the suspension flow $\Phi_{\alpha+\beta}^t$. 
The top and bottom surfaces are $\varSigma_{0,4} \times \{1\}$ and $\varSigma_{0,4} \times \{0\}$ respectively. 
They are identified by $\Phi_{\alpha+ \beta}$. 
Two edges with the same kind (red, green, yellow) 
are identified in  ${\Bbb T}_{\phi_{\alpha+ \beta}}$.
The vertical arrows with the labeling $t$ indicate the flow directions. 
(1) $F_{\alpha} = F_A$. 
(2) $F_{\beta} = F_B$. 
(3) $F_{\kappa_+}= F_{K_+}$. 
(4) $F_{\kappa_-}= F_{K_-}$. 
(5) $F_{\beta}^t$ for some $0 < t<1$.} 
\label{fig_normOne}
\end{figure}
\end{center}

\begin{center}
\begin{figure}
\includegraphics[width=2.5in]{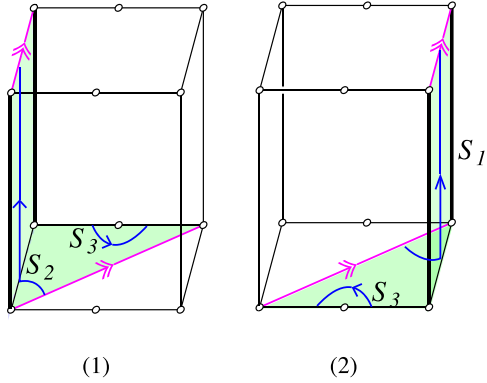}
\caption{(1) $S_2$, $S_3$ and $F_{\alpha}$ in ${\Bbb T}_{\phi_{\alpha+ \beta}}$. 
(2) $S_1$, $S_3$ and $F_{\beta}$ in ${\Bbb T}_{\phi_{\alpha+ \beta}}$.}
\label{fig_scc}
\end{figure}
\end{center}

 In the end of this subsection, we introduce surfaces $F_{\beta}^t$. 
 We denote by $F_{\beta}^t$, the oriented surface  in ${\Bbb T}_{\phi_{\alpha+ \beta}}$ 
 which is obtained from $F_{\beta}$  by pushing $F_{\beta}^0$ along the flow direction for $t \ge 0$ times, 
 see Figure~\ref{fig_normOne}(5). 
 Clearly $\beta= [F_{\beta}^t]$. 
  In the same manner we can define other surfaces 
 $F_{\kappa_{\pm}}^t$ and  $F_{\alpha+ \beta}^t$.

\subsection{Branched surfaces which carry fibers $F_{(i,j,k)_{\pm}}$} 
\label{subsection_Branched}

We first construct the branched surface $\mathcal{B}_+$. 
We choose $\delta, \epsilon>0$ so that 
$0 < \delta < 2 \delta < 1 - \epsilon$. 
We consider surfaces 
$F_{\kappa_+}^{\delta}$, $F_{\beta}^{2 \delta}$, $F_{\alpha+ \beta}^{1 - \epsilon} \subset  {\Bbb T}_{\phi_{\alpha+ \beta}}$. 
We have $\kappa_+= [F_{\kappa_+}^{\delta}]$, $\beta= [F_{\beta}^{2 \delta}]$ and 
$\alpha+ \beta= [F_{\alpha+ \beta}^{1- \epsilon}]$. 
Let  
\begin{equation}
\label{equation_pre-Bpulus}
\widehat{\mathcal{B}}_+=
F_{\kappa_+}^{\delta} \cup  F_{\beta}^{2 \delta} \cup F_{\alpha+ \beta}^{1- \epsilon} \subset {\Bbb T}_{\phi_{\alpha+ \beta}}.
\end{equation}
The intersection of  each pair of the three surfaces is as follows. 
\begin{enumerate}
\item[$(1_+)$] 
$F_{\alpha+ \beta}^{1 -\epsilon} \cap  F_{\beta}^{2 \delta} = I^{1- \epsilon}$, 

\item[$(2_+)$] 
$F_{\alpha+ \beta}^{1- \epsilon} \cap  F_{\kappa_+}^{\delta} = J^{1- \epsilon} \cup K^{1- \epsilon}$, 

\item[$(3_+)$] 
$F_{\beta}^{2 \delta} \cap F_{\kappa_+}^{\delta} = J^{2 \delta} \cup U^{\delta}$. 
(See Remark~\ref{rem_ukappa} and Figure~\ref{fig_normOne}(3)(5).)  
\end{enumerate}
The local picture near each intersection looks as in Figure~\ref{fig_branchedsurface}(1). 
A common property of the pairs is that 
locally, a surface near the intersection is parallel to $F_{\alpha+ \beta}$, 
and  it intersects with the  flowband 
which is the subset of the other surface (in fact $F_{\kappa_+}^{\delta}$ or $ F_{\beta}^{2 \delta}$). 
The branched surface $\mathcal{B}_+$ can be obtained  from $\widehat{\mathcal{B}}_+$ 
by modifying each flowband of $F_{\kappa_+}^{\epsilon}$ and $ F_{\beta}^{2 \epsilon}$ 
as in Figure~\ref{fig_branchedsurface}(2) 
so that the modifications agree with the orientations of the two surfaces. 
Each point in the intersection of surfaces belongs to the {\it branched locus} of $\mathcal{B}_+$ 
(that is, the union of points of the branched surface none of whose neighborhood are manifolds).

We build  the branched surface $\mathcal{B}_-$ in the same manner: 
Take $F_{\kappa_-}^{\delta}$ which represents $\kappa_-$ instead of $F_{\kappa_+}^{\delta}$. 
Let 
\begin{equation}
\label{equation_pre-Bminus}
\widehat{\mathcal{B}}_-=
F_{\kappa_-}^{\delta} \cup  F_{\beta}^{2 \delta} \cup F_{\alpha+ \beta}^{1- \epsilon} \subset {\Bbb T}_{\phi_{\alpha+ \beta}}.
\end{equation}
We have 
\begin{enumerate}
\item[$(1_-)$]
$F_{\alpha+ \beta}^{1- \epsilon} \cap  F_{\beta}^{2 \delta} = I^{1- \epsilon}$, 

\item[$(2_-)$]
$F_{\alpha+ \beta}^{1- \epsilon} \cap  F_{\kappa_-}^{\delta} = J^{1- \epsilon} \cup K^{1- \epsilon}$, 
 
\item[$(3_-)$]
$F_{\beta}^{2 \delta} \cap F_{\kappa_-}^{\delta} = J^{2 \delta}$. 
(See Remark~\ref{rem_ukappa} and Figure~\ref{fig_normOne}(4)(5).)  
\end{enumerate}
The branched surface $\mathcal{B}_-$ is obtained from $\widehat{\mathcal{B}}_-$ 
by modifying each flowband of $F_{\kappa_-}^{\epsilon}$ and $ F_{\beta}^{2 \epsilon}$ 
in the similar manner as in the construction of $\mathcal{B}_+$.

Figure~\ref{fig_identify}(1) (resp. Figure~\ref{fig_identify}(3)) 
illustrates   three pieces (bottom, middle, top) 
for building $\mathcal{B}_+$ (resp. $\mathcal{B}_- $). 
In this  figure, we have 
two kinds (solid/broken) of  segments without arrows and 
two kinds (solid/broken) of segments with arrows. 
The broken segments without arrows are parts of the orbits of punctures 
$\mathfrak{b}_1$, $\mathfrak{b}_2$, $\mathfrak{b}_3$ and $ \mathfrak{b}_4$ 
(which are circles in the figure)  under the flow. 
In other words, they lie on the cusps of ${\Bbb T}_{\phi_{\alpha+ \beta}} \simeq N$. 
We now explain our convention of Figure~\ref{fig_identify}(1)(3). 
Firstly,   
there are some pairs of  segments with the same labeling. 
(Exceptionally, the three segments in Figure~\ref{fig_identify}(1) have the labeling L6.) 
The two segments with the same labeling mean that 
one of them is connected to the other with respect to the flow, 
see Definition~\ref{def_flowband}.  
For the exceptional labeling L6 in Figure~\ref{fig_identify}(1), 
the bottom segment with the labeling L6 is connected to the middle segment with L6, and 
the top segment with L6 is connected to the bottom segment with the same labeling. 
Secondly, we also identify the segment having the labeling $L*$ with the segment having the labeling $L*'$. 
The resultant belongs to the branched locus. 
(For example,  the  segment with the labeling $L4'$ and the two segments with the labeling $L4$ 
are identified, and the resultant segment belongs to the branched locus.)  
Lastly, 
we can obtain the whole pictures of $\mathcal{B}_{\pm}$ 
if we insert a suitable flowband between every  two segments with the same labeling. 
For example, the bottom segment with the labeling $L4$ is connected to the top segment with the same labeling $L4$. 
We insert a suitable flowband (of the form $[K^{1/3 + \epsilon_0}, K^{1 - \epsilon}]$) between them. 
Also the top segment with the labeling $5$ is connected to the bottom segment with the same labeling $5$. 
Thus we insert the flowband 
$[K^1, W^{\delta}]= [W^0, W^{\delta}]$  between them. 
(Note that $K^1 = W^0$ in ${\Bbb T}_{\phi_{\alpha+ \beta}}$.)
Under the identification of segments and inserting suitable flowbands, 
each polygon bounded by the solid segments becomes a {\it sector} of the branched surface. 
In general, the {\it sectors} of the branched surface $\mathcal{B}$ are the closures in $\mathcal{B}$ 
of the components of $\mathcal{B} \setminus (\mbox{the\ branched\ locus\ of}\  \mathcal{B}$).

We turn to find surfaces carried by $\mathcal{B}_{\pm}$. 
To do this, 
given a fibered class $(i,j,k)_{\pm}$ (hence $i \ge 0$, $j \ge 0$ and $k \ge 1$), 
we assign  these integers $i $, $j $ and $k$ 
for the sectors of $\mathcal{B}_{+}$ (resp. $\mathcal{B}_{-}$)
 as in Figure~\ref{fig_identify}(2) (resp. Figure~\ref{fig_identify}(4)).  
This  is a natural assignment, which we explain the reason now. 
We assign  integers $i \ge 0 $, $j \ge 0$ and $k \ge 1$  to 
$F_{\kappa_+}^{\delta}$, $F_{\beta}^{2 \delta}$ and  $F_{\alpha+ \beta}^{1- \epsilon}$ 
(resp. $F_{\kappa_-}^{\delta}$, $F_{\beta}^{2 \delta}$ and $F_{\alpha+ \beta}^{1- \epsilon}$) 
consisting of $\widehat{\mathcal{B}}_{+}$ (resp. $\widehat{\mathcal{B}}_{-}$). 
Then we reconstruct $\mathcal{B}_{+}$ (resp. $\mathcal{B}_{-}$)  with the integers assigned. 
What we obtain is the assignment of integers in question. 
Then  the branched surface $\mathcal{B}_{\pm}$ enjoys the branch equations of a particular type 
such that $X= X'$, $Y= Y'$ and $(X,Y) \in \{(i,j), (j,i), (k,i), (k,j)\}$ 
as in Figure~\ref{fig_branchedsurface_1d}(1). 
Thus this assignment  determines a surface $ S_{(i,j,k)_{\pm}}$ 
which is carried by $\mathcal{B}_{\pm}$. 
(See the illustration of Figure~\ref{fig_branchedsurface_1d}(3) which shows the surface induced by some branch equation.) 
Said differently, 
$S_{(i,j,k)_{\pm}}$ is obtained from the union 
$$\widehat{S}_{(i,j,k)_{\pm}}= (i \ \mbox{parallel\ copies\ of\ } F_{\kappa_{\pm}}) \cup 
(j \ \mbox{parallel\ copies\ of\ } F_{\beta}) \cup (k \ \mbox{parallel\ copies\ of\ } F_{\alpha+ \beta})$$ 
by cut and past construction of surfaces (cf. Figure~\ref{fig_branchedsurface}). 
Therefore $ (i,j,k)_{\pm}= [S_{(i,j,k)_{\pm}}]$.

\begin{center}
\begin{figure}
\includegraphics[width=3in]{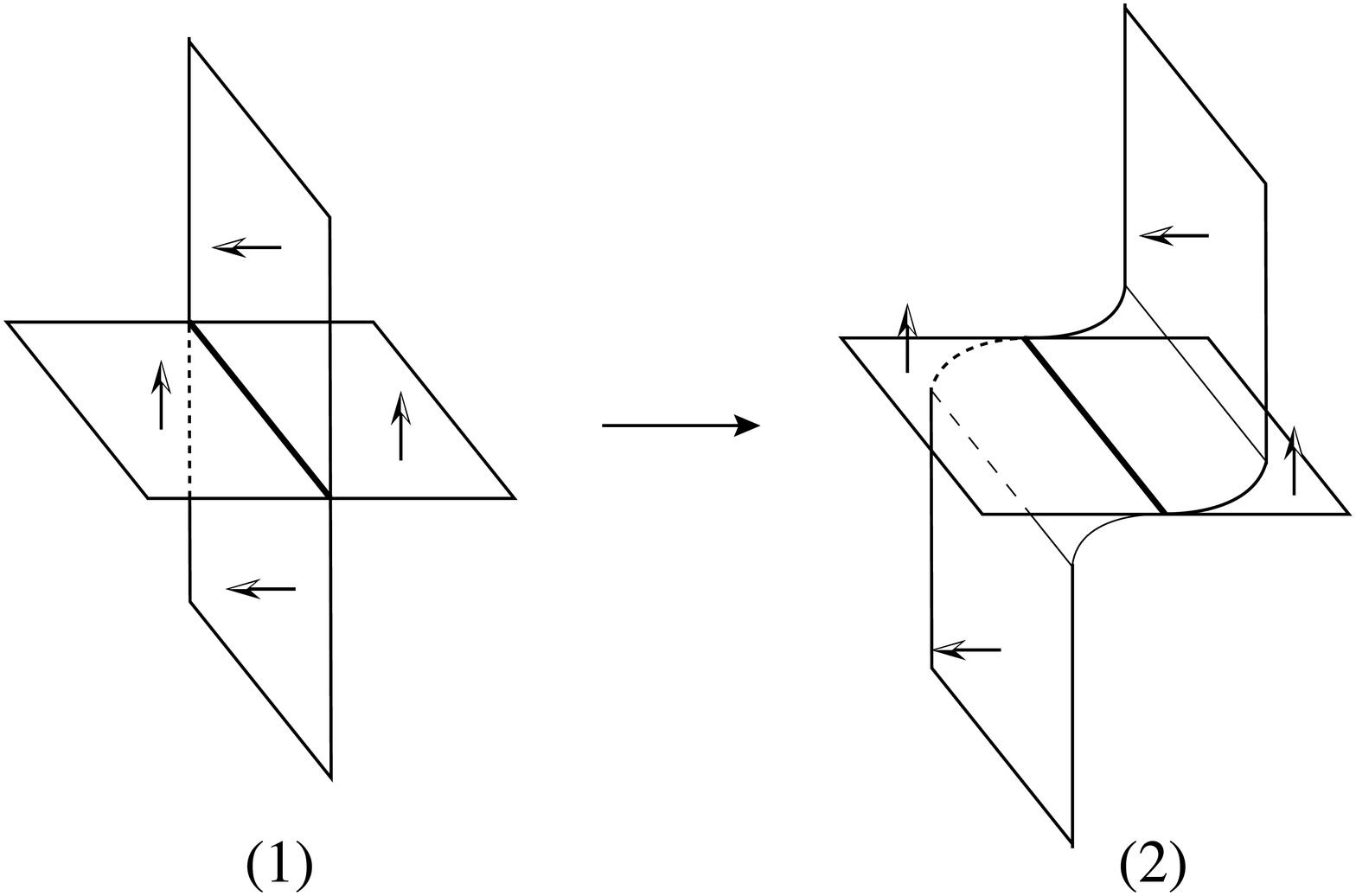}
\caption{(1) Near the intersection of the oriented surfaces
(Arrows shows the normal directions). 
(2) Modification near the intersection 
which agrees with the orientations of two surfaces.}
\label{fig_branchedsurface}
\end{figure}
\end{center}

\begin{center}
\begin{figure}
\includegraphics[width=4in]{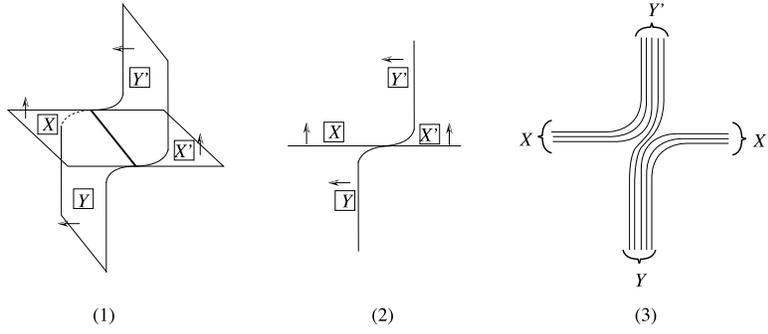}
\caption{(1) Branch equation $X+Y= Y'+X'$. 
(2) Side view of (1) in this figure.  
(3) Side view of the surface induced by the branch equation. 
(In this case $X= X'=3$, $Y= Y'=5$.)}
\label{fig_branchedsurface_1d}
\end{figure}
\end{center}

\begin{center}
\begin{figure}
\includegraphics[width=5in]{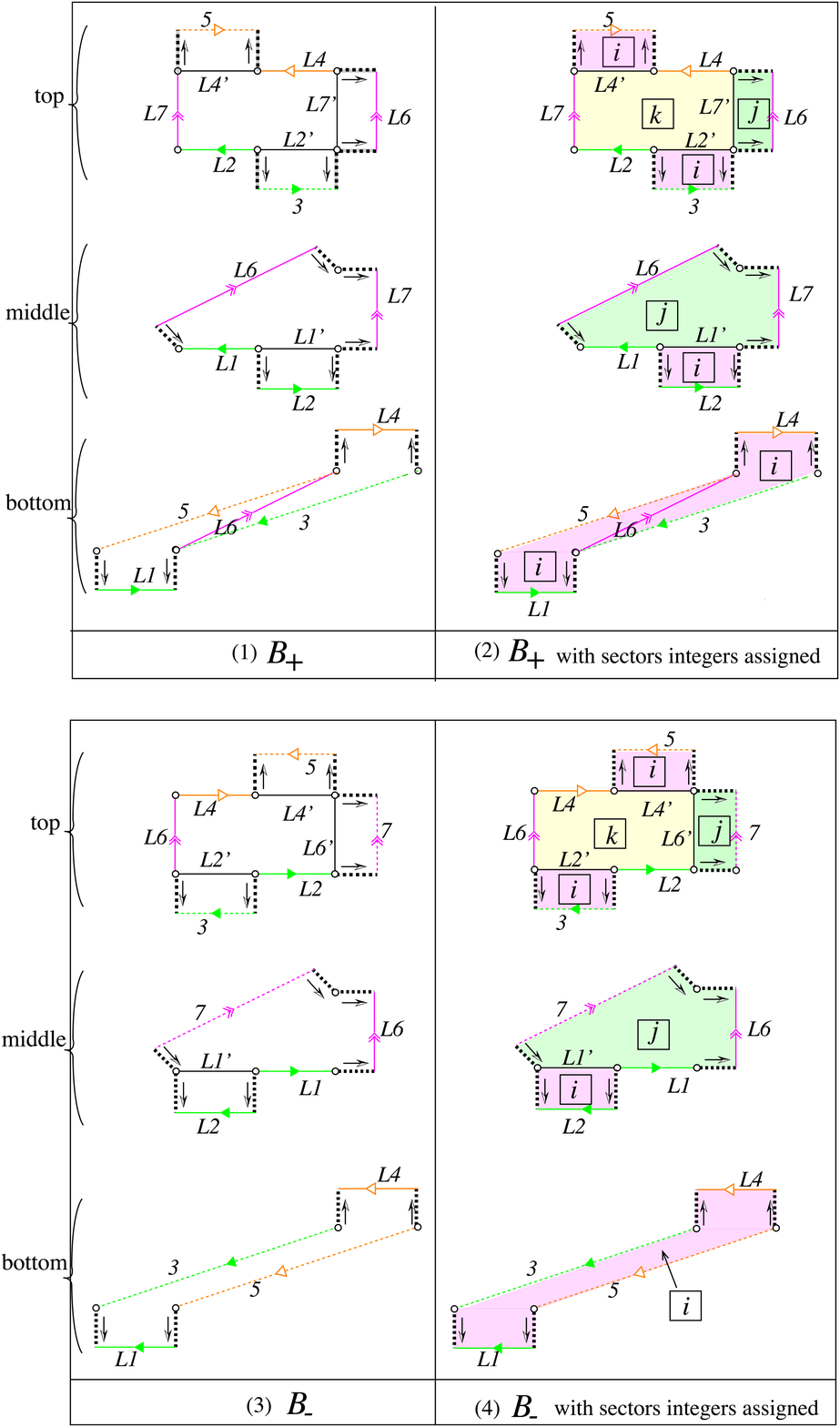} 
\caption{
(1) Three pieces for building $\mathcal{B}_+$. 
(2) Assignment of sectors of  $\mathcal{B}_+$. 
(3) Three pieces for building $\mathcal{B}_-$. 
(4) Assignment of sectors of  $\mathcal{B}_-$. 
The types of arrows (red, green, yellow) are compatible with the ones illustrated in Figure~\ref{fig_normOne}.}  
\label{fig_identify}
\end{figure}
\end{center}

\begin{lem}
\label{lem_fiber}
The surface $S_{(i,j,k)_{\pm}}$ 
is the minimal representative of   the fibered class $(i,j,k)_{\pm}$.  
\end{lem}

\begin{proof}
By definition, 
$\Delta_{\pm}$ is a convex hull in $\Delta$ 
containing $\kappa_{\pm}$ ($\kappa_+= [1,1]$, $\kappa_-=[0,0]$), $\beta= [0,1]$, 
$\tfrac{\alpha+ \beta}{2}= [\tfrac{1}{2}, \tfrac{1}{2}]$, 
see Figure~\ref{fig_Face_map}(3). 
The fibered class $(i,j,k)_{\pm}$ is in the cone over $\Delta_{\pm}$, 
and the surface $S_{(i,j,k)_{\pm}}$ is built from $\widehat{S}_{(i,j,k)_{\pm}}$ 
(which is the union of  parallel copies of the minimal representatives $F_{\kappa_{\pm}}$, $F_{\beta}$ and $F_{\alpha+ \beta}$) 
by cut and past construction. 
Thus $S_{(i,j,k)_{\pm}}$ must be the minimal representative of $(i,j,k)_{\pm}$. 
\end{proof}

\begin{rem}
\label{rem_fiber}
By Lemma~\ref{lem_fiber}, 
we write $F_{(i,j,k)_{\pm}}= S_{(i,j,k)_{\pm}}$, and 
$F_{(i,j,k)_{\pm}}$ becomes a fiber of the fibration associated to  $(i,j,k)_{\pm}$. 
By Theorem~\ref{thm_Fried2}, 
$F_{(i,j,k)_{\pm}}$ (up to isotopy) is transverse to the flow $\Phi_{\alpha+ \beta}^t$. 
The first return map  $\Phi_{(i,j,k)_{\pm}}: F_{(i,j,k)_{\pm}} \rightarrow F_{(i,j,k)_{\pm}}$ 
with respect to  $\Phi_{\alpha+ \beta}^t$ becomes 
the pseudo-Anosov monodromy of the fibration associated $(i,j,k)_{\pm}$. 
\end{rem}

\subsection{Train tracks $\tau_{(i,j,k)_{\pm}}$}   
\label{subsection_Train}

We note that the unstable foliation $\mathcal{F}_{\alpha+ \beta}$ of $ \Phi_{\alpha+ \beta}$ 
is carried by  $\tau_{\alpha + \beta}$. 
Let $\widehat{\mathcal{F}} $ be the suspension of $\mathcal{F}_{\alpha+ \beta}$ by $\Phi_{\alpha+ \beta}$ 
which is the $2$-dimensional foliation of ${\Bbb T}_{\phi_{\alpha+ \beta}}$.  
We now construct the branched surfaces $\mathcal{B}_{\Delta_{\pm}}$, 
each of which carries  $\widehat{\mathcal{F}}$. 
Let $\delta$ and $\epsilon$ be as in Section~\ref{subsection_Branched}, 
that is $\delta$ and $\epsilon$ are constants such that 
$0 < \delta < 2 \delta< 1-\epsilon$. 
Hereafter we fix $\delta= \tfrac{1}{3}$. 
We choose two families $\{\tau_t^+\}_{0 \le t \le 1}$ and $\{\tau_t^-\}_{0 \le t \le 1}$ of train tracks 
with the following properties. 
(See Figure~\ref{fig_Bsurface_latter}, 
in which the  time $t$  increases along arrows.) 
\begin{enumerate}
\item[(1)]
$\tau_0^{+}=\tau_0^{-} =\Phi_{\alpha+ \beta}(\tau_{\alpha+ \beta})$. 

\item[(2)] 
$\tau_t^{+}= \tau_t^{-}=\tau_{\alpha+ \beta}$ for $1 - \epsilon_0  \le t \le 1$. 

\item[(3)] 
$\tau_t^{\pm}$ is obtained from $\tau_s^{\pm}$ by folding edges of $\tau_s^{\pm}$ 
for each $0 \le s < t \le 1 - \epsilon $, 
or $\tau_t^{\pm}$ is isotopic to $\tau_s^{\pm}$, 

\item[(4)] 
$\tau_t^+= \tau_t^{-}$ for  $0 \le t \le \frac{2}{3}$, 
and  $\{\tau_t^{\pm}\}_{0 \le t \le \frac{2}{3}}$ 
is given as in Figure~\ref{fig_Bsurface_latter}(3),  

\item[(5)] 
$\{\tau_t^{+}\}_{\frac{2}{3} < t \le 1 - \epsilon}$ (resp. $\{\tau_t^{-}\}_{\frac{2}{3} < t \le 1 - \epsilon}$)
is given  as in Figure~\ref{fig_Bsurface_latter}(1) (resp. (2)). 
\end{enumerate}
In Figure~\ref{fig_Bsurface_latter}, 
$\sigma \xrightarrow[\mathrm{folding}]{}  \tau$ 
(resp. $\sigma \xrightarrow[\mathrm{isotopy}]{}  \tau$) 
means that $\tau$ is obtained from $\sigma$ by folding edges of $\sigma$, 
(resp. $\tau$ is isotopic to $\sigma$). 
Observe that non-loop edges  of $\tau_{1/3}^{+}=\tau_{1/3}^{-}$ 
(resp. $\tau_{2/3}^{+}= \tau_{2/3}^{-}$) 
do not intersect with $V$ and $W$ (resp. with $U$), see Figure~\ref{fig_Bsurface_latter}(4)(5).  
The branched surfaces $\mathcal{B}_{\Delta_{\pm}} \subset {\Bbb T}_{\phi_{\alpha+ \beta}}$ are defined to be 
\begin{equation}
\label{equation_Bsurface}
\mathcal{B}_{\Delta_{\pm}}= \bigcup_{0 \le t \le 1} \tau_t^{\pm} \times \{t\}/ \sim ,
\end{equation}
where $\sim$ identifies $(x,1)$ and $(\Phi_{\alpha + \beta}(x),0)$ for $x \in \tau_{\alpha+ \beta}$.

\begin{center}
\begin{figure}
\includegraphics[width=6in]{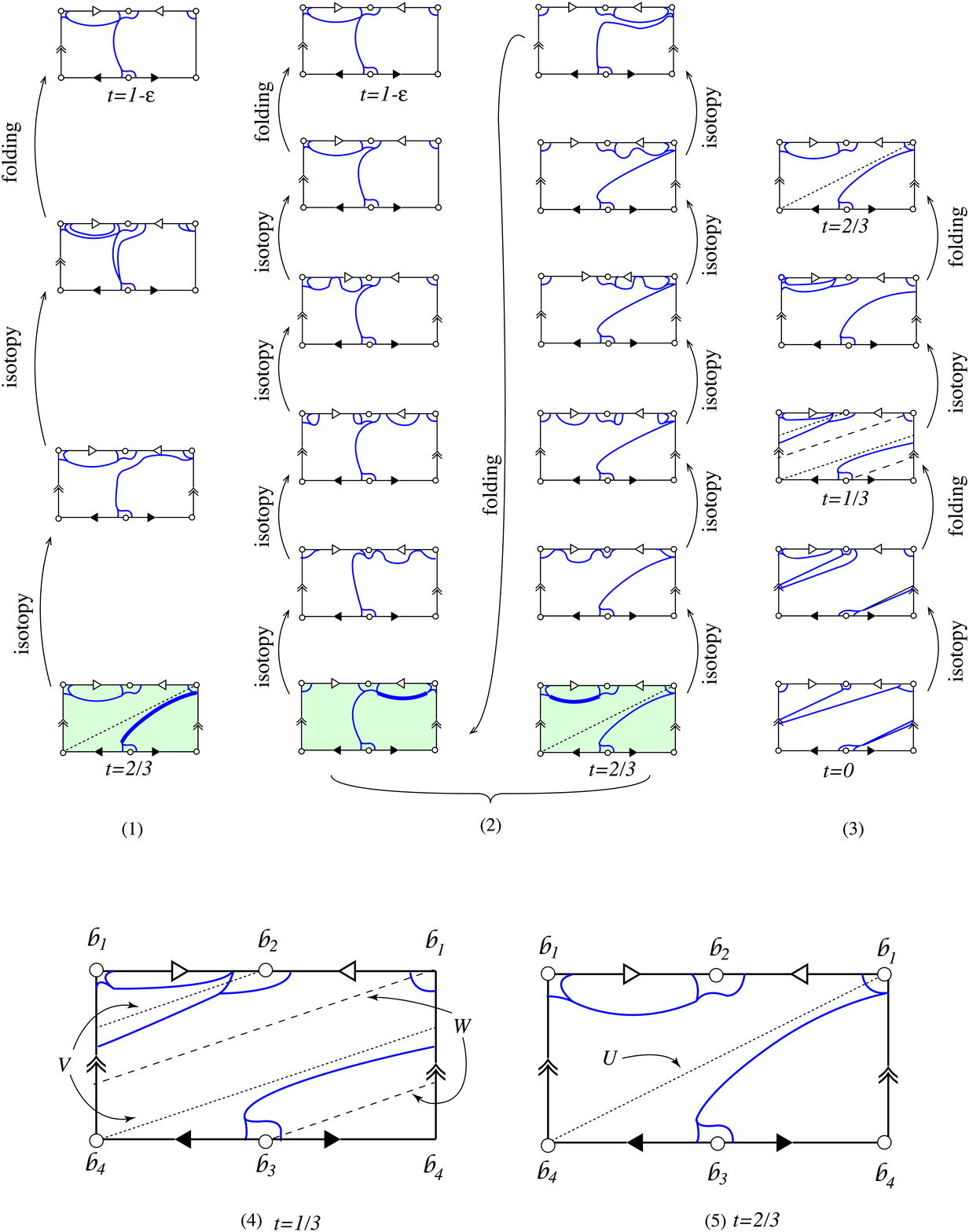}
\caption{
(Some of times $t$ are indicated near the puncture $\mathfrak{b}_3$.) 
(1) $\{\tau_t^{+}\}_{2/3 < t \le 1 - \epsilon}$. 
(2) $\{\tau_t^{-}\}_{2/3< t \le 1 - \epsilon}$. 
(3) $\{\tau_t^{\pm}\}_{0 \le t \le  2/3}$. 
(4) $\tau_{1/3}^{+}= \tau_{1/3}^{-}$. 
(5) $\tau_{2/3}^{+}= \tau_{2/3}^{-}$.} 
\label{fig_Bsurface_latter}
\end{figure}
\end{center}

\begin{center}
\begin{figure}
\includegraphics[width=5in]{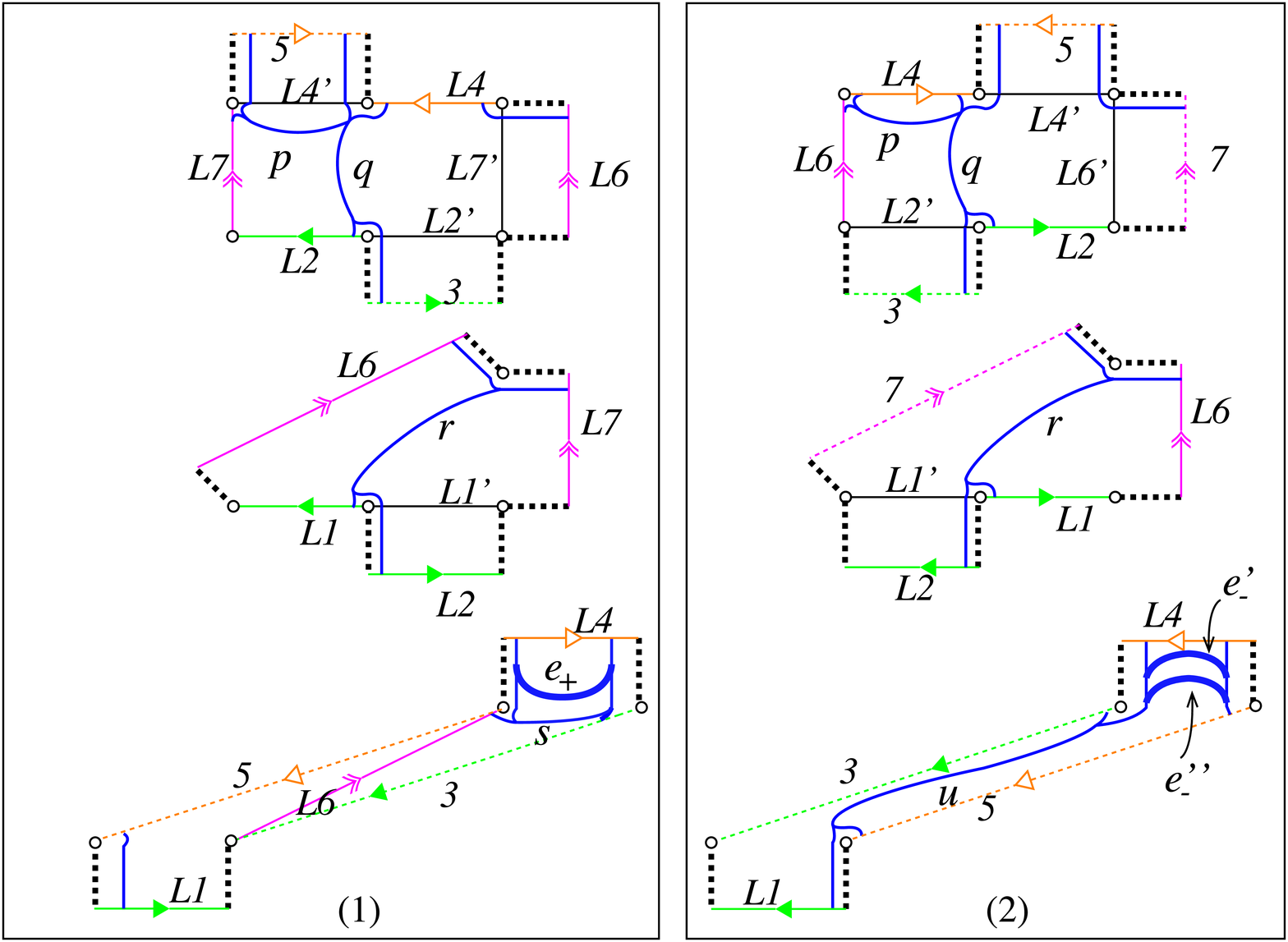}
\caption{
(1)  $ \mathcal{B}_+ \cap \mathcal{B}_{\Delta_+}$. 
(2) $ \mathcal{B}_- \cap \mathcal{B}_{\Delta_-}$. 
(cf. Figure~\ref{fig_identify}(1)(3).)}  
\label{fig_p_Train}
\end{figure}
\end{center}

\begin{rem}
The condition (5) above makes the difference between $\mathcal{B}_{\Delta_+}$ and $\mathcal{B}_{\Delta_-}$. 
The conditions (1)--(4) (without (5)) allow us to construct a branched surface 
which carries $\widehat{\mathcal{F}}$. 
The reason we require  (5) is that 
it is easy to extract a train track on   $F_{(i,j,k)_{\pm}}$ from the intersection 
$ F_{(i,j,k)_{\pm}} \cap \mathcal{B}_{\Delta_{\pm}}$ 
with the extra condition (5), 
see Lemma~\ref{lem_preTrainTrack}(2). 
 \end{rem}
 
 \begin{rem}
 \label{rem_TwiceOnce}
The following analysis is used in the proof of Lemma~\ref{lem_preTrainTrack}(2). 
 It happens twice that 
an edge of some element  in the subfamily $\{\tau_t^{-}\}_{1/3 \le t \le 1 - \epsilon}$ is passing through the segment $K$ through the isotopy, 
see Figure~\ref{fig_Bsurface_latter}(2)(3). 
On the other hand, the same thing happens once in the subfamily $\{\tau_t^{+}\}_{1/3 \le t \le 1 - \epsilon}$, 
see Figure~\ref{fig_Bsurface_latter}(1)(3).  
In the same figure, 
$4$-punctured disks containing the track tracks in question are colored, 
and edges of these train tracks in question are made thick. 
 \end{rem}

Since $F_{(i,j,k)_{\pm}}$ (up to isotopy) is transverse to the flow $\Phi_{\alpha+ \beta}^t$, 
we may assume that $F_{(i,j,k)_{\pm}}$ is transverse to $\mathcal{B}_{{\Delta}_{\pm}}$. 
We let 
\begin{equation}
\label{equation_pre}
\tau_{(i,j,k)_{\pm}}'= F_{(i,j,k)_{\pm}} \cap \mathcal{B}_{{\Delta}_{\pm}}.
\end{equation}

\begin{lem} 
\label{lem_preTrainTrack}
\ 

\begin{enumerate}
\item[(1)] 
The unstable foliation $\mathcal{F}_{(i,j,k)_{\pm}}$ of the pseudo-Anosov $\Phi_{(i,j,k)_{\pm}}$ is carried by $\tau_{(i,j,k)_{\pm}}'$. 

\item[(2)] 
Each component of $F_{(i,j,k)_{\pm}} \setminus \tau_{(i,j,k)_{\pm}}'$ is either a bigon (a disk with $2$ cusps) or 
a once punctured disk (i.e, annulus) with $k $ cusps for some $k \ge 1$.  
\end{enumerate}
\end{lem}

\begin{proof}
(1)  
By Theorem~\ref{thm_Fried2}, we have 
$\mathcal{F}_{(i,j,k)_{\pm}} = F_{(i,j,k)_{\pm}} \cap \widehat{\mathcal{F}}$. 
Moreover 
its suspension $\widehat{\mathcal{F}}_{(i,j,k)_{\pm}}$ by $\Phi_{(i,j,k)_{\pm}}$ 
is isotopic to  $\widehat{\mathcal{F}}$, see \cite[Corollary~3.2]{McMullen}. 
Since $\widehat{\mathcal{F}}$ is carried by the branched surface $\mathcal{B}_{\Delta_{\pm}}$, 
so is  $\widehat{\mathcal{F}}_{(i,j,k)_{\pm}}$. 
This implies that 
$\mathcal{F}_{(i,j,k)_{\pm}}$ is carried by $\tau_{(i,j,k)_{\pm}}'= F_{(i,j,k)_{\pm}} \cap \mathcal{B}_{{\Delta}_{\pm}}$. 
\medskip

\noindent
(2) 
We recall that 
the suitable branch equations on $\mathcal{B}_{\pm}$ induce $F_{(i,j,k)_{\pm}}$. 
Also recall that we need to insert some flowbands between the two segments with the same labeling 
to get the whole picture of $\mathcal{B}_{\pm}$. 
To view components of $F_{(i,j,k)_{\pm}} \setminus \tau_{(i,j,k)_{\pm}}'$, 
let us consider 
$ \mathcal{B}_{\pm} \cap \mathcal{B}_{\Delta_{\pm}}$  on $\mathcal{B}_{\pm}$ 
which is a $1$-branched manifold. 
We recall the definition of $\widehat{\mathcal{B}}_{\pm}$, 
see (\ref{equation_pre-Bpulus}) and (\ref{equation_pre-Bminus}). 
The constant $\delta$ is chosen to be $\tfrac{1}{3}$ 
and we have the train tracks 
$\tau_{1- \epsilon}^{+}= \tau_{1- \epsilon}^{-}= \tau_{\alpha+ \beta}$,  
$\tau_{2/3}^{+}= \tau_{2/3}^{-}$, 
$\tau_{1/3}^{+}= \tau_{1/3}^{-}$, see Figure~\ref{fig_Bsurface_latter}. 
Then the `patterns' in Figure~\ref{fig_p_Train}(1)(2) are obtained from 
$\mathcal{B}_{\pm} \cap \mathcal{B}_{\Delta_{\pm}}$ 
after folding or splitting all edges which appear on the flowbands. 
In fact the thick edge $e_+$  in Figure~\ref{fig_p_Train}(1) 
is the one 
(resp. thick edges $e_-'$ and $e_-''$  in Figure~\ref{fig_p_Train}(2) are the ones) by folding some edge 
(resp. by splitting some edges) 
which appear(s) on the flowband between the segments with the labeling $L4$, 
see also Remark~\ref{rem_TwiceOnce}. 
Then we reconstruct the fibers $F_{(i,j,k)_{\pm}}$ 
obtained from the suitable branch equations on $\mathcal{B}_{\pm}$ with the `patterns' in  Figure~\ref{fig_p_Train}(1)(2) .  
That we get is $\tau_{(i,j,k)_{\pm}}'$ (up to folding and splitting the edges) on $F_{(i,j,k)_{\pm}}$. 
Each edge of $\tau_{(i,j,k)_{\pm}}'$ originates in some edge of 
the branched $1$-manifold $\mathcal{B}_{\pm} \cap \mathcal{B}_{\Delta_{\pm}}$. 
We denote some of the edges of $\mathcal{B}_{+} \cap \mathcal{B}_{\Delta_{+}}$ 
(resp. $\mathcal{B}_{-} \cap \mathcal{B}_{\Delta_{-}}$) 
by $p$, $q$, $r$, $s$ (resp. $p$, $q$, $r$, $u$) 
as in Figure~\ref{fig_p_Train}(1) (resp. (2)).

We can fold 
all edges of $\tau_{(i,j,k)_{+}}'$ (resp. $\tau_{(i,j,k)_{-}}'$) 
which originate in $e_+$ (resp. $e_-'$ or $e_-''$)  into some edges. 
This means that these edges lie on the boundaries of some components of 
$F_{(i,j,k)_{\pm}} \setminus \tau_{(i,j,k)_{\pm}}'$ that are bigons. 
We fold all these edges as much as possible (i.e, collapse bigons), and 
we consider complementary regions of the resulting $1$-branched manifold. 
The combinatorics from Figures~\ref{fig_identify}(2)(4) and \ref{fig_p_Train} tell us that 
each component of the resulting $1$-branched manifold is a once punctured disk 
with $k$ cusps for some $k \ge 1$. 
\end{proof}

Let $\tau_{(i,j,k)_{\pm}}$ be the branched $1$-manifold obtained from $\tau_{(i,j,k)_{\pm}}'$ by collapsing all bigons of $\tau_{(i,j,k)_{\pm}}'$. 
By Lemma~\ref{lem_preTrainTrack}, we immediately have:

\begin{lem}
$\tau_{(i,j,k)_{\pm}}$ is  a train track on $F_{(i,j,k)_{\pm}}$ which carries $\mathcal{F}_{(i,j,k)_{\pm}}$.   
\end{lem}

\subsection{Monodromies $\Phi_{(i,j,k)_{\pm}}:  F_{(i,j,k)_{\pm}} \rightarrow F_{(i,j,k)_{\pm}}$ of the fibrations associated to $(i,j,k)_{\pm}$} 
\label{subsection_Monodromies}

First of all, we represent  fibers $F_{(i,j,k)_{\pm}}$ more simpler as follows. 
We shrink each flowband of $F_{(i,j,k)_{\pm}}$ as much as possible  along  flow lines 
into some edge.  
(Note that this operation does not change the topological type of  fibers.) 
Said differently, we simplify the branched surfaces $\mathcal{B}_{\pm}$ as follows. 
Shrink each flowband of the three pieces, see Figure~\ref{fig_identify}(1) (resp. (3)), 
as much as possible 
along  flow lines into some edge. 
Then the branch equations on $\mathcal{B}_{\pm}$ induces the one on 
such a simplified branched surface, 
from which one gets a surface in question which is homeomorphic to $F_{(i,j,k)_{\pm}}$. 

After shrinking flowbands, the resulting three pieces (the bottom, middle, top pieces) are: 
\begin{enumerate}
\item[(1)]
$\kappa_+$-patch: 
the two acute-angled triangles sharing vertices $\mathfrak{b}_1$ and $\mathfrak{b}_4$ 
\\
(resp. $\kappa_-$-patch: the parallelogram),

\item[(2)] 
$\beta$-patch: 
the right-angled triangle, and 

\item[(3)] 
$\alpha+ \beta$-patch: 
the rectangle.
\end{enumerate}
We call these pieces 
(1) {\it $\kappa_+$-patch} (resp. {\it $\kappa_-$-patch} ), 
(2) {\it $\beta$-patch} and (3) {\it $\alpha+ \beta$-patch}. 
See Figure~\ref{fig_map1_precise}(1) (resp. Figure~\ref{fig_map2_precise}(1)) for  three kinds of patches. 
We often draw  pairs of the two acute-angled triangles for the $\kappa_+$-patch separately 
as in Figure~\ref{fig_map1_precise}(1), 
but they should share the two vertices $\mathfrak{b}_1$ and $\mathfrak{b}_4$ 
(cf. Figure~\ref{fig_triangle}(3)).    
We can get the simplified fibers $F_{(i,j,k)_+}$ (resp. $F_{(i,j,k)_-}$) 
from $i$ parallel copies of $\kappa_+$-patches (resp. $\kappa_-$-patches), 
$j$ parallel copies of $\beta$-patches and $k$ parallel copies of $\alpha+ \beta$-patches 
under  suitable identifications of the boundaries of patches.

If we have $j$ parallel copies of the same kind of patches, say $\beta$-patches, 
${\Bbb P}_{\beta}^1, \cdots, {\Bbb P}_{\beta}^{j}$, 
we have ${\Bbb P}_{\beta}^{\ell} \subset F_{\alpha+ \beta}^{n_{\ell}}$ 
for some $0< n_{\ell}<1$, where $\ell \in \{1, \cdots, j\}$. 
For the notation of $F_{\alpha+ \beta}^{n_{\ell}}$, see the end of Section~\ref{subsection_Minimal}. 
If $0< n_1 < \cdots < n_{j-1} <n_{j}<1$, then 
we call the $\beta$-patch ${\Bbb P}_{\beta}^1$ {\it the bottom ($\beta$-)patch}, 
and call   ${\Bbb P}_{\beta}^{j}$ {\it the top ($\beta$-)patch}. 
Other $\beta$-patches are called  {the middle ($\beta$-)patches}. 
We define the top, middle, bottom for other patches similarly. 
We color each top of the three kinds of patches, 
see Figure~\ref{fig_map1_precise}(1)(left column), Figure~\ref{fig_map2_precise}(1)(left column). 
Here we label $E'$ for the top $\beta$-patch, and 
label $F'$ and $G'$ (resp. $F'$) for the top $\kappa_+$-patch (resp. $\kappa_-$-patch) in the same figure. 
We label $A'$, $B'$, $C'$ and $D'$ for  isosceles right-angled triangles which lies on the top $\alpha+ \beta$-patch.

The patches needed for building $F_{(i,j,k)_{\pm}}$ for the non-degenerate class $(i,j,k)_{\pm}$ 
are given  as in the same figure. 
We can think three kinds of patches are in the cylinder $\varSigma_{0,4} \times (0,1) \subset {\Bbb T}_{\phi_{\alpha+ \beta}}$. 
We have the flow direction in the cylinder from the `bottom' $\varSigma_{0,4} \times \{0\}$ to the `top' $\varSigma_{0,4} \times \{1\}$. 
The types of arrows (red, green, yellow) in the figure  are compatible with the ones illustrated 
in Figures~\ref{fig_normOne} and \ref{fig_identify}. 
Among the  patches with the same kind ($\kappa_{\pm}$-patches, $\beta$-patches or $\alpha+ \beta$-patches), 
the way to label  parallel segments with the same kind of arrow  
is that the number for the labeling   increases (cyclically) along the flow direction. 
We often omit to label segments which lie on the middle patches. 
To get the fiber $F_{(i,j,k)_{\pm}}$, we identify the two segments with the same kind of arrow and with the same labeling (same number) 
by using the flow $\Phi_{\alpha+ \beta}^t$.   
In the right  column of (1) in  the same figure, 
the labeling of  segments on  patches are the same as the one given in the left column.

Let us turn to construct   $\Phi_{(i,j,k)_{\pm}}: F_{(i,j,k)_{\pm}} \rightarrow F_{(i,j,k)_{\pm}} $ explicitly. 
It is enough to describe where each patch maps to. 
Since the desired monodromy $\Phi_{(i,j,k)_{\pm}}$ is the first return map on $F_{(i,j,k)_{\pm}}$ with respect to $\Phi_{\alpha+ \beta}^t$, 
we see the followings. 
All patches but the top of  each kind of patches 
map to the next above patch (of the same kind) along the flow direction. 
Thus the monodromy $\Phi_{(i,j,k)_{\pm}}$ restricted to these patches is just a shift map. 
On the other hand, 
each top patches  map to some bottom patches (possibly with different kinds),  
see Figure~\ref{fig_map1_precise}(1) for  $\Phi_{(i,j,k)_{+}}$ and 
see Figure~\ref{fig_map2_precise}(1) for $\Phi_{(i,j,k)_{-}}$, 
where $A,B, \cdots$ are the images of $A',B', \cdots $ under $\Phi_{(i,j,k)_{\pm}}$. 
More precisely, 
we can get the image of the top  $\beta$-patch under $\Phi_{(i,j,k)_{\pm}}$ 
when we push the fiber $F_{(i,j,k)_{\pm}}$ along the flow direction and see how this top patch hits to the 
bottom $\alpha+ \beta$-patch. 
Similarly, one can get the image of the top $\kappa_+$-patch (resp. top $\kappa_-$-patch) 
under $\Phi_{(i,j,k)_+}$ (resp. $\Phi_{(i,j,k)_-}$) 
if we see how this top patch hit to the both  bottom $\beta$-patch and bottom $\alpha+ \beta$-patch. 
To get the images of the isosceles right-angled triangles $A'$, $B'$, $C'$ and $D'$ which lies on the top $\alpha+ \beta$-patch, 
we first consider the acute-angled triangles $A$, $B$, $C$ and $D$ which lie on $\varSigma_{0,4} \times \{0\}$ 
as in Figure~\ref{fig_pApillow}(right). 
Then investigate how these acute-angled triangles hit  bottoms  patches, 
when we push them  along the flow direction.

The monodromies of the fibrations associated to the degenerated classes can be constructed similarly. 
As an example, we deal with the degenerated classes 
$a= (j,k)_0$'s, see Figure~\ref{fig_map3_precise}. 

\begin{rem}
\label{rem_delta0}
Suppose that $ (j,k)_0$ is primitive (i.e, $\gcd(j,k)=1$). 
Then the fiber  $F_{(j,k)_0}$ is connected, and it  has genus $0$, see \cite{KT}. 
Many pseudo-Anosovs with small dilatations defined on the surfaces of genus $0$ 
are contained in the family of fibered classes  $ (j,k)_0$'s, 
see Examples~\ref{ex_braid1}, \ref{ex_braid2} and \cite[Section~4.1]{KT}. 
By using the Artin generators of the braid groups, 
the words which represent  $\phi_{(j,k)_0}$'s are given in \cite[Theorem~3.4]{KT}. 
They are quite simple words. 
\end{rem}

\begin{center}
\begin{figure}
\includegraphics[width=5.5in]{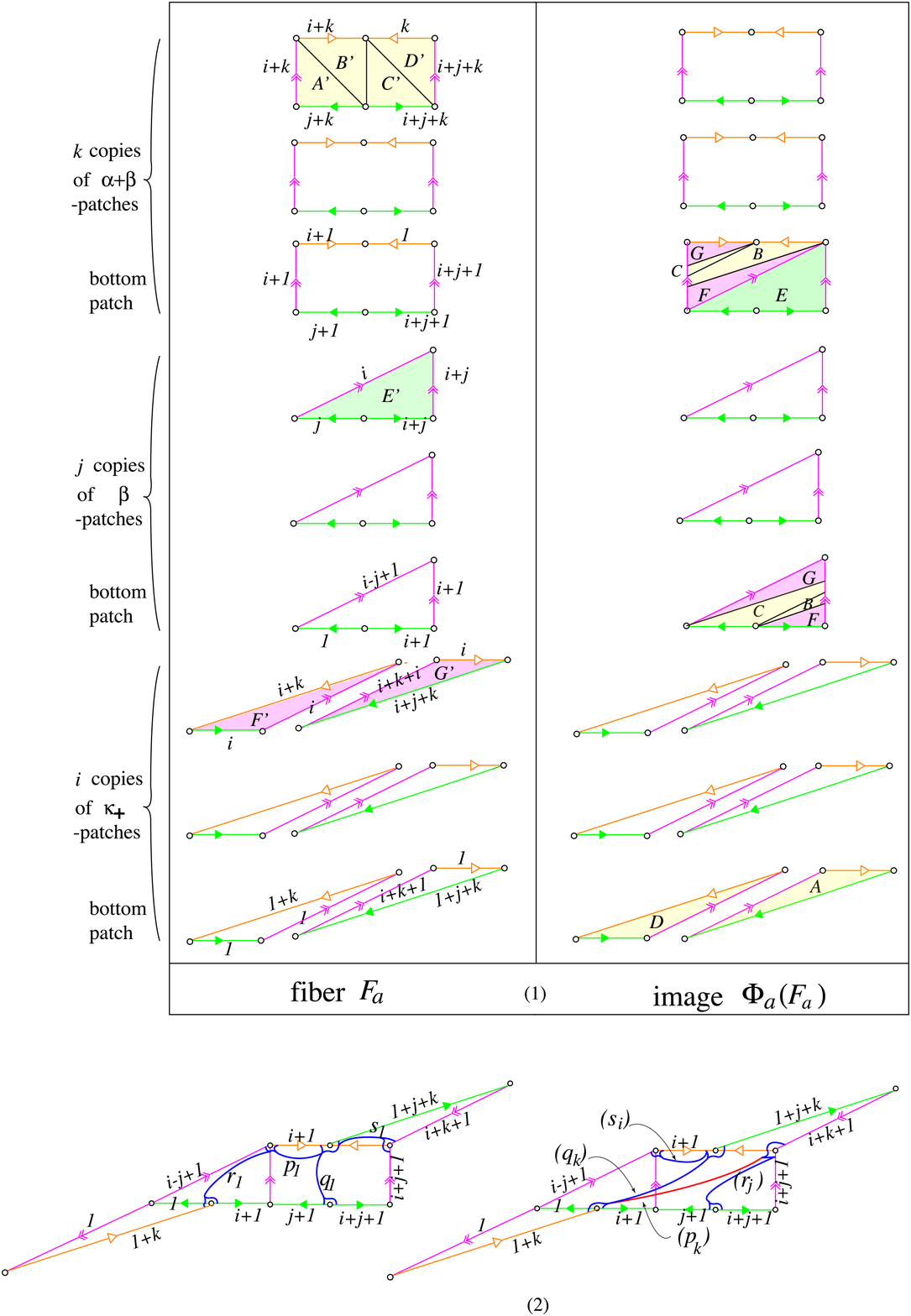}
\caption{
Labeling of segments in the right column of (1) 
is the same as the one in the left column. 
(1) 
$\Phi_a: F_a \rightarrow F_a$ for non-degenerate $a=(i,j,k)_+$. 
(2) 
(left) Bottom  edges  $s_1$, $r_1$, $p_1$ and $q_1$ 
which lie on  three bottom patches. 
(right) Images of   top edges $s_i$, $r_j$, $p_k$ and $q_k$ 
under $\Phi_{a}$ which lie on  three bottom patches. 
We denote by $(s_i)$ etc., the image $\Phi_{a}(s_i)$ etc.} 
\label{fig_map1_precise}
\end{figure}
\end{center}

\begin{center}
\begin{figure}
\includegraphics[width=5in]{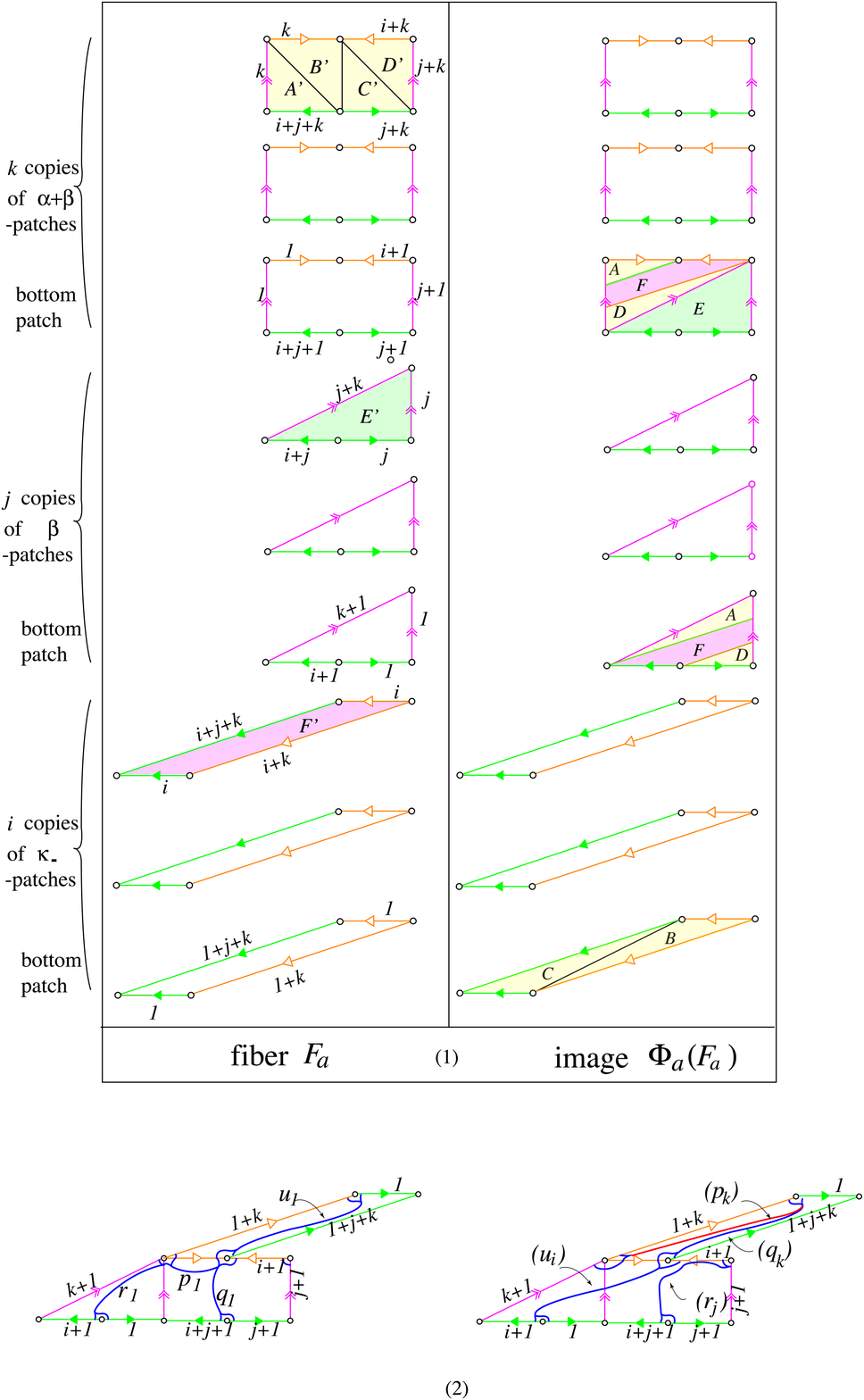}
\caption{(1) $\Phi_a: F_a \rightarrow F_a$ for non-degenerate $a=(i,j,k)_-$.  
(2)  
(left) Bottom  edges  $u_1$, $r_1$, $p_1$ and $q_1$ 
which lie on  three bottom patches. 
(right) Images of top edges $u_i$, $r_j$, $p_k$ and $q_k$ 
under $\Phi_{a}$  which lie  on  three bottom patches.} 
\label{fig_map2_precise}
\end{figure}
\end{center}

\begin{center}
\begin{figure}
\includegraphics[width=3.5in]{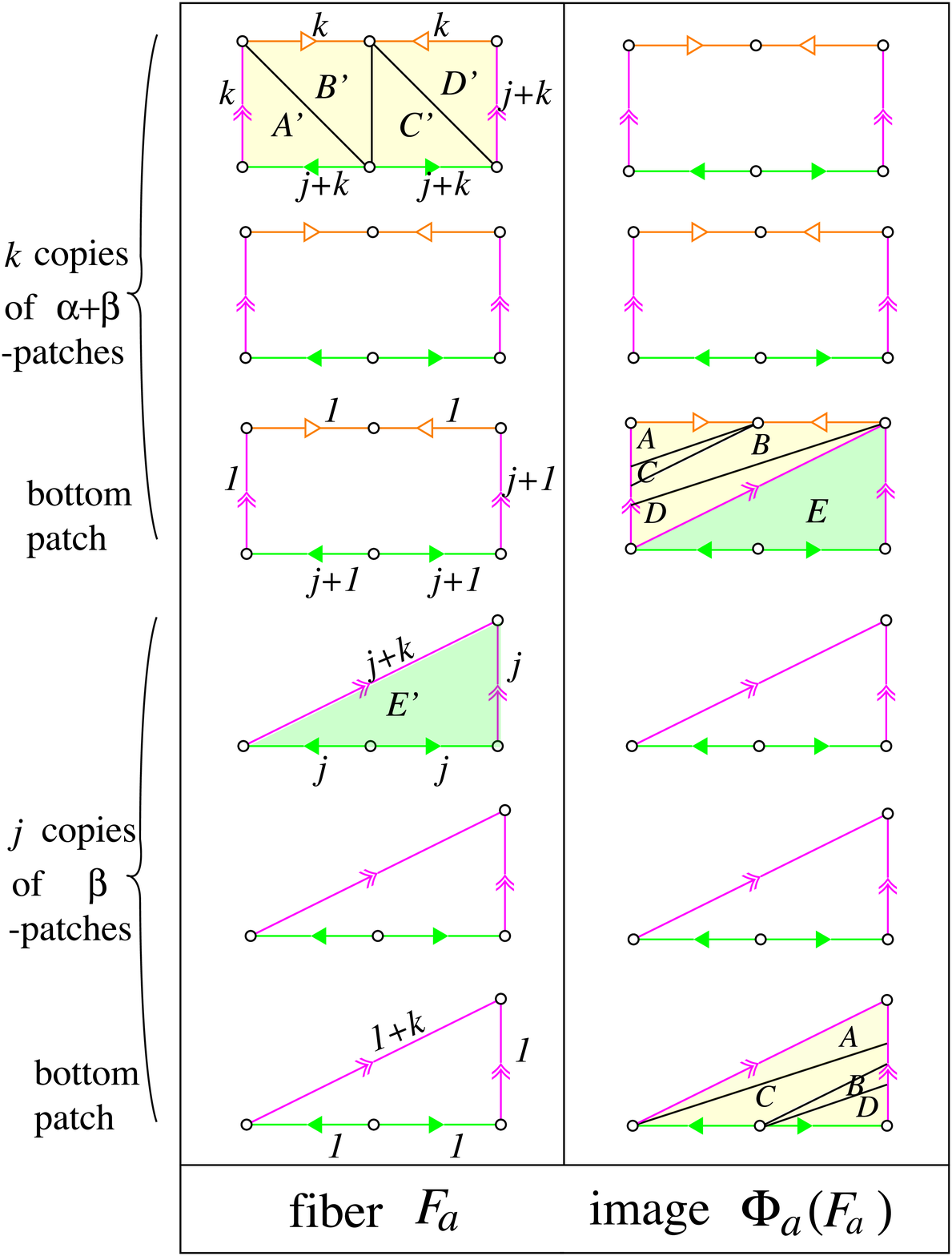}
\caption{$\Phi_a: F_a \rightarrow F_a$ for degenerate  $a =(j,k)_0$.} 
\label{fig_map3_precise}
\end{figure}
\end{center}

\begin{center}
\begin{figure}
\includegraphics[width=6in]{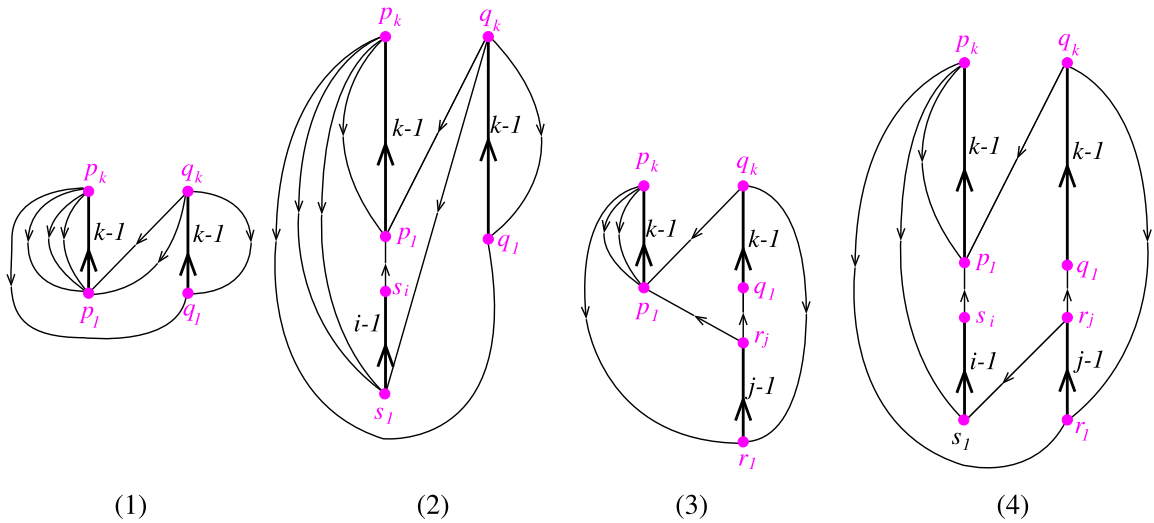}
\caption{
$\Gamma_a$ for $a = (i,j,k)_+ $. 
(1)(2)(3) degenerate cases. (4) non-degenerate case. 
(1) $a= (0,k)_0$ with $k>0$. 
(2) $a= (i,0,k)_+$ with $i ,k>0$. 
(3) $a= (j,k)_0$ with $j, k>0$. 
(4) $a= (i,j,k)_+$ with $i,j,k >0$.} 
\label{fig_digraph_p}
\end{figure}
\end{center}

\begin{center}
\begin{figure}
\includegraphics[width=5.8in]{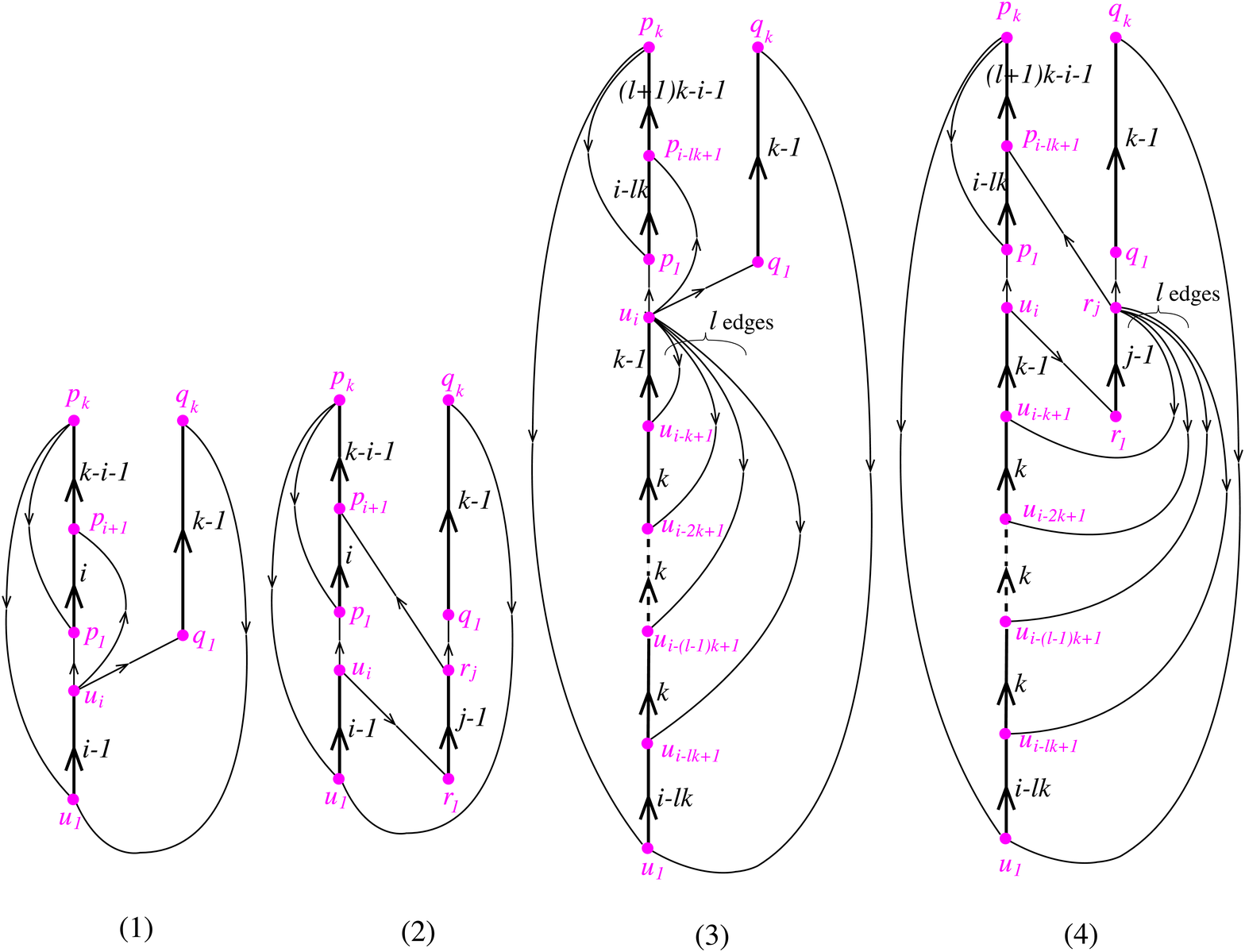}
\caption{
$\Gamma_a$ for $a= (i,j,k)_-$. 
(1)(3) degenerate cases. (2)(4) non-generate cases. 
(1) $a= (i,0,k)_-$ with $0< i < k$. 
(2) $a=(i,j,k)_-$ with $0< i < k$. 
(3) $a= (i,0,k)_-$ with $i  \ge k>0$. 
(4) $a= (i,j,k)_-$ with $i \ge k>0$. 
For (3)(4), $\ell \ge 1$ is an integer such that 
$0 \le i-k\ell \le k-1$.} 
\label{fig_digraph_n}
\end{figure}
\end{center}

\subsection{Train track representatives 
$\mathfrak{p}_{(i,j,k)_{\pm}}: \tau_{(i,j,k)_{\pm}} \rightarrow \tau_{(i,j,k)_{\pm}}$ of $\phi_{(i,j,k)_{\pm}}$}  
\label{subsection_TTmap}

In the following lemma, we use 
the metrized, directed graph  $\Gamma_{(i,j,k)_{\pm}}$ given in Figures~\ref{fig_digraph_p} and \ref{fig_digraph_n}, 
where each edge with no labeling 
means that its length is equal to $1$, 
and all edges with labeling are made thick in the figures.

\begin{lem}
\label{lem_invariant}
\ 

\begin{enumerate}
\item[(1)] 
The train track $\tau_{(i,j,k)_{\pm}}$ is invariant under $\phi_{(i,j,k)_{\pm}} = [\Phi_{(i,j,k)_{\pm}}]$. 
If we let $\mathfrak{p}_{(i,j,k)_{\pm}}: \tau_{(i,j,k)_{\pm}} \rightarrow \tau_{(i,j,k)_{\pm}}$ 
be the train track representative of $\phi_{(i,j,k)_{\pm}}$, then 
its incidence matrix is Perron-Frobenius. 

\item[(2)] 
The directed graph  $\Gamma_{(i,j,k)_{\pm}}$ 
is the one induced by $\mathfrak{p}_{(i,j,k)_{\pm}}: \tau_{(i,j,k)_{\pm}} \rightarrow \tau_{(i,j,k)_{\pm}}$, 
\end{enumerate}
\end{lem}

Once  we prove that $\tau_{(i,j,k)_{\pm}}$ is invariant under $\phi_{(i,j,k)_{\pm}}$, 
the claim that the incidence matrix of $\phi_{(i,j,k)_{\pm}}$ is Perron-Frobenius follows from Theorem~\ref{thm_PaPe}.

\begin{proof}[Proof of Lemma~\ref{lem_invariant}] 
By construction of $\tau_{(i,j,k)_{\pm}}$, 
each edge of $\tau_{(i,j,k)_{\pm}}$ originates in some edge of the intersection $\mathcal{B}_{\pm} \cap \mathcal{B}_{\Delta_{\pm}}$ 
(see Figure~\ref{fig_p_Train}). 
We only  label 
edges of $\tau_{(i,j,k)_+}$ (resp. $\tau_{(i,j,k)_-}$) which originate in the  edges $p,q,r$ and $s$ 
(resp. $p,q,r$ and $u$), 
since it turns out that the images of other edges under $\Phi_{(i,j,k)_{\pm}}$ are eventually periodic up to isotopy. 
We first explain the way to label the edges of $\tau_{(i,j,k)_{\pm}}$, 
which is similar to the one to label segments of patches given in Section~\ref{subsection_Monodromies}. 
We label parallel  edges of $\tau_{(i,j,k)_{\pm}}$ 
(which originate in the same edge of $\mathcal{B}_{\pm} \cap \mathcal{B}_{\Delta_{\pm}}$) so that  
the number for the labeling  increases along the flow direction. 
For example, see Figures~\ref{fig_ex_LT}--\ref{fig_ex_Tsai}.

We only prove the claims (1) and (2) for  non-degenerate classes $(i,j,k)_{\pm}$. 
(The proofs for degenerate classes are similar.)  
Let us describe the image of an edges $e$ with labeling under  $\Phi_{(i,j,k)_{\pm}}$. 
Such an edge $e$ lies on some patch  for building $F_{(i,j,k)_{\pm}}$. 
Suppose that  $e$ lies on a patch, say ${\Bbb P}_e$ which is not a top patch. 
Then $e$ maps  (under $\Phi_{(i,j,k)_{\pm}}$) to 
the next above  edge $e'$ which lies on the same kind of patch ${\Bbb P}_{e'}$ as ${\Bbb P}_{e}$. 
Clearly, $e$ and $e'$ originate in the same edge of $\mathcal{B}_{\pm} \cap \mathcal{B}_{\Delta_{\pm}}$. 
Suppose that an edge $e_{\mathrm{top}}$ of $\tau_{(i,j,k)_{\pm}}$ lies on some top patch. 
We call  $e_{\mathrm{top}}$ the {\it top edge}. 
If an edge $e_{\mathrm{bot}}$ of $\tau_{(i,j,k)_{\pm}}$ lies on some bottom patch, 
then we call $e_{\mathrm{bot}}$ the {\it bottom edge}.

We first consider the non-degenerate class $(i,j,k)_+$. 
Let us consider  the images $\Phi_{(i,j,k)_{+}}(e_{\mathrm{top}})$'s for all top edges $e_{\mathrm{top}}$'s.  
Then we can put  $\Phi_{(i,j,k)_{+}}(e_{\mathrm{top}})$'s  
in the tie neighborhood of $\tau_{(i,j,k)_{+}}$ which are transverse to the ties up to isotopy, 
where the support of the isotopy can be taken on the neighborhood of the three bottom patches, 
see  Figure~\ref{fig_map1_precise}(2). 
Thus $\tau_{(i,j,k)_{+}}$ is invariant under $\phi_{(i,j,k)_{+}}$. 
 In fact we can get the edge path 
  $\mathfrak{p}_{(i,j,k)_+}(e_{\mathrm{top}})$ 
from Figure~\ref{fig_map1_precise}(2): 
The left of (2) of the same figure illustrates 
the union of all bottom edges of $\tau_{(i,j,k)_+}$  which lie on the union of all bottom patches. 
The right of (2) of the same figure shows the image of all top edges under $\Phi_{(i,j,k)_+}$. 
One can find from the right in (2) of the same figure that 
$\mathfrak{p}_{(i,j,k)_+}(p_k)$ and  $\mathfrak{p}_{(i,j,k)_+}(q_k)$ 
pass through the three edges $r_1$, $p_1$,  $s_1$, and  the two edges $r_1$, $p_1$ respectively. 
The edge path $\mathfrak{p}_{(i,j,k)_+}(r_j)$ passes through  the two edges $q_1$, $s_1$.

Note that by Remark~\ref{rem_real-edge}, we see that 
edges of $\tau_{(i,j,k)_+}$ which originate in the  edges $p,q,r$ and $s$ 
are real edges for $\mathfrak{p}_{(i,j,k)_+}$. 
Others are infinitesimal edges. 
We explain a structure of the directed graph in Figure~\ref{fig_digraph_p}(4).  
When $e$ is an edge with some label 
(which are made thick in the figure), then the end points of $e$ are the vertices 
having the same origin ($p$, $q$, $r$ or $s$) on $\mathcal{B}_+$. 
Suppose that  
$e$ is an edge whose length equals $k-1$, and suppose that 
 $e$ has end points $p_1$ and $p_k$ having the same origin $p$. 
Then the edge $e$ with labeling $k-1$ 
corresponds to the following edge path with length $k-1$: 
$$p_1 \rightarrow p_2 \rightarrow \cdots \rightarrow p_{k-1} \rightarrow p_k.$$ 
In particular all vertices between $p_1$ and $p_k$ have the same origin $p$ on $\mathcal{B}_+$. 
Putting these things together, we can check that 
 metrized, directed graph $\Gamma_{(i,j,k)_+}$   in Figure~\ref{fig_digraph_p}(4) 
 is the one induced by 
$\mathfrak{p}_{(i,j,k)_+}: \tau_{(i,j,k)_{+}} \rightarrow \tau_{(i,j,k)_{+}}$.

We turn to the non-degenerate class $(i,j,k)_-$. 
In the same manner as in the class $(i,j,k)_+$, 
we can see that 
$\tau_{(i,j,k)_-}$ is invariant under $\phi_{(i,j,k)_-}$. 
The edges of $\tau_{(i,j,k)_-}$ which originate in the  edges $p,q,r$ and $u$ 
are real edges for the train track representative $\mathfrak{p}_{(i,j,k)_-}$ of $\phi_{(i,j,k)_-}$. 
Others are infinitesimal edges. 
A hint to obtain $\mathfrak{p}_{(i,j,k)_-}(e_{\mathrm{top}})$ 
is given in Figure~\ref{fig_map2_precise}(2). 
The left of (2) in the same figure illustrates 
the union of all bottom edges of $\tau_{(i,j,k)_-}$  which lie on the union of all bottom patches. 
The right of (2) in the same figure shows the image of all top edges under $\Phi_{(i,j,k)_-}$. 
The images under $\mathfrak{p}_{(i,j,k)_-}$ of  all top edges  but  $r_j$ 
are edge paths written by bottom edges. 
For example, $\mathfrak{p}_{(i,j,k)_-}(u_i)$ is an edge path which passes through $r_1$ and $p_1$. 
On the other hand, 
to put  $\Phi_{(i,j,k)_-}(r_j)$ in the tie neighborhood of $\tau_{(i,j,k)_{-}}$ 
one needs to make $\Phi_{(i,j,k)_-}(r_j)$ (up to isotopy) across the segment $K$. 
An analysis  to identify boundaries of patches for building $F_{(i,j,k)_-}$  
enables us to get the image $\mathfrak{p}_{(i,j,k)_-}(r_j)$. 
We can verify that 
the directed graph $\Gamma_{(i,j,k)_-}$  given  in Figure~\ref{fig_digraph_n}(2) (resp. (4)) is the one induced by 
$\mathfrak{p}_{(i,j,k)_-}$ if $0<i<k$ (resp. $i \ge k>0$).  
The edge path $\mathfrak{p}_{(i,j,k)_-}(r_j)$ has length $\ell+2$, where 
$\ell \ge 0$ is an integer such that 
$0 \le i-k\ell \le k-1$.  
\end{proof}

\begin{rem}
Metrized and directed graphs  $\Gamma_{(i,j,k)_{\pm}}$ for non-degenerate  classes 
can recover ones for  degenerate classes. 
See Figures~\ref{fig_p_Train}(1)(2)(3) and \ref{fig_digraph_p}(1)(3). 
To see this, consider $\Gamma_{(i,j,k)_+}$ for the non-degenerate class   $(i,j,k)_+$. 
We have the edge $p_k \to s_1$ and the edge path 
from $s_1$ to $p_1$ for $\Gamma_{(i,j,k)_+}$, see  Figure~\ref{fig_digraph_p}(4). 
Connecting these edge paths, we have the edge path  from $p_k$ to $p_1$ with length $>1$, see Figure~\ref{fig_digraph_p}(4). 
This determines {\it the edge} $p_k \to p_1$ with length $1$ in the degenerate class  $(0,j,k)_+= (j,k)_0$ 
by `eliminating'  vertices $s_1, \cdots, s_i$, see Figure~\ref{fig_digraph_p}(3). 
Another example is this. 
We have the edge $r_j \to s_1$ and the edge path from $s_1$ to $p_1$ for the non-degenerate class $(i,j,k)_+$, 
see Figure~\ref{fig_digraph_p}(4). 
They determine {\it the edge}  $r_j \to p_1$ with length $1$ for the degenerate class $(0,j,k)_+= (j,k)_0$, 
see Figure~\ref{fig_digraph_p}(3).
\end{rem}

\subsection{Curve complexes $G_{(i,j,k)_{\pm}}$ and their clique polynomials $Q_{(i,j,k)_{\pm}}(t)$} 
\label{subsection_Clique}

We define some graphs. 
Let $K_{m,n}$ denote the complete bipartite graph with $m$ and $n$ vertices. 
We denote by $K_{m,n}^{**}$, the disjoint union of $K_{m,n}$ and the graph with two vertices and with no edges.

The metrized, directed graphs $\Gamma_{(i,j,k)_+}$'s and $\Gamma_{(i,j,k)_-}$'s  
were given in Section~\ref{subsection_TTmap}. 
Here we exhibit their curve complexes $G_{(i,j,k)_+}$'s in Figure~\ref{fig_wgraph_p} 
and $G_{(i,j,k)_-}$'s  in Figure~\ref{fig_wgraph_n}.  
In the next proposition,  
we give the clique polynomial $Q_{(i,j,k)_{\pm}}$ for the computation of $\lambda_{(i,j,k)_{\pm}}$. 
We find that it is equal to the polynomial $f_{(i,j,k)_{\pm}}(t)$ in Lemma~\ref{lem_summary2}(1).

\begin{prop}
\label{prop_clique}
The clique polynomial $Q_{(i,j,k)_{\pm}}(t)$ of  the curve complex $G_{(i,j,k)_{\pm}}$ 
is the following  reciprocal polynomial 
$$Q_{(i,j,k)_{\pm}}(t)= 1 - (t^k+ t^{i+k}+ t^{j+k}+ t^{i+j+k})+ t^{i+j+2k}.$$ 
In particular, 
the dilatation $\lambda_{(i,j,k)_+} = \lambda_{(i,j,k)_-} $ is the largest root of 
$Q_{(i,j,k)_+}(t)= Q_{(i,j,k)_-}(t)$. 
\end{prop}

\begin{proof}
It is straightforward to compute the clique polynomial $Q_{(i,j,k)_{\pm}}(t)$. 
The dilatation $\lambda_{(i,j,k)_{\pm}}$ equals the growth rate of $\Gamma_{(i,j,k)_{\pm}}$. 
By Theorem~\ref{thm_McMullen}, 
$\tfrac{1}{\lambda_{(i,j,k)_{\pm}}}$ is the smallest root of  $Q_{(i,j,k)_{\pm}}(t)$. 
Note that $Q_{(i,j,k)_{\pm}}(t)$ is a reciprocal polynomial, 
i.e, 
$Q_{(i,j,k)_{\pm}}(t)= t^{i+j+2k} Q_{(i,j,k)_{\pm}}(t^{-1})$.  
Thus 
the largest root of $Q_{(i,j,k)_{\pm}}(t)$ equals $\lambda_{(i,j,k)_{\pm}}$. 
\end{proof}

\begin{center}
\begin{figure}
\includegraphics[width=4in]{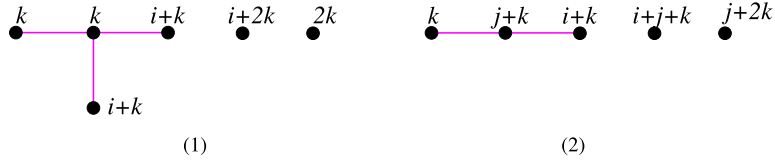}
\caption{$G_a$ for $a = (i,j,k)_+$. 
(1) $G_a= K_{1,3}^{**}$ if 
$a= (i,0,k)_+$ including  the case $i=0$ and $k>0$, see Figure~\ref{fig_digraph_p}(1)(2). 
(2)  $G_a= K_{1,2}^{**}$  if 
$a= (i,j,k)_+$ including the case $i=0$, $j>0$ and $k>0$, see Figure~\ref{fig_digraph_p}(3)(4).} 
\label{fig_wgraph_p}
\end{figure}
\end{center}

\begin{center}
\begin{figure}
\includegraphics[width=4in]{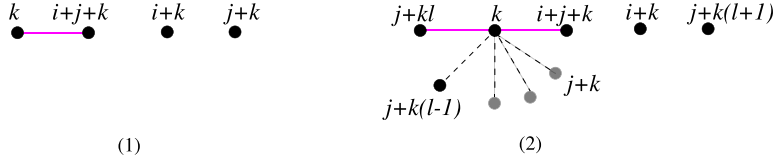}
\caption{$G_a$ for $a= (i,j,k)_- $. 
(1) $G_a= K_{1,1}^{**}$ 
if $a= (i,j,k)_-$ with 
$0< i < k$ including the case $j=0$,
 see Figure~\ref{fig_digraph_n}(1)(2).
(2) $G_a= K_{1, \ell+1}^{**}$ 
if $a= (i,j,k)_-$ with 
$i \ge k>0$ including the case $j=0$, 
 see Figure~\ref{fig_digraph_n}(3)(4). 
Here $\ell \ge 1$ is an integer such that 
$0 \le i-k\ell \le k-1$. 
The weights of $\ell+1$ vertices of the subgraph $K_{1, \ell+1}$ surrounding the centered vertex 
are given by $i+j+k$, $j+k, \cdots, j+ k(\ell-1), j+k\ell$ clockwise.}  
\label{fig_wgraph_n}
\end{figure}
\end{center}

\begin{proof}[Proof of Theorem~\ref{thm_algorithm}]
Let $a$ be a fibered class of $N$. 
By  a symmetry of Thurston norm ball $U_N$ and 
by Lemma~\ref{lem_summary}(4), 
we may suppose that $a$ is of the form $(i,j,k)_{\pm} \in int(C_{\Delta})$. 
Suppose that $a= (0,0,1)_{\pm}= (0,1)_0$. 
In this case, 
we have constructed  $\Phi_a: F_a \rightarrow F_a$ and $\mathfrak{p}_a: \tau_a \rightarrow \tau_a$ explicitly 
in Section~\ref{subsection_seed}.

Let us consider other fibered classes $a= (i,j,k)_{\pm}$'s. 
The explicit construction of the monodromy $\Phi_a: F_a \rightarrow F_a$ of the fibration associated to $a$ 
is given in Section~\ref{subsection_Monodromies}. 
If $a$ is primitive, then 
the explicit construction of the desired train track representative 
$\mathfrak{p}_{a}: \tau_a \rightarrow \tau_a$ of $\phi_a$ together with 
the induced directed graph $\Gamma_a$ 
of $\mathfrak{p}_{a}$ (Figures~\ref{fig_digraph_p} and \ref{fig_digraph_n}) 
is given in Section~\ref{subsection_TTmap}. 
\end{proof}

\section{A catalogue of small dilatation pseudo-Anosovs} 
\label{section_catalogue}

\subsection{Fibered classes of Dehn fillings $N(r)$}

Recall that $T_{\beta}$ (resp. $T_{\alpha}$, $T_{\gamma}$) 
is a torus which is the boundary of a regular neighborhood of the component 
$K_2$ (resp. $K_1$, $K_3$) of the $3$ chain link $\mathcal{C}_3$. 
Recall that  $N(r)$ is the manifold obtained from $N$ by Dehn filling the cusp 
along the slope $r \in {\Bbb Q} \cup \{\infty\} $. 
It is known by  \cite{MP}  that 
$N(r)$  is hyperbolic unless 
$r \in \{\infty, -3, -2, -1, 0\}$.  
The manifolds $N(\tfrac{3}{-2})$ and $N(\tfrac{1}{-2})$ are 
the exterior of the Whitehead sister link  (i.e, $(-2,3,8)$-pretzel link) and 
the exterior of the $3$-braided link $\mathrm{br}(\sigma_1^{-1} \sigma_2)$ 
(or $6_2^2$ link in Rolfsen's table) respectively. 
Also $N(1)$ is the exterior of the Whitehead link 
(see \cite{MP} for example). 
These Dehn fillings  $N(\tfrac{3}{-2})$, $N(\tfrac{1}{-2})$ and $N(1) $ 
play an important rule to study pseudo-Anosovs with small dilatations (see \cite{KKT2}), 
which we recall quickly in Section~\ref{subsection_Infinite}.

First of all, 
we consider the relation of  fibered classes between $N$ and Dehn fillings $N(r)$.  
We suppose that 
 $N(r)$ is the manifold obtained from $N$ by Dehn filling the cusp 
specified by $T_{\beta}$ along the slope $r $. 
By using Lemma~\ref{lem_summary}(2) (which says that the boundary slope of $\partial_{\beta} F_a$ equals $\tfrac{z+x}{-y}$), 
we have the following: 
There exists a natural injection  
$$\iota_{\beta}: H_2(N(r), \partial N(r)) \rightarrow H_2(N, \partial N)$$ 
whose image equals $S_{\beta}(r)$, 
where 
\begin{equation*} 
	S_{\beta}(r)= \{(x,y,z) \in H_2(N, \partial N)\ |\ -ry= z+x\},   
\end{equation*} 
see \cite[Proposition~2.11]{KKT2}. 
We choose  $r \in {\Bbb Q} $ with $r \not\in \{-3, -2, -1, 0\}$, 
and let  $a \in S_{\beta}(r) = \mathrm{Im} \, \iota_{\beta}$  be a fibered class in  $H_2(N, \partial N)$.  
Then 
$\overline{a} = \iota_{\beta}^{-1}(a) \in H_2(N(r), \partial N(r))$  is also 
a fibered class of  $N(r)$. 
This is because 
the monodromy $\Phi_a: F_a \rightarrow F_a$ of the fibration on $N$ associated to $a$ 
extends to the monodromy $\overline{\Phi}_a$ of the fibration on $N(r)$ associated to $\overline{a}$ 
by capping each boundary component of $\partial_{\beta} F_a$ with the disk. 
If the (un)stable foliation $\mathcal{F}_a$ of $\Phi_a$ has a property such that 
each component of $\partial_{\beta} F_a$ has no $1$ prong, then 
$\mathcal{F}_a$  extends canonically 
to the (un)stable foliation  $\overline{\mathcal{F}}_a$  of    $\overline{\Phi}_a$, 
and hence $\overline{\Phi}_a$ becomes a pseudo-Anosov homeomorphism with the same dilatation as $\Phi_a$. 
We sometimes denote $N(r)$ by $N_{\beta}(r)$ 
when we need to specify the cusp which is filled. 
By using this notation, we may write 
$\overline{a} \in H_2(N_{\beta}(r), \partial N_{\beta}(r))$.

Similarly, 
when  $N(r)$  is the manifold obtained from $N$ by Dehn filling the cusp specified 
by  another torus $T_{\gamma}$  along the slope  $r$, 
we have a natural injection, 
\begin{align*}
	\iota_{\gamma} & : H_2(N(r), \partial N(r)) \to H_2(N, \partial N) 
\end{align*}
whose image $ \mathrm{Im} \, \iota_{\gamma}$ is given by 
\begin{align*}
	S_{\gamma}(r) & = \{ (x, y, z) \in H_2(N, \partial N) \, | \, -rz = x+y \}.  
\end{align*}

\subsection{Concrete examples}
\label{subsection_examples}

In the following example,  we consider some fibered classes $a$ and compute their dilatations $\lambda(a)$. 
We also exhibit 
$\tau_a \subset F_a$, its image $\Phi_a(\tau_a)$ putting into the tie neighborhood 
$\mathcal{N}(\tau_a)$ (so that it is transverse to the ties) 
and the directed graph $(\Gamma_a, {\bf 1})$.

\begin{ex}
\label{example_Ex}
\begin{enumerate}

\item[(1)]  
If  $a= (1,4,1)_+= (2,6,1)$, then $\lambda(a)  \approx  1.7220$ 
is the largest root of 
$$Q_{(1,4,1)_+}(t)= f_{(2,6,1)}(t)= (t^3+1)(t^4 - t^3 - t^2-t +1).$$  
By Lemma~\ref{lem_summary}(3),  the topological type of $F_a$ is  $\varSigma_{2,5}$. 
By Lemma~\ref{lem_summary}(7),  $\mathcal{F}_a$ is orientable. 
This  implies that 
the monodromy 
$\Phi_a: F_a \rightarrow F_a$ of the fibration  associated to $a$ 
extends to the pseudo-Anosov $\widehat{\Phi}_a: \widehat{F}_a \rightarrow \widehat{F}_a$ 
on the closed surface of genus $2$ with orientable invariant foliations 
by capping  each boundary component of $ F_a$ with the disk. 
The dilatation of $\widehat{\Phi}_a$ is same as that of $\Phi_a$, that is $\lambda(a)$. 
The minimal dilatations $\delta_2$ and $\delta_2^+$ are computed in \cite{CH} and \cite{Zhirov} respectively, 
and $\delta_2= \delta_2^+$ holds. 
In fact, $\delta_2$ is the largest root of $t^4 - t^3 - t^2-t +1$, 
and hence we have $\lambda(a)= \delta_2= \delta_2^+ $. 
Thus  $\widehat{\Phi}_a$ is a minimizer of $\delta_2= \delta_2^+$. 
See Figure~\ref{fig_6_2_1}.

\item[(2)] 
If $a= (3,1,1)_-= (1,2,-3)$,  then  $\lambda(a) \approx 1.7816$ 
is the largest root of 
$$Q_{(3,1,1)_-}(t)= f_{(1,2,-3)}(t)= t^6-t^5-t^4-t^2-t+1.$$
See Figure~\ref{fig_1_2_-3}.

\item[(3)] 
If $a= (2,3)_0= (3,5,0)$, then $\lambda(a) \approx 1.4134$ 
is the largest root of 
$$ Q_{(2,3)_0}(t)= f_{(3,5,0)}(t)= t^8- 2t^5-2t^3+1.$$
The fiber $F_a$ is homeomorphic to $\varSigma_{0,10}$. 
On the other hand, $\sharp (\partial_{\alpha} F_a) = \sharp (\partial_{\beta} F_a)=1$, and 
$\mathcal{F}_a$ has a property such that 
each component of  $\partial_{\alpha} F_a \cup \partial_{\beta} F_a$ 
has no $1$ prong. 
Hence by capping either $\partial_{\alpha} F_a$ or  $ \partial_{\beta} F_a$ with the disk, 
we get the pseudo-Anosov $\overline{\Phi}_a: \overline{F}_a \rightarrow \overline{F}_a$ 
on the $(8+1)$-punctured sphere  
with the same dilatation as $\lambda(a)$. 
Since $\overline{\Phi}_a$ fixes the boundary component 
($\partial_{\alpha} F_a$ or  $ \partial_{\beta} F_a$) of $F_a$, 
it defines the pseudo-Anosov on $D_8$. 
Lanneau and Thiffeault computed $\delta(D_8)$ in \cite{LT1}. 
We see that  $\lambda(a)$ is equal to $\delta(D_8)$,  
and hence  $\overline{\Phi}_a$ is a minimizer of $\delta(D_8)$. 
See Figure~\ref{fig_3_5_0}. 
\end{enumerate}
\end{ex}

\begin{center}
\begin{figure}
\includegraphics[width=5.5in]{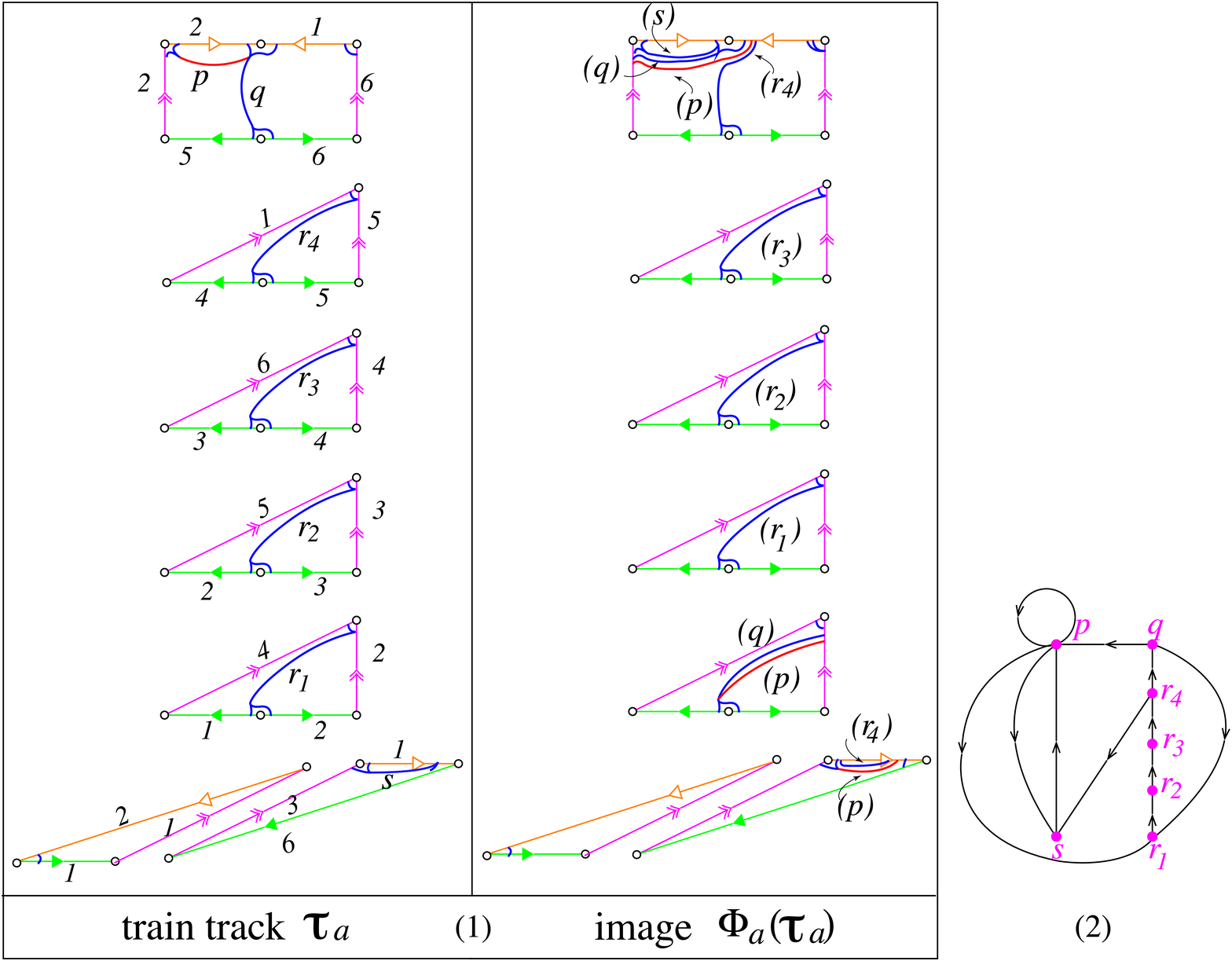}
\caption{(1) $\tau_a \subset F_a$ and $\Phi_a(\tau_a)$ (up to isotopy) 
(2) $\Gamma_a$ for $a=(1,4,1)_+$. }
\label{fig_6_2_1}
\end{figure}
\end{center}

\begin{center}
\begin{figure}
\includegraphics[width=5in]{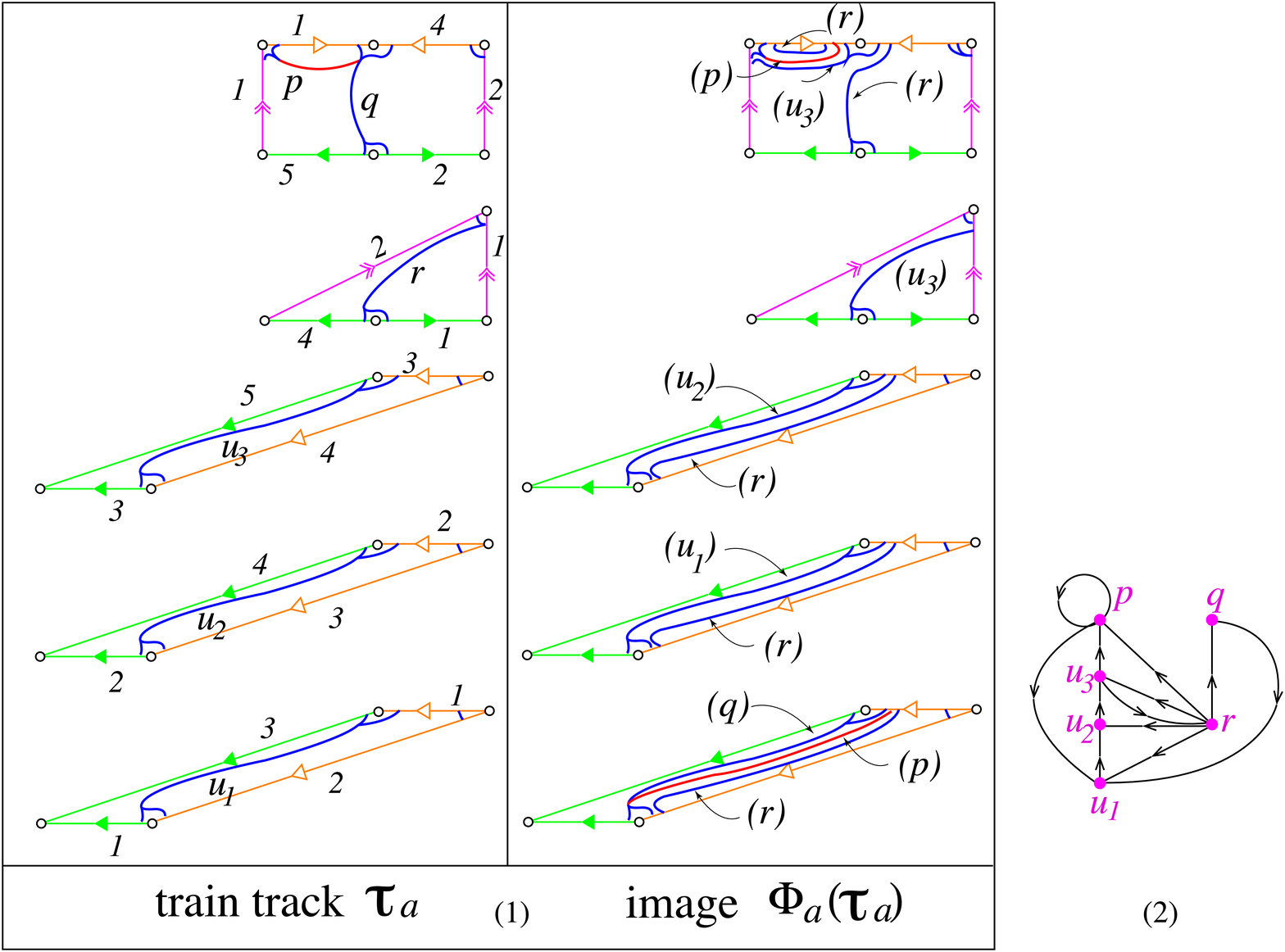}
\caption{(1) $\tau_a \subset F_a$ and $\Phi_a(\tau_a)$ (up to isotopy) 
(2) $\Gamma_a$ for $a= (3,1,1)_-$.}
\label{fig_1_2_-3}
\end{figure}
\end{center}

\begin{center}
\begin{figure}
\includegraphics[width=4.5in]{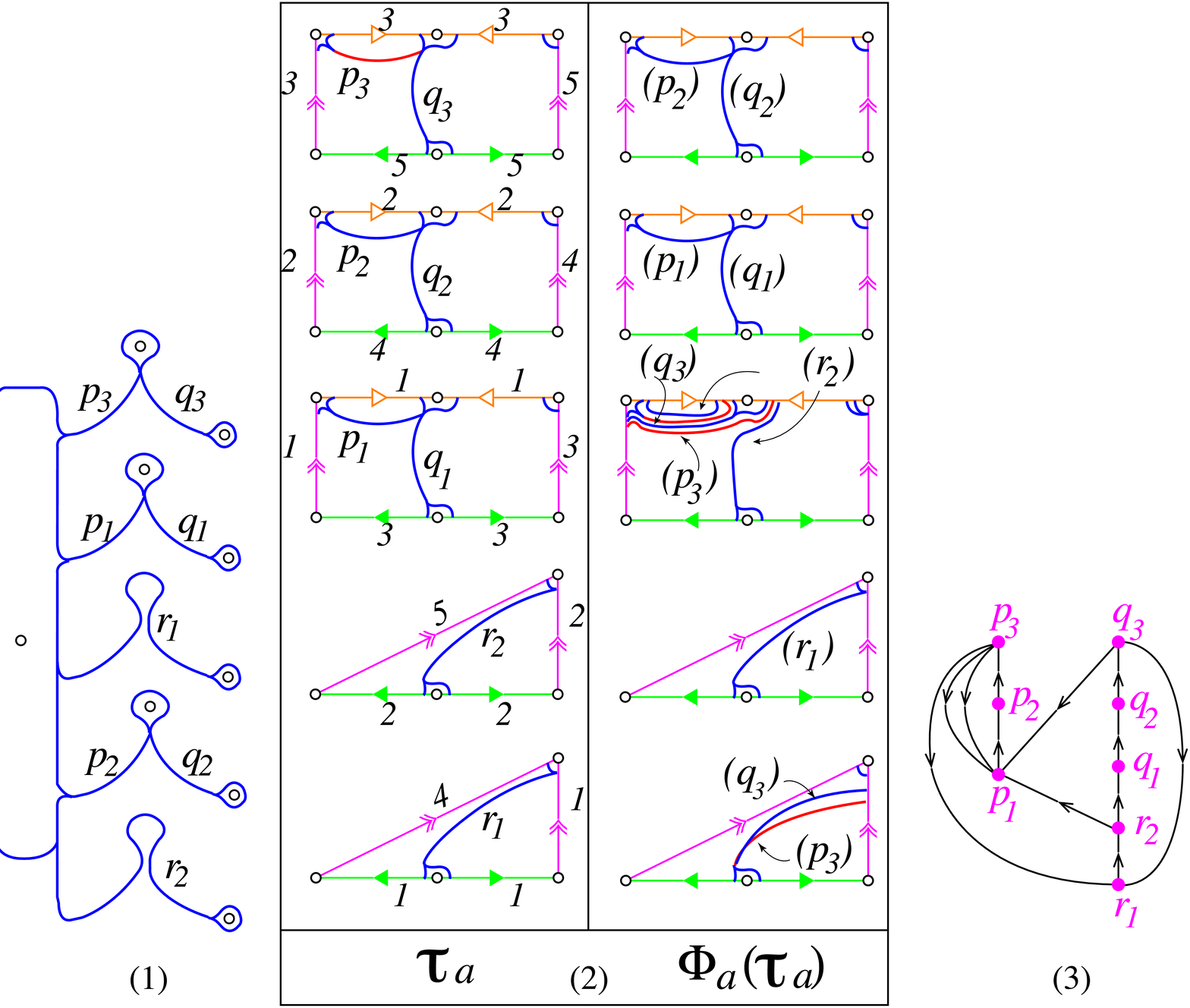}
\caption{(1) $\tau_a $ 
(2) Another description of $\tau_a \subset F_a$ and $\Phi_a(\tau_a)$ (up to isotopy) 
(3) $\Gamma_a$ for $a= (2,3)_0$.}
\label{fig_3_5_0}
\end{figure}
\end{center}

\subsection{Infinite sequences of fibered classes}
\label{subsection_Infinite}

We exhibit several infinite sequences of fibered classes of $N$, 
each of which corresponds to a subsequence of pseudo-Anosovs 
to give the upper bounds (U1), $\cdots$, (U6) except (U3) in Table~\ref{table_section}. 
Each sequence lies on the section either $S_{\beta}(r)$ for $r \in \{\tfrac{3}{-2}, -1, \tfrac{1}{-2}\}$ or 
$S_{\gamma}(r)$ for $r \in \{1, \infty\}$. 
A sequence of fibered classes in $S_{\beta}(\tfrac{3}{-2})$ (resp. $S_{\gamma}(1)$, $S_{\beta}(\tfrac{1}{-2})$) 
defines the sequence of fibered classes of the hyperbolic Dehn filling $N(\tfrac{3}{-2})$ 
(resp. $N(1)$, $N(\tfrac{1}{-2})$). 
We indicate where each sequence sits on the fibered face $\Delta$, 
see Figure~\ref{fig_candidatePL}. 
The purpose of the following examples is to construct the invariant train track $\tau_a$, the induced directed graph $\Gamma_a$ 
and curve complex $G_a$
associated to each fibered class $a$ in each infinite sequence concretely 
and to see a structure of the monodromy $\Phi_a$ of the fibration 
associated to $a$ via the curve complex $G_a$.  
We will see that the topological type of   $G_a$ is fixed 
and it is either $K_{1,2}^{**}$  or $K_{1,3}^{**}$ 
for each infinite sequence. 

Other sequences of fibered classes of $N$ with small dilatations can be found in \cite[Section~5]{Kin}.

\begin{center}
\begin{figure}
\includegraphics[width=5in]{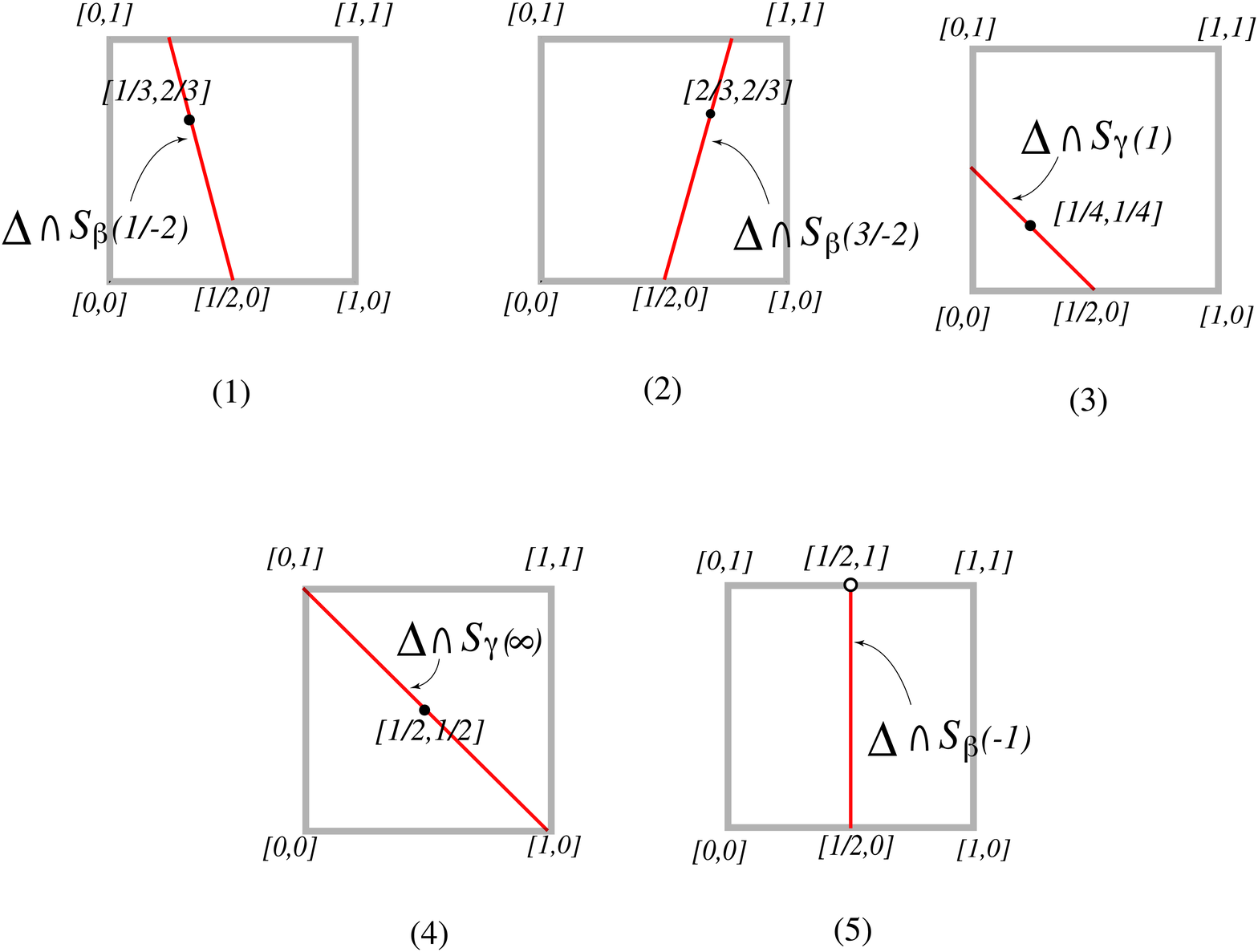}
\caption{(1) $\Delta \cap S_{\beta}(\tfrac{1}{-2})$. 
(2) $\Delta \cap S_{\beta}(\tfrac{3}{-2})$. 
(3) $\Delta \cap S_{\gamma}(1)$. 
(4) $\Delta \cap S_{\gamma}(\infty) = \Delta_0$. 
(5) $\Delta \cap S_{\beta}(-1)$.}
\label{fig_candidatePL}
\end{figure}
\end{center}

\begin{ex}[Figure~\ref{fig_ex_LT}]
\label{ex_LT}
Suppose that $g \equiv 2,4$ $\pmod{6}$. 
Take a sequence of primitive fibered classes 
$$a_g= (1, g+2, g-1)_+= (g, 2g+2,1) \in S_{\beta}(\tfrac{1}{-2}).$$
(The fibered class in Example~\ref{example_Ex}(1) is equal to $a_2$ in the present example.)  
Observe that 
the projections  $[a_g]$'s go to $ [\tfrac{1}{3}, \tfrac{2}{3}] \in int(\Delta)$ as $g$ goes to $\infty$, 
see Figure~\ref{fig_candidatePL}(1). 
By Lemma~\ref{lem_summary}(3)(7),  
it was shown in \cite{KT1} that 
the fiber $F_{a_g}$ has genus $g$, and $\mathcal{F}_{a_g}$ is orientable. 
The dilatation $\lambda_{(1, g+2, g-1)_+}$ 
is equal to the largest root of 
$$Q_{(1, g+2, g-1)_+}(t)= (t^{g+1}+1)(t^{2g}- t^{g+1}-t^g-t^{g-1}+1).$$
Since $\mathcal{F}_{a_g}$ is orientable, the monodromy 
$\Phi_{a_g}: F_{a_g} \rightarrow F_{a_g}$ of the fibration associated to $a_g$ extends canonically to the pseudo-Anosov homeomorphism 
$\widehat{\Phi}_{g}$ on the closed surface $\varSigma_g$ of genus $g$ with the dilatation $\lambda_{(1, g+2, g-1)_+}$. 
We now introduce 
the {\it Lanneau-Thiffeault polynomial} 
$$f_{(a,b)}(t)= t^{2a} - t^{a+b} - t^a - t^{a-b}+1.$$ 
Then $f_{(g,1)}(t)$ is a factor of $Q_{(1, g+2, g-1)_+}(t)$. 
If we set   $\lambda_{(g,1)}$  to be  the largest root of $f_{(g,1)}(t)$, 
then we have $\lambda_{(g,1)}= \lambda_{(1, g+2, g-1)_+} = \lambda(\widehat{\Phi}_{g})$. 
It is known that 
$\delta_2^+= \lambda_{(2,1)}$ \cite{Zhirov}, 
$\delta_4^+=\lambda_{(4,1)}$ \cite{LT} and 
$\delta_8^+= \lambda_{(8,1)}$ \cite{LT,Hironaka}. 
Lanneau and Thiffeault asked in \cite{LT} whether  the equality $\delta_g^+= \lambda_{(g,1)}$ holds  for even $g$. 
The existence of pseudo-Anosovs $\widehat{\Phi}_{g}$ for $g \equiv 2,4$ $\pmod{6}$ in the present example were 
discovered by Hironaka in  \cite{Hironaka}, see also Remark~\ref{rem_hironaka-san}. 
 The sequence of pseudo-Anosovs $\{\widehat{\Phi}_g\}$ can be used for a subsequence 
 to prove both upper bounds (U1) and (U2) in Table~\ref{table_section}. 
 In fact we have 
 $$\lim_{g \to \infty} g \log \lambda(\widehat{\Phi}_g)= \log (\tfrac{3+ \sqrt{5}}{2}).$$ 
\end{ex}

\begin{ex}[Figure~\ref{fig_ex_ori79}]
\label{ex_ori79}
An explicit value of  $\delta_7^+$ is known. 
The lower bound of $\delta_7^+$ was given in \cite{LT} and 
an example which realizes this lower bound was found in \cite{AD} and \cite{KT1} independently. 
However an explicit construction of a minimizer of $\delta_7^+$ was not given in  \cite{AD,KT1}.  
Here we construct such a minimizer.  
Suppose that $g  \equiv 7,9 \pmod {10}$. 
Following \cite{KT1}, 
we take a sequence of primitive fibered classes 
$$a_g= (g+6, 2, g)_+=(2g+6,2g+8,g+6) \in S_{\beta}(\tfrac{3}{-2}).$$
Then $F_{a_g}$ has genus $g$ (\cite{KT1}) and 
 the projections  $[a_g]$'s  go to $ [\tfrac{2}{3}, \tfrac{2}{3}]  \in int(\Delta)$ 
as $g$ goes to $\infty$, see Figure~\ref{fig_candidatePL}(2). 
We have 
$$Q_{ (g+6, 2, g)_+}(t)= (t^{g+4}+1) (t^{2g+4} -t^{g+4}-t^{g+2}-t^{g}+1),$$
and the dilatation $\lambda_{ (g+6, 2, g)_+}$ is equal to the largest root of this polynomial. 
In particular $\lambda_{ (g+6, 2, g)_+}$ is equal to the largest root of the second factor of $Q_{ (g+6, 2, g)_+}(t)$. 
Lemma~\ref{lem_summary}(7) ensures that the monodromy 
$\Phi_{a_g}: F_{a_g} \rightarrow F_{a_g}$ of the fibration  associated to $a_g$ 
extends to the pseudo-Anosov $\widehat{\Phi}_g: \varSigma_g \rightarrow \varSigma_g$  
with orientable invariant foliations. 
It is known that $\delta_7^+$ is equal to the largest root of the second factor of $Q_{(13, 2, 7)_+}(t)$ as above. 
Hence $\delta_7^+ = \lambda_{(13, 2, 7)_+}$, and 
 the extension $\widehat{\Phi}_7$ of ${\Phi}_{(13, 2, 7)_+}$ is a minimizer of $\delta_7^+$. 
The existence of  the sequence of pseudo-Anosovs $\{\widehat{\Phi}_g\}$ was discovered in \cite{AD,KT1}. 
We have  
 $$\lim_{g \to \infty} g \log \lambda(\widehat{\Phi}_g)= \log (\tfrac{3+ \sqrt{5}}{2}),$$
 see the upper bounds (U1) and (U2). 
\end{ex}

\begin{ex}[Figure~\ref{fig_ex_ori15}]
\label{ex_ori15}
Suppose that $g  \equiv 1,5 \pmod {10}$. 
Following \cite{KT1}, 
we take a sequence of primitive fibered classes 
$$a_g= (g+10,4,g-2)_+= (2g+8,2g+12,g+10)  \in S_{\beta}(\tfrac{3}{-2}).$$
Then the fiber $F_g$ has genus $g$ (\cite{KT1}), and  
the dilatation $\lambda_{(g+10,4,g-2)_+}$ is equal to the largest root of 
$$Q_{(g+10,4,g-2)_+}(t)=  (t^{g+6}+1)(t^{2g+4}-t^{g+6}-t^{g+2}-t^{g-2}+1).$$
Observe that the projections $[a_g]$'s go to $  [\tfrac{2}{3}, \tfrac{2}{3}] \in int(\Delta)$ as $g$ goes to $\infty$, 
see Figure~\ref{fig_candidatePL}(2). 
We see that 
the monodromy $\Phi_{a_g}: F_{a_g} \rightarrow F_{a_g}$ of the fibration associated to $a_g$ 
extends to the pseudo-Anosov 
$\widehat{\Phi}_g$ on $\varSigma_g$ with orientable invariant foliations. 
The dilatation of $\widehat{\Phi}_g$ is same as the dilatation of $\Phi_{a_g}$, that is $\lambda_{(g+10,4,g-2)_+}$. 
The equality $\delta_5^+ = \lambda_{(15,4,3)_+}$ holds, see \cite{LT}. 
 The sequence of pseudo-Anosovs $\{\widehat{\Phi}_g\}$ satisfies 
 $$\lim_{g \to \infty} g \log \lambda(\widehat{\Phi}_g)= \log (\tfrac{3+ \sqrt{5}}{2}),$$
 see the upper bounds (U1) and (U2). 
\end{ex}

\begin{ex}[Figure~\ref{fig_ex_whitehead}]
\label{ex_whitehead}
Following \cite{KKT2}, 
we take a sequence of primitive fibered classes 
$$a_n= (2n-1, 1, n-1)_-= (n-1,n,-2n+1) \in S_{\gamma}(1) .$$
The dilatation $\lambda_{(2n-1, 1, n-1)_-}$ is equal to the largest root of 
$$Q_{(2n-1, 1, n-1)_-}(t)= t^{4n-2} - t^{3n-1}- t^{3n-2}-t^n - t^{n-1}+1.$$ 
The projections of $a_n$'s go to $[\tfrac{1}{4}, \tfrac{1}{4}] \in int(\Delta)$ as $n$ goes to $\infty$, 
see  Figure~\ref{fig_candidatePL}(3). 
We find that $a_n$ defines a fibered class $\overline{a_n} \in H_2(N_{\gamma}(1), \partial N_{\gamma}(1))$, 
see  the beginning of Section~\ref{section_catalogue}. 
By using Lemma~\ref{lem_summary}(6), 
we see that the monodromy of the fibration on $N(1)$ associated to $\overline{a_n} $ 
has the dilatation $\lambda_{(2n-1, 1, n-1)_-}$. 
Observe that the topological type of the fiber 
(i.e, minimal representative) of $\overline{a_n}$ is homeomorphic to $\varSigma_{1, 2n-1}$. 
The sequence of the monodromies $\Phi_{\overline{a_n}}$ of fibrations associated to $\overline{a_n}$'s 
on the Whitehead link exterior $N_{\gamma}(1)$ 
can be used for a subsequence to prove (U5). 
In fact, we have 
$$\lim_{n \to \infty} (2n-1) \log \lambda(\Phi_{\overline{a_n}}) = 2 \log \delta(D_4).$$
\end{ex}

\begin{ex}[Figure~\ref{fig_ex_braid1}] 
\label{ex_braid1} 
For $n \ge 3$, 
consider a sequence of primitive fibered classes 
$$a_n= (1,n-1)_0= (n-1,n,0) \in S_{\gamma}(\infty),$$ 
which is studied in \cite{KT}. 
We see that $F_{a_n}$ is homeomorphic to $\varSigma_{0,2n+1}$. 
The projection of $a_n$ goes to $[\tfrac{1}{2}, \tfrac{1}{2}] \in int(\Delta)$ as $n$ goes to $\infty$, 
see Figure~\ref{fig_candidatePL}(4). 
The dilatation $\lambda_{(1,n-1)_0}$ equals the largest root of 
$$Q_{(1,n-1)_0}(t)=    t^{2n-1}- 2(t^{n-1}+ t^n)+1.$$ 
By using the same arguments as in Example~\ref{example_Ex}(3), 
we see that 
the monodromy $\Phi_{a_n}: \varSigma_{0,2n+1} \rightarrow \varSigma_{0,2n+1}$ of the fibration associated to $a_n$ 
defines the pseudo-Anosov $\overline{\Phi}_{a_n}: D_{2n-1} \rightarrow D_{2n-1}$ 
 with the same dilatation $\lambda_{(1,n-1)_0}$. 
Such a pseudo-Anosov homeomorphism $\overline{\Phi}_{a_n}: D_{2n-1} \rightarrow D_{2n-1}$ is studied in \cite{HK}. 
It is known that $\delta(D_5)= \lambda_{(1,2)_0}$, see \cite{HS} and 
$\delta(D_7) = \lambda_{(1,3)_0}$, see \cite{LT1}. 
The sequence of pseudo-Anosovs $\{\Phi_{a_n}: \varSigma_{0, 2n-1} \rightarrow \varSigma_{0, 2n-1}\}$ 
can be used for a subsequence to prove (U4): 
$$\lim_{n \to \infty} (2n-1) \log \lambda(\Phi_{a_n}) = 2 \log (2+ \sqrt{3}).$$
\end{ex}

\begin{ex}[Figure~\ref{fig_ex_braid2}]
\label{ex_braid2} 
For $n \ge 2$, 
let us consider a sequence of fibered classes 
$$a_n= (2,2n-1)_0= (2n-1,2n+1,0) \in S_{\gamma}(\infty),$$
which is studied in \cite{KT}. 
(The fibered class in Example~\ref{example_Ex}(3) is equal to $a_2$ in the present example.)  
We see that $F_{a_n}$ is homeomorphic to $\varSigma_{0,4n+2}$. 
The projection of $a_n$ goes to $[\tfrac{1}{2}, \tfrac{1}{2}]  \in int(\Delta)$ as $n$ goes to $\infty$. 
The dilatation $\lambda_{(2,2n-1)_0}$ equals the largest root of 
$$Q_{(2,2n-1)_0}(t)=   t^{4n}-2(t^{2n-1}+ t^{2n+1}) +1. $$
By using the same arguments as in Example~\ref{example_Ex}(3), 
we see that 
the monodromy $\Phi_{a_n}: \varSigma_{0,4n+2} \rightarrow \varSigma_{0,4n+2}$ of the fibration associated to $a_n$ 
defines the pseudo-Anosov $\overline{\Phi}_{a_n}: D_{4n} \rightarrow D_{4n}$ with the  dilatation $\lambda_{(2,2n-1)_0}$. 
As we have seen in Example~\ref{example_Ex}(3) that 
$\delta(D_8)= \lambda_{(2,3)_0}$ holds. 
Such a pseudo-Anosov homeomorphism $\overline{\Phi}_{a_n}: D_{4n} \rightarrow D_{4n}$  is also studied in \cite{Venzke}. 
The sequence of pseudo-Anosovs $\{\Phi_{a_n}: \varSigma_{0,4n+2} \rightarrow \varSigma_{0,4n+2}\}$ 
can be used for a subsequence to prove (U4): 
$$\lim_{n \to \infty} (4n+2) \log \lambda(\Phi_{a_n}) = 2 \log (2+ \sqrt{3}).$$
\end{ex}

\begin{ex}[Figure~\ref{fig_ex_Tsai}]
\label{ex_Tsai}
Throughout  this example, 
we fix $g \ge 0$. 
The following sequence of fibered classes is used for the proof of Theorem~\ref{thm_boundary}, see \cite{KT2}: 
$$a_{(g,p)}= (p-g, p-g, 2g+1)_+= (p+g+1, 2p+1, p-g) \in S_{\beta}(-1).$$ 
The projections $[a_{(g,p)}]$'s go to $[\tfrac{1}{2},1] \in  \partial \Delta$  as $p$ goes to $\infty$, 
see Figure~\ref{fig_candidatePL}(5). 
It is not hard to see that  $a_{(g,p)}$ is primitive if and only if $2g+1$ and $p+g+1$ are relatively prime. 
If $a_{(g,p)}$ is primitive, then 
the fiber $F_{a_{(g,p)}}$ is homeomorphic to $\varSigma_{g, 2p+4}$, and we have 
$$\sharp(\partial_{\beta} F_{(g,p)})= 2p+1\ \mbox{and}\  
\sharp(\partial_{\alpha} F_{(g,p)})+ \sharp(\partial_{\gamma} F_{(g,p)})= 3.$$
There exists a sequence of primitive fibered classes $\{a_{(g, p_i)}\}_{i=0}^{\infty}$ 
with $p_i \to \infty$ when $i \to \infty$, 
where $p_i$ depends on $g$. 
(For example, take $p_i= (g+1)+ i(2g+1)$.) 
This means that   $F_{a_{(g, p_i)}}$ has genus $g$ and  the number of the boundary components of $F_{a_{(g, p_i)}}$ 
(which is equal to $2p_i+4$) 
goes to $\infty$ as $i$ does. 
As we have seen above, 
the projection of such a class $a_{(g,p_i)}$ goes to the same point $[\tfrac{1}{2},1] \in \partial \Delta$ 
(which does not depend on $g$) as $p_i$ goes to $\infty$. 
\end{ex}

\begin{center}
\begin{figure}
\includegraphics[width=6in]{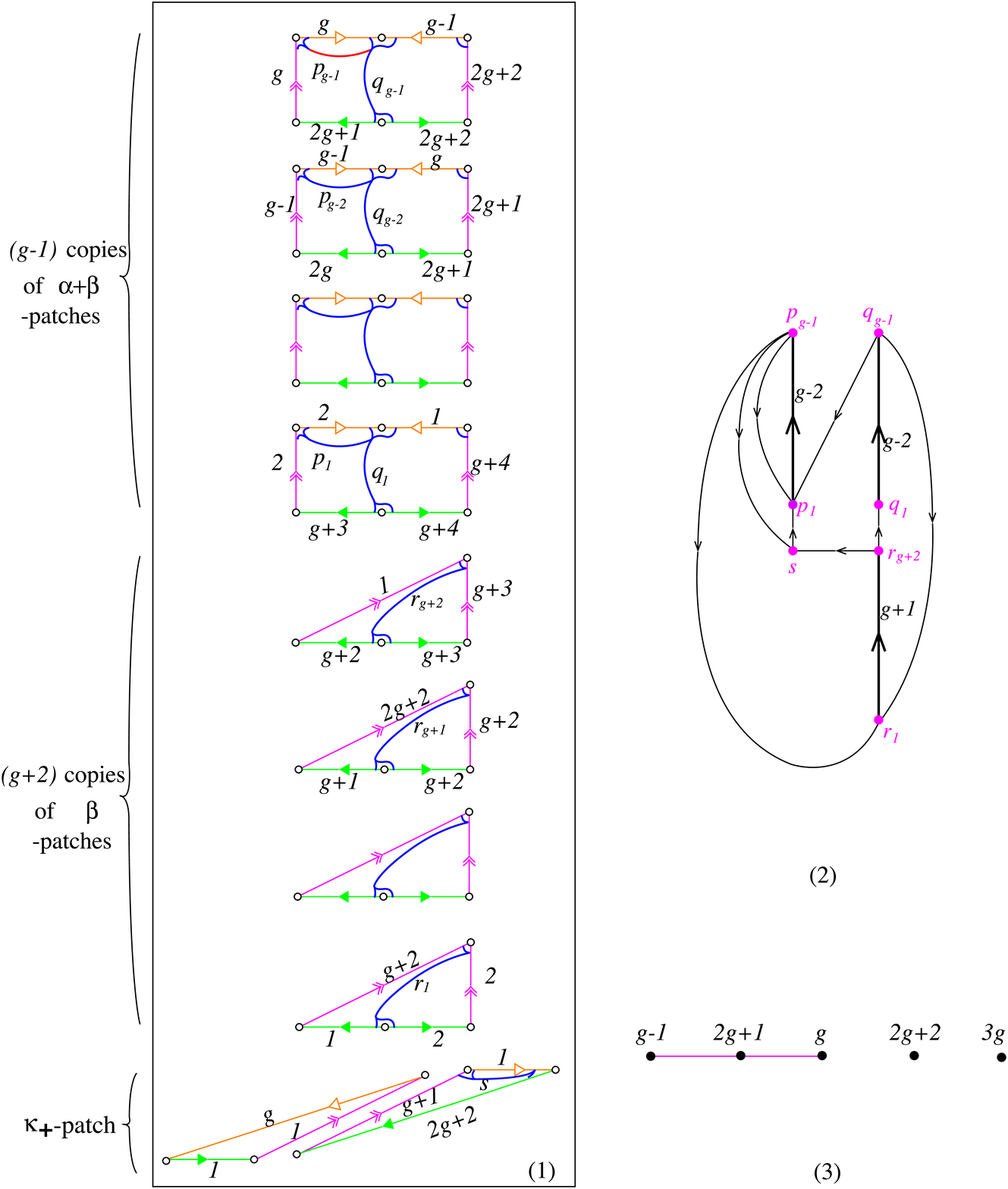}
\caption{Example~\ref{ex_LT}. 
(1) $\tau_{a_g} \subset F_{a_g}$ 
(2) $\Gamma_{a_g}$ 
(3) $G_{a_g}$ for $a_g= (1, g+2, g-1)_+$.}
\label{fig_ex_LT}
\end{figure}
\end{center}

\begin{center}
\begin{figure}
\includegraphics[width=6in]{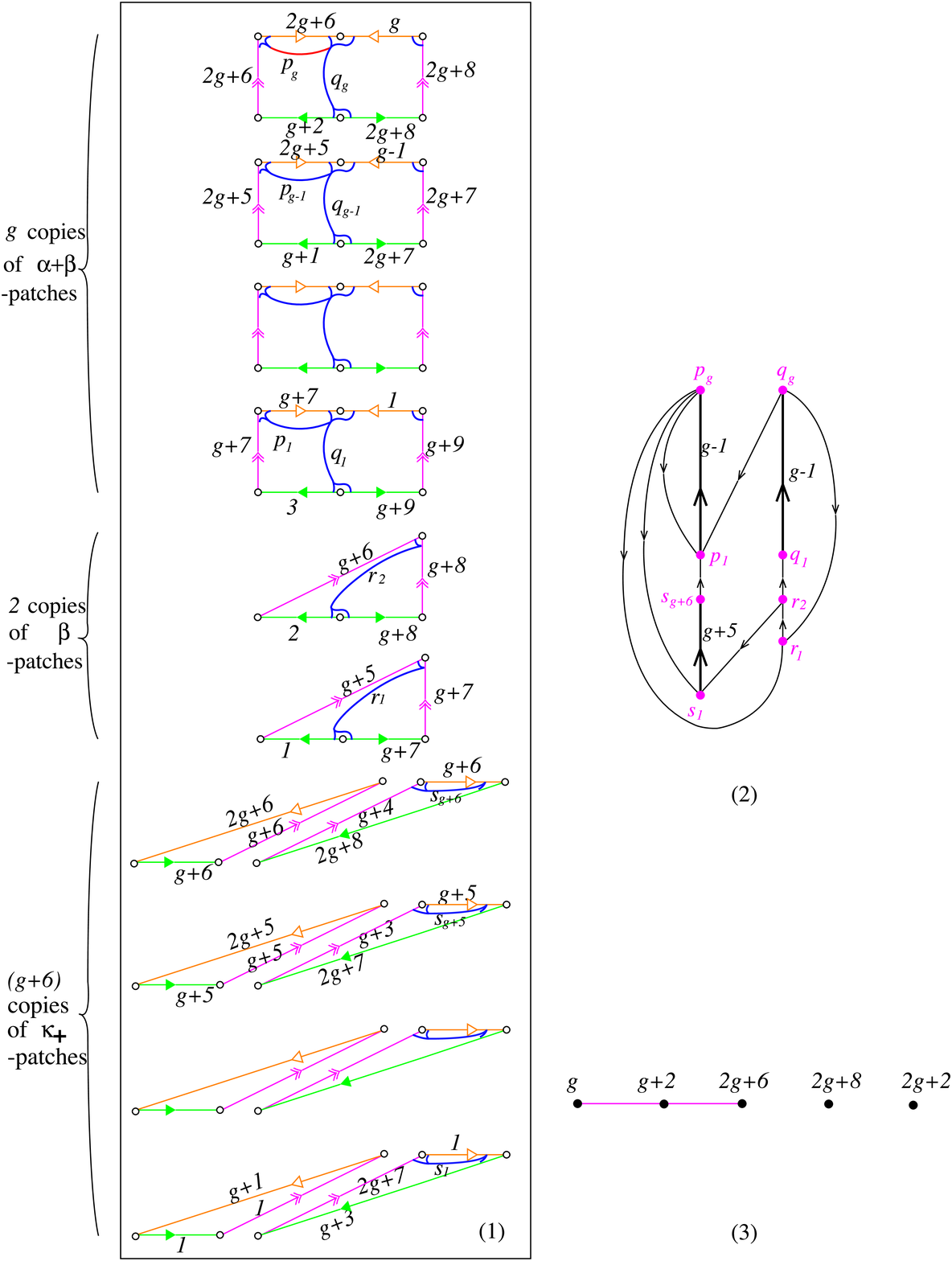}
\caption{Example~\ref{ex_ori79}. 
(1) $\tau_{a_g} \subset F_{a_g}$ 
(2) $\Gamma_{a_g}$ 
(3) $G_{a_g}$ for $a_g= (g+6, 2, g)_+$.}
\label{fig_ex_ori79}
\end{figure}
\end{center}

\begin{center}
\begin{figure}
\includegraphics[width=5.7in]{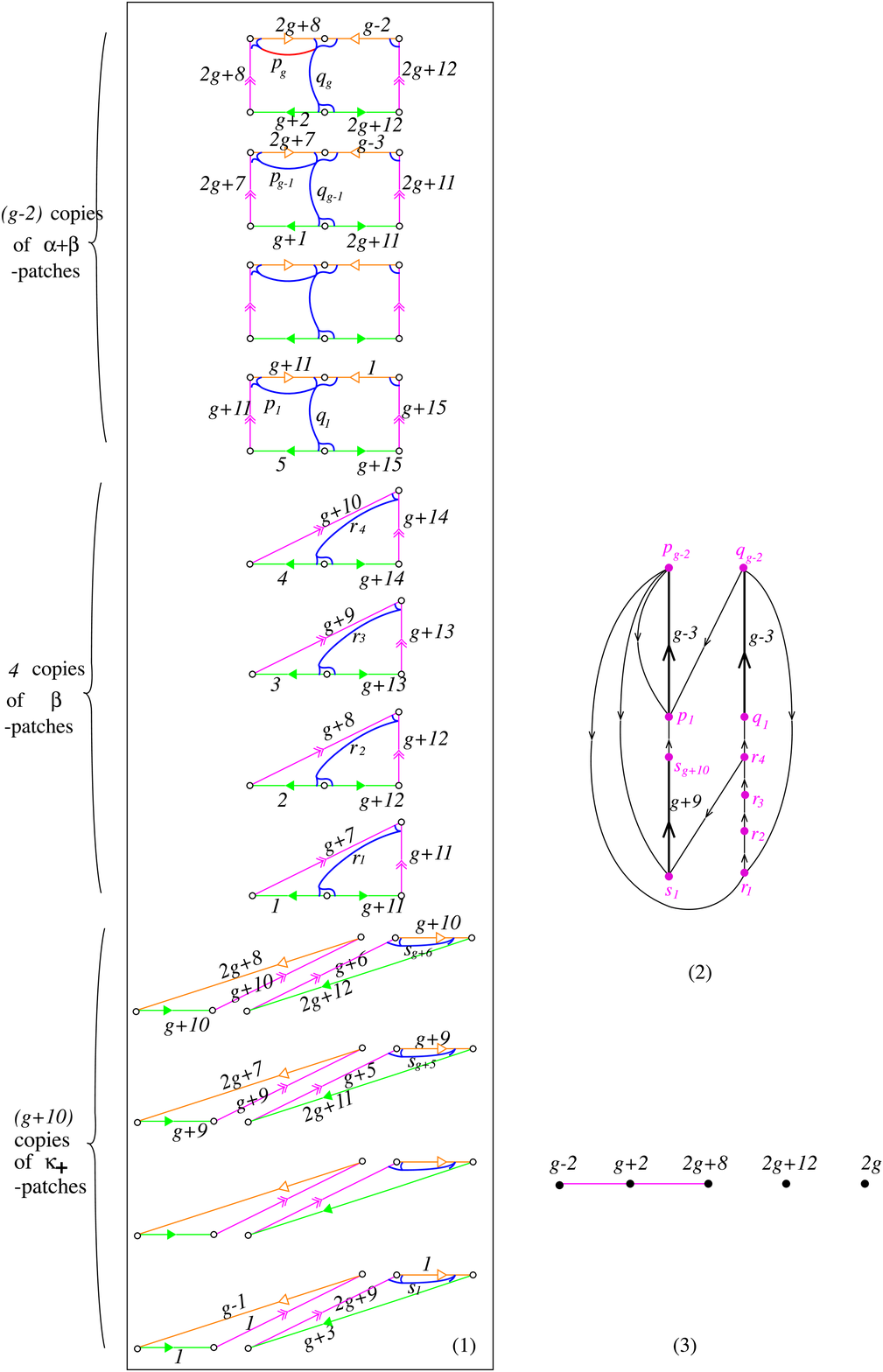}
\caption{Example~\ref{ex_ori15}. 
(1) $\tau_{a_g} \subset F_{a_g}$ 
(2) $\Gamma_{a_g}$ 
(3) $G_{a_g}$ for $a_g=(g+10,4,g-2)_+$.}
\label{fig_ex_ori15}
\end{figure}
\end{center}

\begin{center}
\begin{figure}
\includegraphics[width=6in]{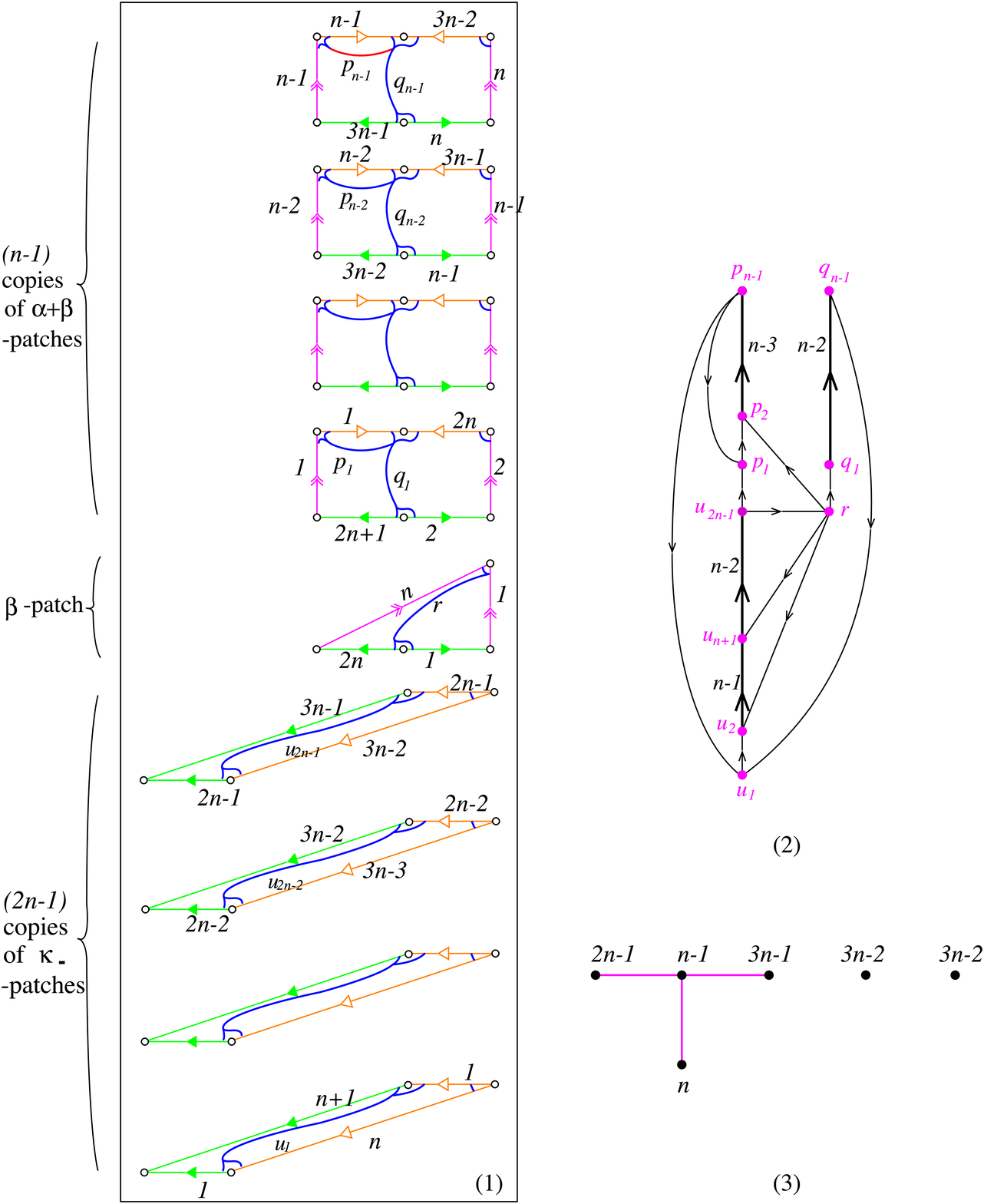}
\caption{Example~\ref{ex_whitehead}. 
(1) $\tau_{a_n} \subset F_{a_n}$ 
(2) $\Gamma_{a_n}$ 
(3) $G_{a_n}$ for $a_n= (2n-1, 1, n-1)_-$.}
\label{fig_ex_whitehead}
\end{figure}
\end{center}

\begin{center}
\begin{figure}
\includegraphics[width=3.5in]{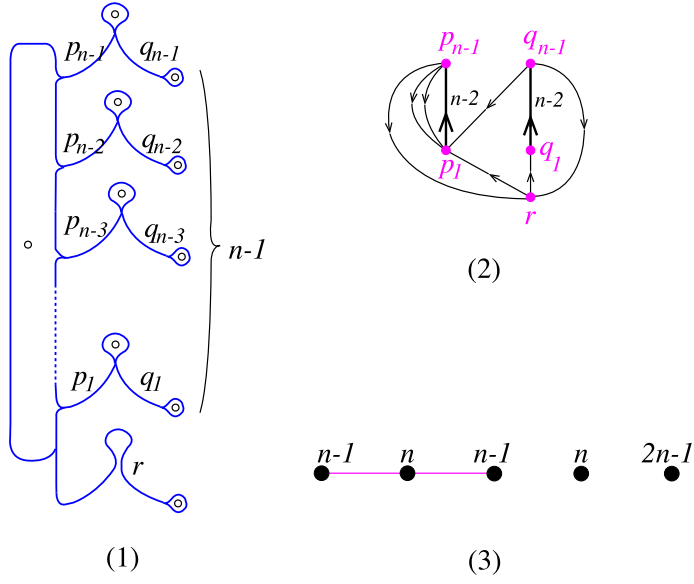}
\caption{Example~\ref{ex_braid1}. 
(1) $\tau_{a_n}$ (circles indicate components of $\partial F_{a_n}$) 
(2) $\Gamma_{a_n}$ 
(3) $G_{a_n}$ for $a_n= (1,n-1)_0$.}
\label{fig_ex_braid1}
\end{figure}
\end{center}

\begin{center}
\begin{figure}
\includegraphics[width=3.5in]{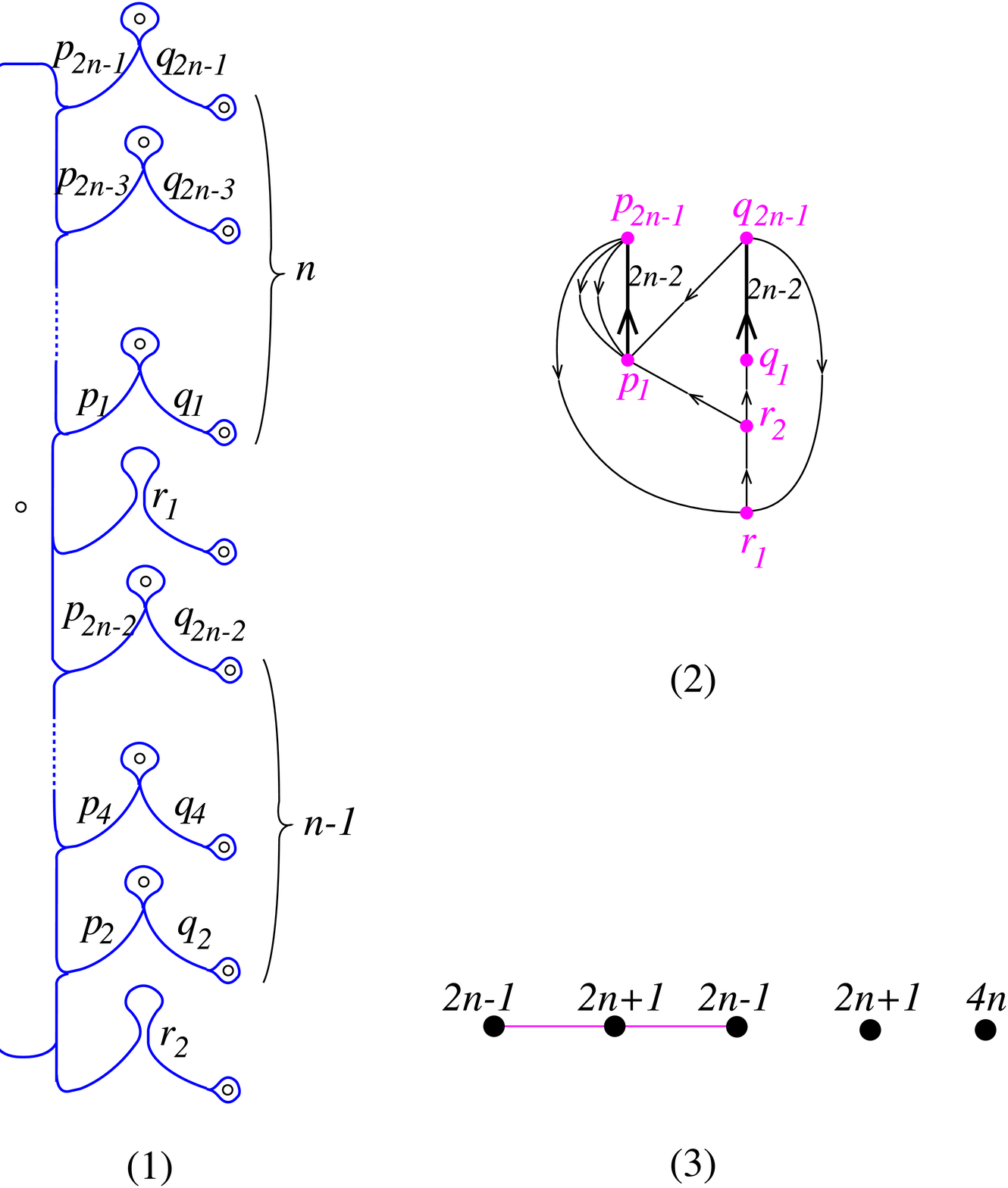}
\caption{Example~\ref{ex_braid2}. 
(1) $\tau_{a_n}$ 
(2) $\Gamma_{a_n}$ 
(3) $G_{a_n}$ for $a_n= (2, 2n-1)_0$.}
\label{fig_ex_braid2}
\end{figure}
\end{center}

\begin{center}
\begin{figure}
\includegraphics[width=5in]{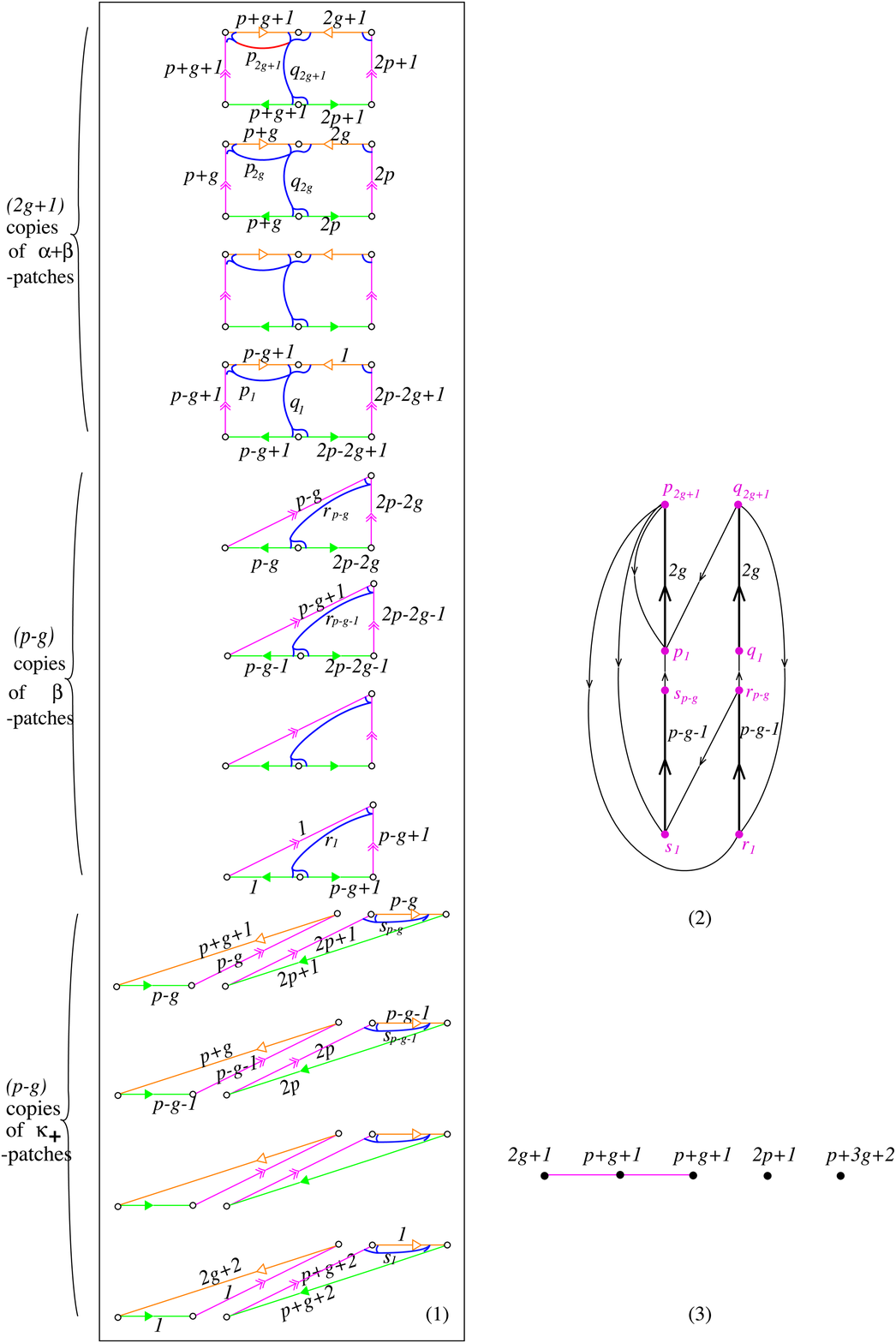}
\caption{Example~\ref{ex_Tsai}. 
(1) $\tau_{a_{(g,p)}} \subset F_{a_{(g,p)}}$ 
(2) $\Gamma_{a_{(g,p)}}$ 
(3) $G_{a_{(g,p)}}$ for $a_{(g,p)}= (p-g, p-g, 2g+1)_+$.}
\label{fig_ex_Tsai}
\end{figure}
\end{center}


\newpage

\begin{thebibliography}{99}

\bibitem{AD} 
J.~W.~Aaber and N.~M.~Dunfield, 
{\it Closed surface bundles of least volume}, 
Algebraic and Geometric Topology 10 (2010), 2315-2342. 









%
%


%
%
%

\bibitem{BH}
M~Bestvina, M~Handel, 
{\it Train--tracks for surface homeomorphisms}, 
Topology 34 (1994) 1909-140. 


\bibitem{Birman} 
J.~Birman, 
{\it PA mapping classes with minimum dilatation and Lanneau-Thiffeault polynomials}, 
preprint. 




%
\bibitem{CH} 
J.~H.~Cho and J.~Y.~Ham,   
{\it The minimal dilatation of a genus-two surface}, 
Experimental Mathematics 17 (2008), 257-267. 





\bibitem{FLP}
A.~Fathi, F.~Laudenbach and V.~Poenaru, 
Travaux de Thurston sur les surfaces,
Ast\'erisque, 66-67, 
Soci\'et\'e Math\'ematique de France, Paris (1979). 


\bibitem{Fried}
D.~Fried, 
{\it Fibrations over $S^1$ with pseudo-Anosov monodromy}, 
Expos\'e~14 in `Travaux de Thurston sur les surfaces' by A.~Fathi, F.~Laudenbach and V.~Poenaru,
Ast\'erisque, 66-67, 
Soci\'et\'e Math\'ematique de France, Paris (1979), 251-266.  




\bibitem{Fried1}
D.~Fried, 
{\it Flow equivalence, hyperbolic systems and a new zeta function for flows}, 
Commentarii Mathematici Helvetici 57 (1982), 237-259. 


\bibitem{Gantmacher}
F.~R.~Gantmacher, 
Matrix Theory, Vol.~2, 
 American Mathematical Society (2000). 


%
%




\bibitem{HS}
J.~Y.~Ham and W.~T.~Song, 
{\it The minimum dilatation of pseudo-{A}nosov $5$-braids}, 
Experimental Mathematics 16 (2007), 167-179. 

\bibitem{Hironaka} 
E.~Hironaka, 
{\it Small dilatation  mapping classes coming from the simplest hyperbolic braid}, 
Algebraic and Geometric Topology 10 (2010), 2041-2060. 


\bibitem{Hironaka1} 
E.~Hironaka, 
{\it Mapping classes associated to mixed-sign Coxeter graphs}, 
preprint. 

\bibitem{Hironaka2} 
E.~Hironaka, 
{\it Quotient families of mapping classes}, 
preprint. 

\bibitem{Hironaka3} 
E.~Hironaka, 
{\it Small dilatation pseudo-Anosov mapping classes and short circuits on train track automata}, 
Mittag-Leffler Institute Preprint Collection, Workshop on Growth and Mahler Measures in Geometry and Topology (2013), 
17-41. 




\bibitem{HK}
E.~Hironaka and E.~Kin, 
{\it A family of pseudo-Anosov braids with small dilatation}, 
Algebraic and Geometric Topology 6 (2006), 699-738. 

\bibitem{Ivanov}
N.~V.~Ivanov,
{\it Stretching factors of pseudo-Anosov homeomorphisms},
Journal of Soviet Mathematics, 52 (1990), 2819--2822, which is translated from
Zap. Nauchu. Sem. Leningrad. Otdel. Mat. Inst. Steklov.
(LOMI), 167 (1988), 111-116. 


\bibitem{Kin} 
E.~Kin, 
{\it Notes on pseudo-Anosovs with small dilatations coming from the magic 3-manifold}, 
Representation spaces, twisted topological invariants and geometric structures of 3-manifolds, 
RIMS Kokyuroku 1836 (2013), 45-64. 
available at 
\verb#http://www.kurims.kyoto-u.ac.jp/~kyodo/kokyuroku/contents/pdf/1836-05.pdf#



\bibitem{KKT2}
E.~Kin, S.~Kojima and M.~Takasawa, 
{\it Minimal dilatations of pseudo-Anosovs generated by the magic $3$-manifold and their asymptotic behavior}, 
Algebraic and Geometric Topology 13 (2013), 3537-3602. 



\bibitem{KT} 
E.~Kin and M.~Takasawa, 
{\it Pseudo-Anosov braids with small entropy and the magic $3$-manifold}, 
Communications in Analysis and Geometry 19 Number 4 (2011), 1-54. 

\bibitem{KT1} 
E.~Kin and M.~Takasawa, 
{\it Pseudo-Anosovs on closed surfaces having small entropy and the Whitehead sister link exterior}, 
Journal of the Mathematical Society of Japan 65, Number 2 (2013), 411-446. 


\bibitem{KT2} 
E.~Kin and M.~Takasawa, 
{\it The boundary of a fibered face of the magic $3$-manifold and 
the asymptotic behavior of the minimal pseudo-Anosovs dilatations.}
arXiv:1205.2956


\bibitem{KLS}
K.~H.~ Ko, J.~Los and W.~T.~ Song, 
{\it Entropies of braids}, 
Journal of Knot Theory and its Ramifications 11 (2002), 647-666. 



\bibitem{LT} 
E.~Lanneau and J.~L.~ Thiffeault, 
{\it On the minimum dilatation of pseudo-Anosov homeomorphisms on surfaces of small genus}, 
Annales de l'Institut Fourier 61 (2011), 105-144. 

\bibitem{LT1} 
E.~Lanneau and J.~L.~ Thiffeault, 
{\it On the minimum dilatation of braids on the punctured disc}, 
Geometriae Dedicata 152 (2011), 165-182. 

\bibitem{LO}
D.~Long and U.~Oertel, 
{\it Hyperbolic surface bundles over the circle}, 
Progress in knot theory and related topics, Travaux en Course 56,  Hermann, Paris (1997), 121-142. 




\bibitem{Los} 
J.~Los, 
{\it Infinite sequence of fixed-point free pseudo-Anosov homeomorphisms}, 
Ergodic Theory and Dynamical Systems 30 (2010), 1739-1755. 

 



\bibitem{MP} 
B.~Martelli and C.~Petronio, 
{\it Dehn filling of the ``magic" $3$-manifold}, 
Communications in Analysis and Geometry 14 (2006), 969-1026.

\bibitem{Matsumoto}
S.~Matsumoto, 
{\it Topological entropy and Thurston's norm of atoroidal surface bundles over the circle}, 
Journal of the Faculty of Science, University of Tokyo, Section IA. Mathematics 34 (1987), 763-778. 




\bibitem{McMullen}
C.~McMullen, 
{\it Polynomial invariants for fibered $3$-manifolds and Teichm\"{u}ler geodesic for foliations}, 
Annales Scientifiques de l'\'{E}cole Normale Sup\'{e}rieure. Quatri\`{e}me S\'{e}rie  33 (2000), 519-560. 

\bibitem{McMullen2} 
C.~McMullen, 
{\it Entropy and the clique polynomial}, 
preprint. 


\bibitem{Minakawa}
H.~Minakawa, 
{\it Examples of pseudo-Anosov homeomorphisms with small dilatations}, 
The University of Tokyo. Journal of Mathematical Sciences 13 (2006), 95-111. 


\bibitem{Oertel} 
U.~Oertel, 
{\it Homology branched surfaces: Thurston's norm on $H_2(M^3)$}, 
LMS Lecture Note Series 112, Low-dimensional Topology and Kleinian Groups, Editor D.~B.~A.~ Epstein (1986), 
253-272.



\bibitem{Oertel2} 
U.~Oertel, 
{\it Affine laminations and their stretch factores}, 
Pacific Journal of Mathematics 182 (2) (1997), 303-328. 



\bibitem{PaPe}
A.~Papadopoulos and R.~ Penner, 
{\it A characterization of pseudo-Anosov foliations}, 
Pacific Journal of Mathematics 130 (2) (1987), 359-377. 



\bibitem{Penner}
R.~C.~Penner, 
{\it Bounds on least dilatations}, 
Proceedings of the American Mathematical Society 113 (1991), 443-450. 


\bibitem{Thurston1}
W.~Thurston, 
{\it A norm of the homology of $3$-manifolds}, 
Memoirs of the American Mathematical Society 339 (1986), 99-130. 


\bibitem{Thurston3} 
W.~Thurston, 
{\it Hyperbolic structures on 3-manifolds II: Surface groups and 
3-manifolds which fiber over the circle}, preprint, 
arXiv:math/9801045

\bibitem{Tsai} 
C.~Y.~Tsai, 
{\it The asymptotic behavior of least pseudo-Anosov dilatations}, 
Geometry and Topology 13 (2009), 2253-2278. 




\bibitem{Venzke}
R.~Venzke, 
{\it Braid forcing, hyperbolic geometry, and pseudo-Anosov sequences of low entropy}, 
PhD thesis, California Institute of Technology (2008). 


\bibitem{Zhirov}
A.~Y.~Zhirov, 
{\it On the minimum dilation of pseudo-Anosov diffeomorphisms on a double torus}, 
Russian Mathematical Surveys 50 (1995), 223-224. 

\end{thebibliography}
\end{document}